\documentclass[aihp]{imsart}

\RequirePackage{amsthm,amsmath,amsfonts,amssymb}
\RequirePackage[numbers]{natbib}
\RequirePackage{graphicx}
\RequirePackage{mathrsfs}

\startlocaldefs
\newtheorem{theorem}{Theorem}

\newtheorem{lemma}{Lemma}
\newtheorem{definition}{Definition}

\newtheorem{proposition}{Proposition}
\endlocaldefs

\begin{document}

\begin{frontmatter}

\title{Isoperimetric inequality for nonlocal bi-axial discrete perimeter}
%\title{A sample article title with some additional note\thanksref{T1}}
\runtitle{Isoperimetric inequality for nonlocal bi-axial discrete perimeter}
%\thankstext{T1}{A sample of additional note to the title.}

\begin{aug}
%%%%%%%%%%%%%%%%%%%%%%%%%%%%%%%%%%%%%%%%%%%%%%%
%% Additional information such as            %%
%% identifying the corresponding author must %%
%% be included in in the Acknowledgments     %%
%% section if necessary.                     %%
\author[A]{\inits{V.}\fnms{Vanessa}~\snm{Jacquier}\ead[label=e1]{vanessa.jacquier@unipd.it}},
\author[B]{\inits{C.}\fnms{Cristian}~\snm{Spitoni}\ead[label=e2]{s.spitoni@uu.nl}}
\and
\author[B]{\inits{W.}\fnms{Wioletta M.}~\snm{Ruszel}\ead[label=e3]{w.m.ruszel@uu.nl}}
%%%%%%%%%%%%%%%%%%%%%%%%%%%%%%%%%%%%%%%%%%%%%%
%% Addresses                                %%
%%%%%%%%%%%%%%%%%%%%%%%%%%%%%%%%%%%%%%%%%%%%%%
\address[A]{University of Padova, via Trieste 63, 35131, Padova, Italy.\printead[presep={,\ }]{e1}}

\address[B]{Utrecht University, Budapestlaan 6, 3584 CD Utrecht, The Netherlands.\printead[presep={,\ }]{e2,e3}}
\end{aug}

\begin{abstract}
 In the present manuscript we address and solve for the first time a nonlocal discrete isoperimetric problem. We consider indeed a generalization of the classical perimeter, what we call a \emph{nonlocal bi-axial discrete perimeter},  where, not only the external boundary of a polyomino $\mathcal{P}$ contributes to the perimeter, but all internal and external components of $\mathcal{P}$.  
Furthermore,  we find and characterize its minimizers in the class of polyominoes with fixed area $n$. 
Moreover, we explain how the solution of the nonlocal discrete isoperimetric problem is related to the rigorous study of the metastable behavior of a \emph{long-range bi-axial Ising model}.
\end{abstract}



\end{frontmatter}

\section{Introduction}
Finding the set of minimal perimeter enclosing a prescribed volume is in a nutshell the classical \emph{isoperimetric} problem.  As proven in \cite{degiorgi} for the Euclidean space $\mathbb{R}^n$, the unique solution of the classical isoperimetric problem is the $n$-sphere.
 
A natural question is about the change of the optimal solutions in the presence of a \emph{nonlocal} term in the perimeter. In \cite{caffa2010}, for instance, the notion of \emph{nonlocal} perimeter has been introduced, and the study of the corresponding minimizing isoperimetric sets has been considered. In particular,  in \cite{figa2015} a sharp quantitative isoperimetric inequality for nonlocal $s$-fractional  $P_s(E)$ perimeters was given for the functional:
$$P_s(E):=\int_E \int_{E^c} \frac{d x d y}{\|x-y\|^{n+s}} $$
where $E\subset \mathbb{R}^n$ and $s\in (0,1)$.  As a consequence of the isoperimetric inequality given in \cite{figa2015}, the $n$-spheres are still the unique nonlocal minimal surface among sets with the same volume, where a nonlocal minimal surface is defined as the boundary of a set which minimizes the nonlocal perimeter.  Nonlocal minimal surfaces have been the object of extensive research: for example, they have been shown to arise as scaling limit of long-range phase coexistence models (see \cite{savin2012}), or as continuous approximations of interfaces of long-range Ising models (see \cite{Cozzi2017}).

Moreover, the competition between short-range and long-range interactions appears in several applications in biology and in physics, in particular in relation with the formation of patterns (see \cite{chak2011},\cite{de2000}). In the presence of a strong short-range attracting force (i.e., the perimeter) and a weak repulsive long-range force (i.e., nonlocal term) it has been shown that periodic stripes are the minimizers of the functional. This problem has been studied both in the continuum (see, for instance, \cite{giuliani2012}, \cite{daneri2019}) and in the discrete case (see~\cite{giuliani2016}).

In the present manuscript, we consider a discrete variational problem for the \emph{nonlocal bi-axial discrete} perimeter $Per_\lambda(\mathcal{P})$, defined as in (\ref{def:per}), for a general two-dimensional \emph{polyomino} $\mathcal{P}$. This type of generalized perimeter takes into account all contributions from the interactions between sites inside and outside the polyomino, weighted inversely proportional to the horizontal and vertical distances raised to the power $\lambda$; consequently, it does not depend solely on the interaction between interior and exterior boundary sites, as is the case for the usual perimeter. The aim of our paper is to find and characterize the minimizers of $Per_\lambda$  among the sets of all the polyominoes with fixed area $n\in\mathbb{N}$. 

Our main result is contained in Theorem~\ref{main_theorem},  stating that the set of these minimizers consists of squares, quasi-squares or rectangles with (or without) protuberances, namely strips of unit squares attached along one side.
In particular, we will show that if a square (quasi-square) and a rectangle with protuberances have the same area but different classical perimeter, then the square (quasi-square) will have smaller nonlocal perimeter. On the other hand, if their respective classical perimeters are equal, then either shape can be the minimum. Differently from the classical two-dimensional isoperimetric problem studied in \cite{alonso1996three}, the nonlocal term will introduce an anisotropy in the system: the quasi-square minimizers with a protuberance, for instance,  must have the protuberance along the shorter side.

The choice of the fractional perimeter is motivated by the form of the Hamiltonian $H^\lambda$ defined in \eqref{hamiltonian_function} of the \emph{two-dimensional long-range Ising model with bi-axial interactions} (i.e., with non-null interaction along the horizontal and the vertical directions), studied in  \cite{coquille2018}. There is a well known relation between long and short-range Ising models: many phenomena which appear for short-range Ising models in $d$ dimensions appear for long-range Ising models in $d-1$ dimensions, such as phase transitions or metastability phenomena (see, for instance, \cite{coquille2018, Affonso,vanenter2019}). In \cite{coquille2018} the authors studied the question of existence of Dobrushin states in two-dimensions, which are extremal, translation invariant Gibbs states. While it is known that Dobrushin states do not exist in two-dimensions for the nearest-neighbour case, the authors proved that they do exist for the two dimensional long-range models, including the bi-axial long-range interaction treated in the present manuscript. 

Moreover, axial Ising models are prototypes of complex spatially modulated magnetic superstructures in crystals. For example, in \cite{elliott1961} the \emph{Axial Next-Nearest Neighbor Ising model} (ANNNI)  was used to explain the existence of modulated phases in rare earth compounds. Furthermore, in real systems, higher-order neighbor interactions are often present, so that extensions of the ANNNI to the third-nearest-neighbor Ising (A3NNI) have been analyzed in \cite{randa1985}. Hence, the Hamiltonian $H^\lambda$  considered in the present manuscript can be considered to be the long-range extension of the Ising-like axial interaction.

A natural application of our result is in the context of the rigorous study of  metastability for the two-dimensional long-range Ising models. In \cite{vanenter2019} the metastability problem was solved in one dimension for a general class of long-range Ising interactions, and the critical configurations were completely characterized. However, at the moment, no rigorous results are known for a two-dimensional long-range Ising model. Our characterization of the solutions of the nonlocal discrete isoperimetric problem represents a key first step for finding the typical configurations that trigger the nucleation process.
 In fact, we will show in Section~\ref{s:application} that the solution of the variational problem for $Per_\lambda$ allows us to find the minimal energy configurations for the two-dimensional long-range Ising model with bi-axial interactions, in the subset of the state space with fixed magnetization.  
 Following the so-called \emph{path-wise approach} to metastability (see  \cite{olivieri2005} and \cite{manzo2004}), the crucial idea is finding the optimal path connecting the metastable state to the stable one and computing its energy height. Therefore, complete characterization of the energy minimizers in any \emph{foliation} with constant magnetization (i.e., with a constant number of plus spins) is a crucial ingredient for solving this variational problem on the entire path space.

The novelty of the paper lies in the fact that we have addressed and solved for the first time a nonlocal discrete isoperimetric problem.
%Moreover, we obtain the surprising results that a minimal rectangle with protuberance will always have the protuberance on the smaller side $a<b$. In fact, the difference in perimeter of putting the protuberance in the longer side $b$ versus the smaller side $a$ is of order $\mathcal{O}(\frac{b-a}{a^{\lambda}})$. 
The proof consists of several crucial observations. First, disconnected polyominoes have always larger nonlocal perimeter than their connected counterparts. Then, a minimizer of the nonlocal perimeter has to have a convex shape, so that concave polyominoes can be discarded in the minimization process.  We call a polyomino concave if it contains a vertical (horizontal) strip which is disconnected. In fact, we will use the property that translating unit squares in the polyomino horizontally (vertically) will either leave the nonlocal perimeter invariant, in case the lengths of horizontal (vertical)  stripes stay the same,  or decrease, in case we translate unit squares in order to fill holes, by definition present in a concave shape. Finally, within the class of convex polyominoes, by using monotonicity arguments and induction, we can characterize the set polyominoes with nonlocal minimal perimeter.

The paper is organized as follows. In Section \ref{sec:secnot} we will introduce all the notation, define the necessary quantities, and state the main result of the manuscript (i.e., Theorem~\ref{main_theorem}). Section \ref{sec:proofs} contains the proof of the main theorem which is based on several propositions. Finally, we will discuss in Section \ref{s:application} how our result is connected to the rigorous study of the metastable behavior of a bi-axial long-range Ising model.
We have postponed some auxiliary derivations and proofs to the Appendix.

\section{Notation and Main Result}\label{sec:secnot}

In this section we introduce the notation used in the paper and we state the main result.

Let $Q(x)\subset \mathbb{R}^2$ denote a \emph{unit square} of area one, such that the center point $x$ belongs to $\mathbb{Z}^2$ and corners to the dual lattice $(\mathbb{Z}+\frac{1}{2})^2$. We call $e_1, e_2$ the standard orthonormal basis of $\mathbb{R}^2$.
If two unit squares, $Q(x), Q(y)$, share a side, then we will call the union $Q(x) \cup Q(y)$ \emph{connected}. 

\begin{definition}
A polyomino $\mathcal{P}$ is a finite union of unit squares. The area $a_{\mathcal{P}}$ of the polyomino equals the number of its unit squares. Let $\mathcal{M}_n$ denote the set of all polyominos with area $n \in \mathbb{N}$.
\end{definition}

Note that, equivalently to \cite{alonso1996three}, a polyomino does not need to be a connected set. 
The {\it classical perimeter} $per (\mathcal{P})$ of a polyomino $\mathcal{P}$ is the length of its boundary. By abuse of notation, we will use the term {\it perimeter} for both the closed path and its length.  In the classical case, the perimeter of a polyomino $\mathcal{P}$ only depends on the squares $Q(x)$ which share a side with $\mathbb{R}^2\setminus \mathcal{P}$, see Figure \ref{fig:polyomino} for an example. 
\begin{figure}[htb!]
\begin{center}
    \includegraphics[scale=0.5]{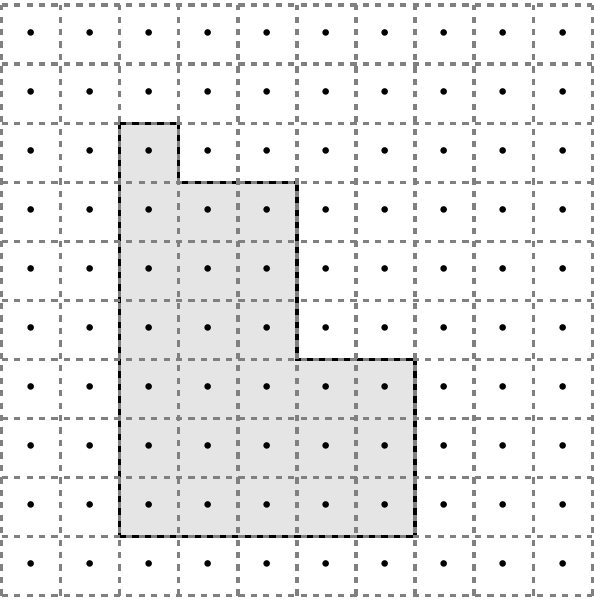}
    \caption{An example of polyomino with area 25 and classical perimeter 24.}
 \label{fig:polyomino}
\end{center}
\end{figure}
When we will refer to the Euclidean distance, we will write $d_E$. 
Given a polyomino $\mathcal{P}$ we define the \emph{internal boundary}  $\partial^- \mathcal{P}$ of $\mathcal{P}$ the set of the unit squares with centers at distance one from $\mathbb{Z}^2 \cap \mathcal{P}^c$, i.e. $\partial^- \mathcal{P}=\{ x \in \mathbb{Z}^2 \cap \mathcal{P} \, : \, d_E(x,y)=1 \text{ with } y \in \mathbb{Z}^2 \cap \mathcal{P}^c\}$, where $\mathcal{P}^c = \mathbb{R}^2 \setminus \mathcal{P}$. 

We define a \emph{row} $r$ and a \emph{column} $c$ the following sets of unit squares. Let $x=(x_1,x_2) \in \mathbb{Z}^2$, then we can define
\begin{align*}
    r_i:=\bigcup_{x_2 \in \mathbb{Z}}Q(i,x_2), \\
    c_i:=\bigcup_{x_1 \in \mathbb{Z}} Q(x_1,i).
\end{align*}
In the sequel we will drop the explicit dependence on $i$ in the notation. For instance, with an abuse of notation, the sum $\sum_{r}f(r)$ should be interpreted as $\sum_{i= - \infty}^\infty f(r_i)$.\\
A (horizontal or vertical) {\it strip} of $\mathcal{P}$ is a subset of connected unit squares with centers in $\mathcal{P}\cap \mathbb{Z}^2$ such that they share the same column or row, i.e. there exists a $z\in  \mathcal{P}\cap \mathbb{Z}^2$ such that
\begin{align}
&S_v= \bigcup \{ Q(x_k) \in \mathcal{P} \, | \, x_k=z+ke_2 \in \, \mathcal{P}\cap \mathbb{Z}^2 \text{ with } k \in \mathbb{Z}, \, \text{ and } d_E(x_k,x_{k+1})=1 \}, \\
&S_h= \bigcup \{ Q(x_k) \in \mathcal{P} \, | \, x_k=z+ke_1 \in \, \mathcal{P}\cap \mathbb{Z}^2 \text{ with } k \in \mathbb{Z}, \, \text{ and } d_E(x_k,x_{k+1})=1 \}. \label{def:strips_row}
\end{align}
The \emph{length} of the strip $S$ is the number of the squares in $S$.
For clarity's sake, we will drop the dependence on $z$ in the sequel.

\begin{definition}
We call \emph{single protuberance of $\mathcal{P}$} a unit square with only one unit edge shared with other unit squares of $\mathcal{P}$.
A \emph{$k$-protuberance} (or a \emph{protuberance with cardinality $k$}) is a strip of length $k$, $S:=\bigcup_{i=1}^k Q(x_i) \subset \mathcal{P}$, such that each unit square $Q(x_i)$ composing it shares a unit edge with $Q(x_{i+1})$, one with $\mathcal{P} \setminus S$, and one with $\mathcal{P}^c$ for $i=1,...,k-1$.
A \emph{$k$-protuberance} (or a \emph{protuberance with cardinality $k$}) is a strip of length $k$, $S:=\bigcup_{i=1}^k Q(x_i) \subset \mathcal{P}$, such that for $i=1,...,k-1$ each unit square $Q(x_i)$ composing it shares a unit edge with $Q(x_{i+1})$, one with $\mathcal{P} \setminus S$, and one with $\mathcal{P}^c$; moreover, $Q(x_1)$ (resp. $Q(x_k)$) shares a unit edge with $Q(x_{2})$ (resp. $Q(x_{k-1})$), one with $\mathcal{P} \setminus S$, and two with $\mathcal{P}^c$.
See the center and right-hand side of Figure \ref{fig:polyomini_Mn} for an example.
\end{definition}

A polyomino with a rectangular shape is called a \emph{rectangle}, and we denote it by $\mathscr{R}_{a,b}$ where $a,b$ are the side lengths. Moreover, we call a \emph{square} (resp. \emph{quasi-square}) a rectangle with side length $a=b$ (resp. $a=b-1$) and we denote it by $\mathcal{Q}_a$ (resp. $\mathscr{R}_{a,a+1}$).
Finally, we denote a rectangle (resp. a square) with side lengths $a,b$ (resp. $a$) and a $k$-protuberance attached along the shortest side by $\mathscr{R}_{a,b}^k$ (resp. $\mathcal{Q}_a^k$). 

Recall that we can write any integer number $n$ as either $n=l^2+k_1$ or $n=l(l+1)+k_2$ for some $l,k_1,k_2 \in \mathbb{N}$ such that $0 \leq k_1 \leq l-1$ and $0 \leq k_2 \leq l$. Let $k\in \mathbb{N}$ such that $0 \leq k \leq a-1$.
Hence, we define the following sets: 
\begin{align}
    &\mathscr{M}_{n}:= 
    \begin{cases}
      \left  \{ \mathscr{Q}_l^{k_1} \in \mathcal{M}_n \right \} 
        \cup \left\{ \mathscr{R}_{a,b}^{k} \in \mathcal{M}_n \, | \, per(\mathscr{R}_{a,b}^{k})=per(\mathscr{Q}_l^{k_1}) \right \} \,\, &\text{if} \,\, n=l^2+k_1=ab+k, \\
   \left \{ \mathscr{R}_{l,l+1}^{k_2} \in \mathcal{M}_n \right \} 
        \cup \left\{ \mathscr{R}_{a,b}^{k} \in \mathcal{M}_n \, | \, per(\mathscr{R}_{a,b}^{k})=per(\mathscr{R}_{l,l+1}^{k_2}) \right\} \,\, &\text{if} \,\, n=l(l+1)+k_2=ab+k,
        \end{cases} 
        \label{minimizers1} \\
    &\mathscr{M}_n^{\text{ext}}:=\mathscr{M}_n \cup \left\{\mathcal{P} \in \mathcal{M}_n \setminus \mathscr{M}_n \, | \, \mathcal{P}=\mathscr{R}^k_{a,b} \right \}. \label{minimizers2}
\end{align}
In Figure~\ref{fig:polyomini_Mn} we consider, as an example, three different polyominoes belonging to the set $\mathscr{M}_n$.
\begin{figure}[htb]
\begin{center}
    \includegraphics[scale=0.55]{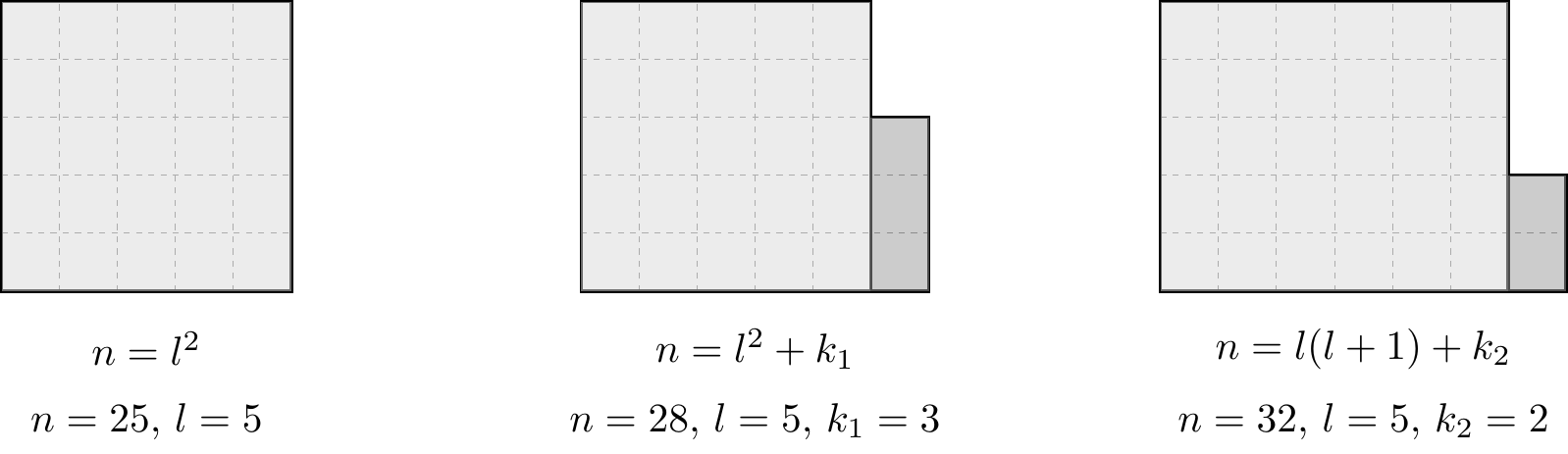}
    \caption{Examples of $\mathcal{P}\in\mathscr{M}_n$ for different $n$.}
 \label{fig:polyomini_Mn}
\end{center}
    \end{figure}

\begin{definition}
A polyomino is called \emph{concave} if there are at least two not connected  strips along the same column or row. A \emph{convex} polyomino is a polyomino which is not concave. Consider a convex polyomino $\mathcal{P}$ and let $\mathscr{R}_{a,b}$ be the smallest circumscribing rectangle of $\mathcal{P}$. If there exist two rectangles $\mathscr{R}_{a,u}$ and $\mathscr{R}_{w,b}$ with $u\leq b$ and $w \leq a$ such that they are contained in $\mathcal{P} \cap \mathscr{R}_{a,b}$, then we call $\mathcal{P}$ a \emph{cross-convex} polyomino.
\end{definition}
 See Figure \ref{fig:concave_clusters} for an example of concave, convex and cross-convex polyomino.

\begin{figure}[htb]
\begin{center}
    \includegraphics[scale=0.5]{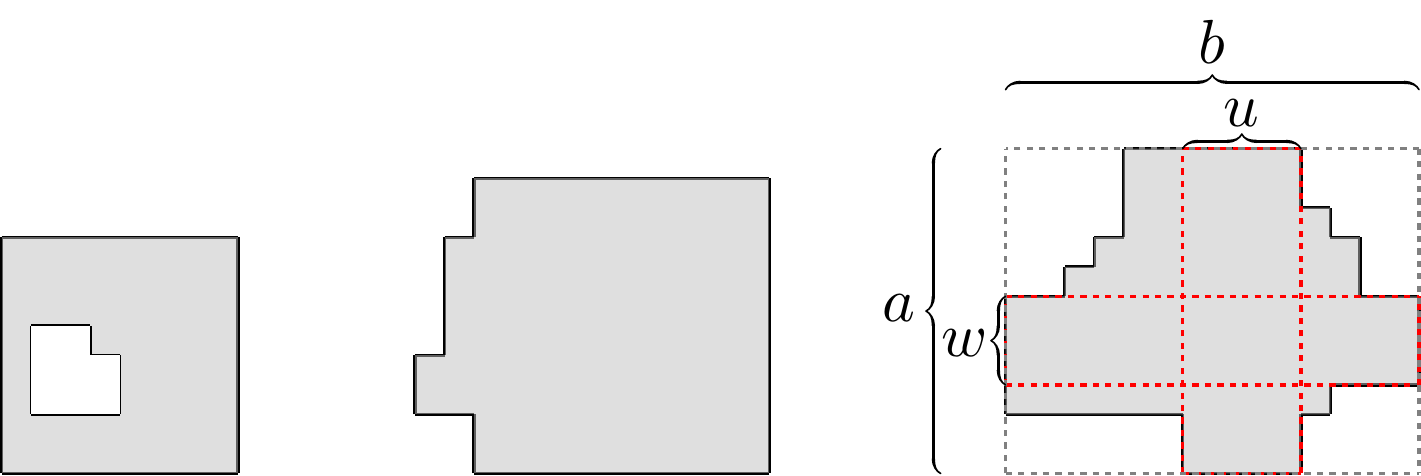}
    \caption{Starting from the left hand side, a concave polyomino, a convex polyomino and a cross-convex polyomino.}
 \label{fig:concave_clusters}
\end{center}
    \end{figure}

\noindent
We are now going to define the discrete version of a nonlocal perimeter for a given polyomino $\mathcal{P}$:

\begin{definition}
The nonlocal perimeter $Per_{\lambda}(\mathcal{P})$ of the polyomino $\mathcal{P}$ with parameter $\lambda$, where $\lambda>1$, is defined as: 
\begin{align}
\label{def:per}
    Per_{\lambda}(\mathcal{P}):=\sum_{x \in \mathbb{Z}^2 \cap \mathcal{P}, \, y \in \mathbb{Z}^2 \cap \mathcal{P}^c} \frac{1}{d^{\lambda}(x,y)},
\end{align}

where $d^{\lambda}(x,y)$ is the \emph{fractional bi-axial  function} defined by the relation:
 
\[
\frac{1}{d^{\lambda}(x,y)} := \frac{1}{(d_V(x,y))^\lambda}\textbf{1}_{\{ x_1=y_1, \, x_2 \neq y_2\}}  + \frac{1}{(d_H(x,y))^{\lambda}} \textbf{1}_{\{  x_2=y_2, \, x_1 \neq y_1\}},
\] 
where
$d_V(x,y)=|x_2-y_2|$, $d_H(x,y)= |x_1-y_1|$, and $x=(x_1,x_2)$, $y=(y_1,y_2)$.
\end{definition}
We note that in the case of the bi-axial function, the nonlocal perimeter can be written as the sum of the horizontal and vertical contributions, i.e.:
\begin{align}\label{eq:hor+ver}
    &Per_{\lambda}(\mathcal{P})=Per_{\lambda}^H(\mathcal{P})+Per_{\lambda}^V(\mathcal{P}),
\end{align}
where
\begin{align*}
   &Per_{\lambda}^H(\mathcal{P}):=\sum_{x \in \mathbb{Z}^2 \cap \mathcal{P}, \, y \in \mathbb{Z}^2 \cap \mathcal{P}^c}\frac{1}{(d_H(x,y))^{\lambda}} \textbf{1}_{\{  x_2=y_2, \, x_1 \neq y_1\}}, \\
   &Per_{\lambda}^V(\mathcal{P}):=\sum_{x \in \mathbb{Z}^2 \cap \mathcal{P}, \, y \in \mathbb{Z}^2 \cap \mathcal{P}^c} \frac{1}{(d_V(x,y))^{\lambda}}\textbf{1}_{\{ x_1=y_1, \, x_2 \neq y_2\}}.
\end{align*}

\medskip

We state now our main result in terms of an isoperimetric inequality for the discrete nonlocal perimeter.
\begin{theorem}\label{main_theorem}
Fix $n\in \mathbb{N}$. There exists a constant independent of $n$, $\lambda_c$, such that for all $\lambda>\lambda_c >1$ and all polyominos $\mathcal{P}\notin \mathscr{M}_n$ with area $a_{\mathcal{P}} = n$ we have that:
\[
Per_{\lambda}(\mathcal{P}) > Per_{\lambda}(\mathcal{R}),
\]
for any $\mathcal{R} \in \mathscr{M}_n$. 
\end{theorem}
We remark that the statement holds true taking $\lambda_c \approx 1.78788$. This value is obtained numerically from computing the smallest value of $\lambda$, such that the perimeter of a square $\mathscr{Q}_2$ is smaller than the one of a rectangle $\mathscr{R}_{1,4}$ in the proof of Proposition \ref{prop:rectangle_square}. 
For simplicity, we will work with $\lambda_c=1.8$.

In order to give an example of the calculation of nonlocal perimeter, we consider the polyomino $\mathcal{P}$ of Figure~\ref{fig:polyomino}, whose nonlocal perimeter is: 
 \[
\begin{split} Per_{\lambda}(\mathcal{P})&=
  2\sum_{r=1}^{\infty} \frac{1}{r^\lambda} + 10 \sum_{i=1}^3\sum_{r=i}^{\infty} \frac{1}{r^\lambda}+ 6\sum_{i=1}^5\sum_{r=i}^{\infty} \frac{1}{r^\lambda}   +4\sum_{i=1}^6\sum_{r=i}^{\infty} \frac{1}{r^\lambda} + 2\sum_{i=1}^7\sum_{r=i}^{\infty} \frac{1}{r^\lambda},\\
& =  24\zeta(\lambda) + 22 \sum_{i=2}^3\zeta(\lambda, i) + 12\sum_{i=4}^5  \zeta(\lambda, i) + 6  \zeta(\lambda, 6) + 2 \zeta(\lambda, 7),
   \end{split}
\]
where $\zeta(\lambda, i)$ denotes the Hurwitz-zeta function, defined as
\begin{equation}
\label{e:zeta}
\zeta(\lambda, i) =\sum_{r=0}^{\infty} \frac{1}{(r+i)^{\lambda}},
\end{equation}
and $\zeta(\lambda) = \zeta(\lambda, 1)$ is just the classical zeta-function. Note that for $\lambda \rightarrow \infty$ we have that all Hurwitz-zeta functions $\zeta(\lambda, i)$ with $i\geq 2$ will converge to 0 while $\zeta(\lambda) \rightarrow 1$. Hence, for $\lambda \rightarrow \infty$ the nonlocal perimeter coincides with the classical one.  

    Moreover, in Figure \ref{fig:minimizer_tab} we have reported the minimizers of $Per_{\lambda}$ belonging to $\mathscr{M}_n$, for areas $n\in\{1,\ldots,30\}$. In particular, we see that, consistently with Theorem~\ref{main_theorem}, there are two possible candidates for the minimal shape for the areas $n\in\{10,17,28\}$. If $\lambda$ is small enough, the rectangular shape will have smaller nonlocal perimeter (i.e., it will be the unique minimizer), while for $\lambda$ large enough the square shape with a protuberance will be in this case the unique minimizer. The values of $\lambda$, indicating the change of shape, were obtained numerically.

\begin{figure}[htb]
\begin{center}
    \includegraphics[scale=0.24
    ]{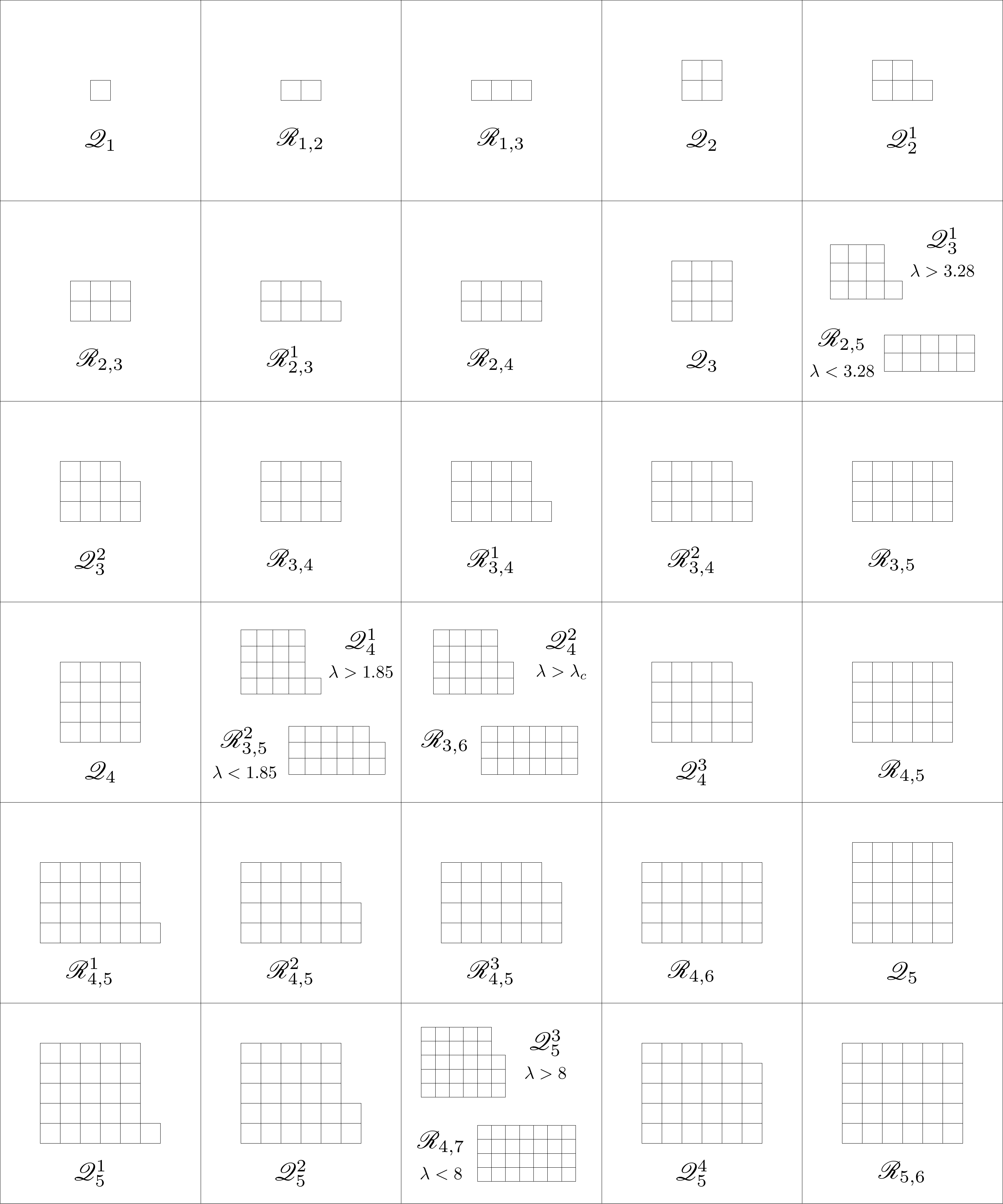} 
    \caption{Minimizers of $Per_{\lambda}$ in $\mathscr{M}_n$ for $n\in\{1,\dots, 30\}$.}
 \label{fig:minimizer_tab}
\end{center}
    \end{figure}

%%%%%%%%%%%%%%%%%%%%%%%%%%%%%%%%%%%%%%%%%
\section{Proof of Theorem~\ref{main_theorem}}\label{sec:proofs}
%%%%%%%%%%%%%%%%%%%%%%%%%%%%%%%%%%%%%%%%%

The structure of the proof is as follows. First, in Proposition \ref{prop:no_rectangle} we will prove that the minimizers of the nonlocal perimeter belong to $\mathscr{M}_n^{\text{ext}}$ which was defined in Equation \eqref{minimizers2}. The proof of Proposition \ref{prop:no_rectangle} is postponed to Section~\ref{poof:proop3.1}. This proposition is the most involved one and it represents the major novelty of the paper. We develop indeed an algorithmic proof (what we call the \texttt{\underline{MAIN ALGORITHM)}} allowing us to associate to any configuration in $\mathcal{M}_n \setminus \mathscr{M}^{\text{ext}}_n$ a configuration in  $\mathcal{M}_n$ with  smaller generalized perimeter and same area. In particular, by the iterative application of Lemmas \ref{reduction_distance}  and \ref{lem:reduction_strips} we exclude from the set of minimizers all polyominoes in $\mathcal{M}_n \setminus \mathscr{M}^{\text{ext}}_n$ that are not cross-convex, e.g. non-connected or concave polyominoes. The application of the \texttt{\underline{ALGORITHM FOR CROSS-CONVEX $\mathcal{P}$}} will further exclude cross-convex polyominoes which are in $\mathcal{M}_n \setminus \mathscr{M}^{\text{ext}}_n$ from the set of possible minimizers.

Furthermore, via Propositions \ref{prop:rectangle_square}, \ref{prop:rectangle_quasisquare}, \ref{prop:rectangle_strip_attached1} and \ref{prop:rectangle_strip_attached2} 
we prove that a polyomino with area $n$ belonging to $\mathscr{M}_n^{\text{ext}} \setminus \mathscr{M}_n$ has nonlocal perimeter larger than that of a polyomino in $\mathscr{M}_n$. 
More precisely, in Proposition \ref{prop:rectangle_square} (resp. in Proposition \ref{prop:rectangle_quasisquare}) we show that, given an area $n=l^2$ (resp. $n=l(l+1)$), for $\lambda>\lambda_c$ the nonlocal perimeter of any rectangle with this area is larger than the nonlocal perimeter of the square (resp. the quasi-square) with the same area. The proof will rely on a careful comparison of the respective nonlocal perimeter functions. The proof of Proposition \ref{prop:rectangle_square} will be given in Section~\ref{proof:prop3.2}, while the proof of Proposition \ref{prop:rectangle_quasisquare} will be postponed to the Appendix~\ref{proof:prop3.3}, since it has a very similar structure.
Finally, in Proposition \ref{prop:rectangle_strip_attached1} (resp. in Proposition \ref{prop:rectangle_strip_attached2}),  given an area $n=l^2+k_1$ (resp. $n=l(l+1)+k_1$), we prove that the nonlocal perimeter of the square (resp. quasi-square) with a protuberance and area $n$, i.e., $\mathscr{Q}_l^{k_1}$ (resp. $\mathscr{R}_{l,l+1}^{k_1}$), is strictly smaller than the respective nonlocal perimeter of a rectangle with a protuberance and the same area which has a larger classical perimeter than the one of $\mathscr{Q}_l^{k_1}$ (resp. $\mathscr{R}_{l,l+1}^{k_1}$).
The proof of Proposition \ref{prop:rectangle_strip_attached1} will be given in Section~\ref{prop:3.4}, while the proof of Proposition \ref{prop:rectangle_strip_attached2} will be given in Appendix~\ref{prop:3.5}.
\\
A schematic representation of the proof of Theorem~\ref{main_theorem} is given in Figure~\ref{fig:diagramm}.

\begin{figure}[htb]
\begin{center}
    \includegraphics[scale=0.21]{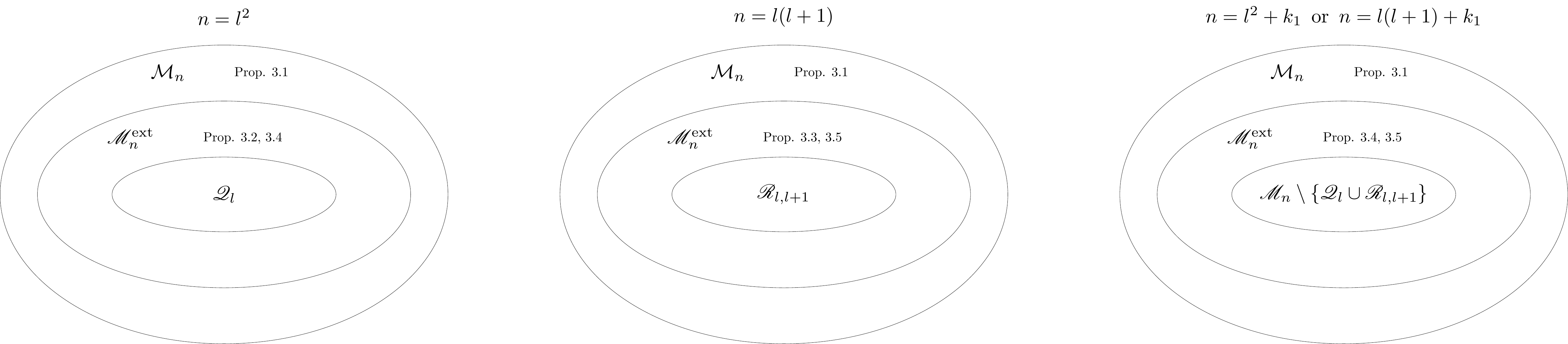}
    \caption{ A schematic representation of the proof of Theorem \ref{main_theorem}. We distinguish three cases according to the value of $n$. We reduce the set of possible minimizers in any of the cases by each shell. The corresponding proposition doing that is indicated in each shell. }
 \label{fig:diagramm}
\end{center}
    \end{figure}

In the following, we will formulate the precise statements of the propositions used in the proof of the main theorem. 

\begin{proposition}\label{prop:no_rectangle}
  There exists $\lambda_c>1$ such that for $\lambda > \lambda_c$, $n \in \mathbb{N}\setminus \{1\}$ and $\mathcal{P} \in \mathcal{M}_{n}$ be a polyomino such that $\mathcal{P} \not \in \mathscr{M}_n^{\text{ext}}$, we have that $$\mathcal{P} \not \in \text{argmin}_{\mathcal{P}' \in \mathcal{M}_{n}} \{Per_{\lambda}(\mathcal{P}') \}.$$
\end{proposition}

\begin{proposition}\label{prop:rectangle_square}
 There exists $\lambda_c>1$ such that for $\lambda > \lambda_c$, $n \in \mathbb{N}\setminus \{1\}$ with $n=l^2=a \cdot b$, $a<b$ and $l, a, b \in \mathbb{N}$  we have that
$Per_{\lambda}(\mathscr{R}_{a,b}) > Per_{\lambda}(\mathscr{Q}_l)$.
\end{proposition}

\begin{proposition}\label{prop:rectangle_quasisquare}
 There exists $\lambda_c>1$ such that for $\lambda > \lambda_c$, $n \in \mathbb{N}\setminus \{1\}$ with $n=l(l+1)=a \cdot b$, $a<b$ and $l, a, b \in \mathbb{N}$  we have that
$Per_{\lambda}(\mathscr{R}_{a,b}) > Per_{\lambda}(\mathscr{R}_{l,l+1})$.
\end{proposition}

\begin{proposition}\label{prop:rectangle_strip_attached1}
Let $a, b, l, k_1, k_2, n \in \mathbb{N}$ be natural numbers and 
let $k_1 \in \{0, \ldots, l-1\}$ and $k_2\in \{0, \ldots, b-1\}$ be two integer numbers such that $k_1+k_2 \neq 0$. Furthermore, we assume that $n=ab+k_2=l^2+k_1$ and $n\geq 2$.
Then we have $\text{argmin}_{\mathcal{P} \in \mathcal{M}_{n}} \{Per_{\lambda}(\mathcal{P}) \} \subset \mathscr{M}_n$, where $\mathscr{M}_n$ was defined in \eqref{minimizers1}.
\end{proposition}

In particular, we note that if $k_1=0$, then by \cite{alonso1996three} it follows that there are no rectangles, or rectangles with protuberances, and perimeter equal to $per(\mathscr{Q}_l)=4l$. Thus, in the case $k_1=0$ (resp. $n=l^2$), by Propositions \ref{prop:rectangle_square}, \ref{prop:rectangle_strip_attached1}, \ref{prop:no_rectangle} we have $\text{argmin}_{\mathcal{P} \in \mathcal{M}_{n}} \{Per_{\lambda}(\mathcal{P}) \}=\{ \mathscr{Q}_l \}$. 

\begin{proposition}\label{prop:rectangle_strip_attached2}
Let $a, b, l, k_1, k_2, n \in \mathbb{N}$ be natural numbers and 
let $k_1 \in \{0, \ldots, l\}$ and $k_2\in \{0, \ldots, b-1\}$ be two integer numbers such that $k_1+k_2 \neq 0$. Furthermore, we assume that $n=ab+k_2=l(l+1)+k_1$ and $n\geq 2$.
We have then $\text{argmin}_{\mathcal{P} \in \mathcal{M}_{n}} \{Per_{\lambda}(\mathcal{P}) \} \subset \mathscr{M}_n$, where $\mathscr{M}_n$ was defined in \eqref{minimizers1}.
\end{proposition}

Again, it follows by \cite{alonso1996three} that, in case $k_1=0$  (resp. $n=l(l+1)$), there are no rectangles (different from the quasi-square), or rectangles with protuberances, and classical perimeter equal to $per(\mathscr{R}_{l,l+1})=4l+2$. Thus, for $k_1=0$, by Propositions \ref{prop:rectangle_quasisquare}, \ref{prop:rectangle_strip_attached2}, \ref{prop:no_rectangle} we have that $\text{argmin}_{\mathcal{P} \in \mathcal{M}_{n}} \{Per_{\lambda}(\mathcal{P}) \}=\{ \mathscr{R}_{l,l+1} \}$.

\subsection{Auxiliary lemmas about $Per_{\lambda}(\mathcal{P})$}
\label{prop:auxiliary}

In order to prove the propositions, we will need the following lemmas regarding properties of polyominoes and of their nonlocal perimeters. For instance, Lemma~\ref{lem:rows} will give a useful way to compute the horizontal (resp. vertical) contribution of the nonlocal perimeter for general polyominos $\mathcal{P}$ by summing the contribution along each row (resp. column).
The proofs of these lemmas are postponed to Section 
\ref{proof_lemmas}.

\begin{lemma}\label{lem:rows}
Let $N\in \mathbb{N}$. Consider a general polyomino $\mathcal{P}$ intersecting $m\in \mathbb{N}$ rows, $r_1,\ldots, r_m$. Given a fixed intersecting row $r$, let $S_{r}^{(1)},\ldots,S_{r}^{(N)}$ be the $N$ horizontal strips in $\mathcal{P} \cap r$ defined in \eqref{def:strips_row} following the lexicographic order. Let $l_1,\ldots,l_N$ be the lengths of $S_r^{(1)},\ldots,S_r^{(N)}$ respectively, and let $d^{(r)}_j$ be the distance between $S^{(j)}_r$ and $S^{(j+1)}_r$ for $j=1,\ldots,N-1$. Then we have for all $\lambda>1$, that the horizontal contribution of the nonlocal perimeter is equal to
\begin{align}\label{eq:contribution_strips}
    Per_{\lambda}^{H}(\mathcal{P} )= Per_{\lambda}^H \left ( \bigcup_{j=1}^{m}  \{ \mathcal{P} \cap r_j\} \right ) =\sum_{j=1}^{m} Per_{\lambda}^{H}(\mathcal{P} \cap r_j).
\end{align}
Moreover,
\begin{align*}
&   Per_{\lambda}^{H}(\mathcal{P} \cap r)\\
    &=\sum_{i=1}^{l_1}\zeta(\lambda, i)
    +\sum_{j=2}^N\sum_{i=1}^{l_j} \left ( \zeta(\lambda, i) -
    \sum_{k=d^{(r)}_{j-1}+i}^{i+l_{j-1}+d^{(r)}_{j-1}-1}\frac{1}{k^\lambda} \right )
    +\sum_{i=1}^{l_N}\zeta(\lambda, i)
    +\sum_{j=1}^{N-1}\sum_{i=1}^{l_j} \left ( \zeta(\lambda, i) -
    \sum_{k=d^{(r)}_{j}+i}^{i+l_{j+1}+d^{(r)}_{j}-1}\frac{1}{k^\lambda} \right ). \notag \\
\end{align*}
In particular, if $N=1$, then
\begin{align}\label{eq:contribution_strips1}
    Per_{\lambda}^{H}(\mathcal{P} \cap r) =2 \sum_{i=1}^{l_1}\zeta(\lambda, i).
\end{align}
The analogous results hold by considering the vertical contribution of the nonlocal perimeter.
\end{lemma}
\begin{lemma}\label{reduction_distance}
Let $\mathcal{P}$ be a polyomino (or a union of polyominos) that contains two vertical (resp. horizontal) strips at distance $d\geq 2$, and let $\mathcal{P}'$ be the polyomino obtained from $\mathcal{P}$ by shifting one of the two strips toward the other, and by reducing the distance by at least one. Then the nonlocal perimeter contribution of $\mathcal{P}'$ along the column (resp. the row) containing the two strips is strictly smaller than the nonlocal perimeter contribution of $\mathcal{P}$ along the same column (resp. row). 
\end{lemma}

\begin{lemma}\label{lem:reduction_strips}
Let $S_1,\dots,S_N$ be $N$, $N \in \mathbb{N}$, vertical strips belonging to the same column $c_{\mathcal{P}}^1$ with lengths $l_1,\ldots,l_N \geq 1$. Furthermore, let $d_i$ be the distance between the strips $S_i$ and $S_{i+1}$ for every $i=1,\ldots,N-1$. Let $A_i$ be the set that contains the rows that intersect $S_i \cap \mathcal{P}$ for $i=1,\ldots,N$. Denote by $c_{\mathcal{P}}^2$, $c_{\mathcal{P}}^2 \neq c_{\mathcal{P}}^1$, another column that intersects $\mathcal{P}$ and let $S_1^{(1)},\ldots,S_1^{(m_1)}, S_2^{(1)},\ldots,S_2^{(m_2)},\ldots,S_N^{(1)},\ldots, S_N^{(m_N)}$, with $m_1,\ldots,m_N \in \mathbb{N}$, be the vertical strips obtained by intersecting $c^2_{\mathcal{P}} \cap \mathcal{P} \cap A_i$ for $i=1,\ldots,N$. If the length $l_i' \geq 0$, of $\bigcup_{j=1}^{m_i} S_i^{(j)}$ is such that $l_i' \leq l_i$ for every $i=1,\ldots,N$, then there exists another polyomino $\mathcal{P}'$ such that the nonlocal perimeter along these two columns satisfies $Per_{\lambda}(c_{\mathcal{P}}^1 \cup c_{\mathcal{P}}^2) \geq Per_{\lambda}(c_{\mathcal{P'}}^1 \cup c_{\mathcal{P}'}^2)$, otherwise $Per_{\lambda}(c_{\mathcal{P'}}^1 \cup c_{\mathcal{P}'}^2) =  Per_{\lambda}(c_{\mathcal{P}}^1 \cup c_{\mathcal{P}}^2)$, where $c^{1,2}_{\mathcal{P'}}$ are the same columns in the polyomino $\mathcal{P}'$.

Again, the analogous results hold by considering the horizontal strips belonging to the same row. %horizontal contribution of the nonlocal perimeter.
\end{lemma}

\begin{lemma}\label{lem:increasing_alpha_2}
The following function $f: \mathbb{N} \setminus \{1\} \to \mathbb{R}$ defined as
    \begin{align*}
        f(x)= & x^2 \sum_{r=1}^{x^2+x} \frac{1}{r^\lambda} - x^2 \sum_{r=1}^{(x+1)^2} \frac{(x+1)^2+1-r}{r^\lambda}- (x+1)^2 \sum_{r=1}^{x^2-1} \frac{x^2-r}{r^\lambda} + 2(x^2+x)\sum_{r=1}^{x^2+x-1} \frac{x^2+x-r}{r^\lambda} -x^2 \sum_{r=x^2+x+1}^{x^2+2x+2} \frac{1}{r^\lambda}
    \end{align*}
    is positive for $\lambda>1.8$.
\end{lemma}

\subsection{Proof of Proposition~\ref{prop:no_rectangle}}
\label{poof:proop3.1}

Let us consider a polyomino $\mathcal{P} \in \mathcal{M}_n \setminus \mathscr{M}_n^{\text{ext}}$  and apply the following \texttt{\underline{MAIN ALGORITHM}} which we will explain below. The algorithm is based on the idea of investigating columns resp. rows of the general polyomino and try to fill \emph{holes}, in order to reduce the nonlocal perimeter. In case the main algorithm produces a cross-convex polyomino, we can exclude it from the set of minimizers by applying the \texttt{\underline{ALGORITHM FOR CROSS-CONVEX $\mathcal{P}$}} which will be defined below as well.

\begin{figure}[htb!]
\begin{center}
    \includegraphics[scale=0.27]{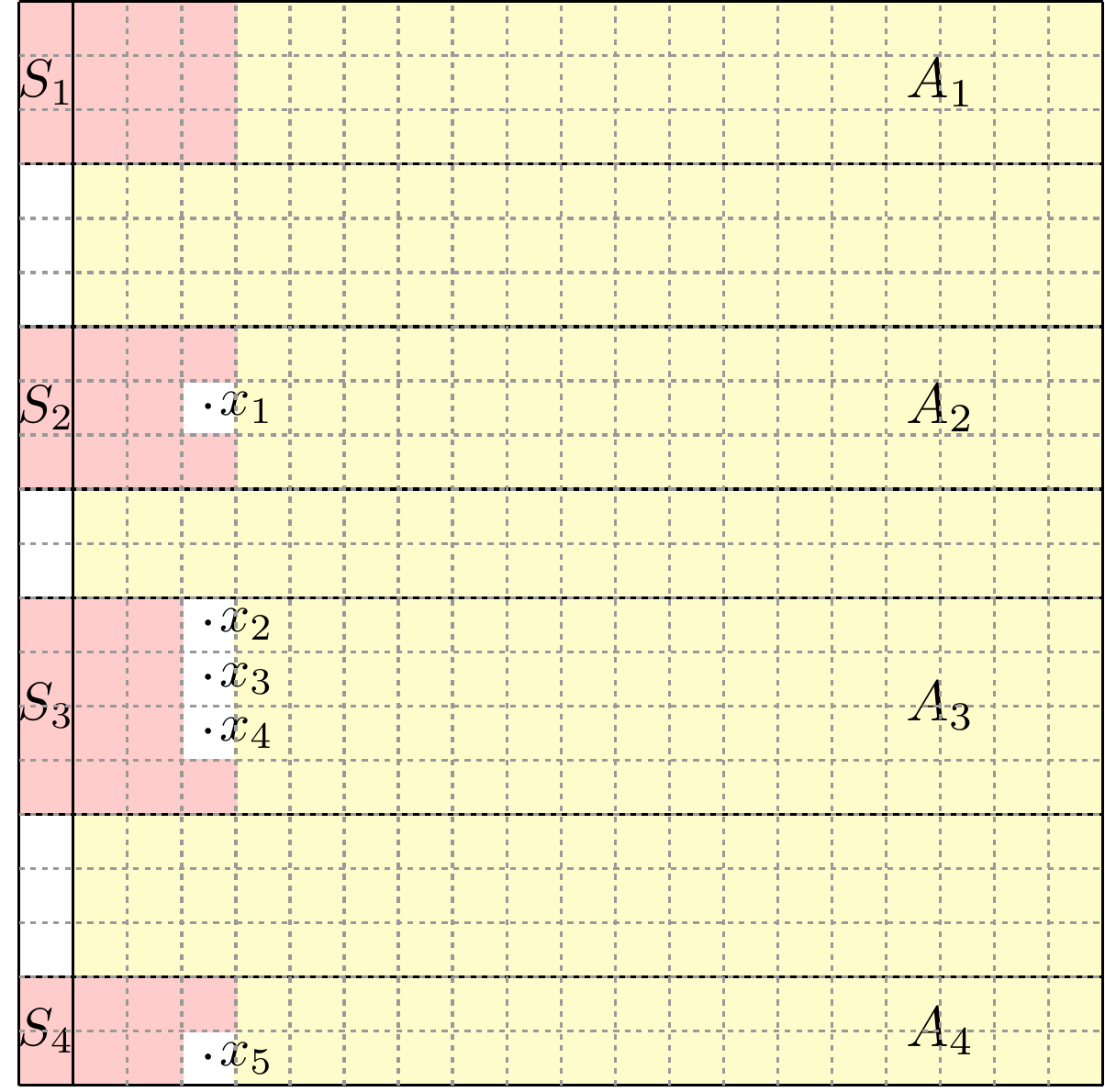} 
\,\,\,\,\,
    \includegraphics[scale=0.27]{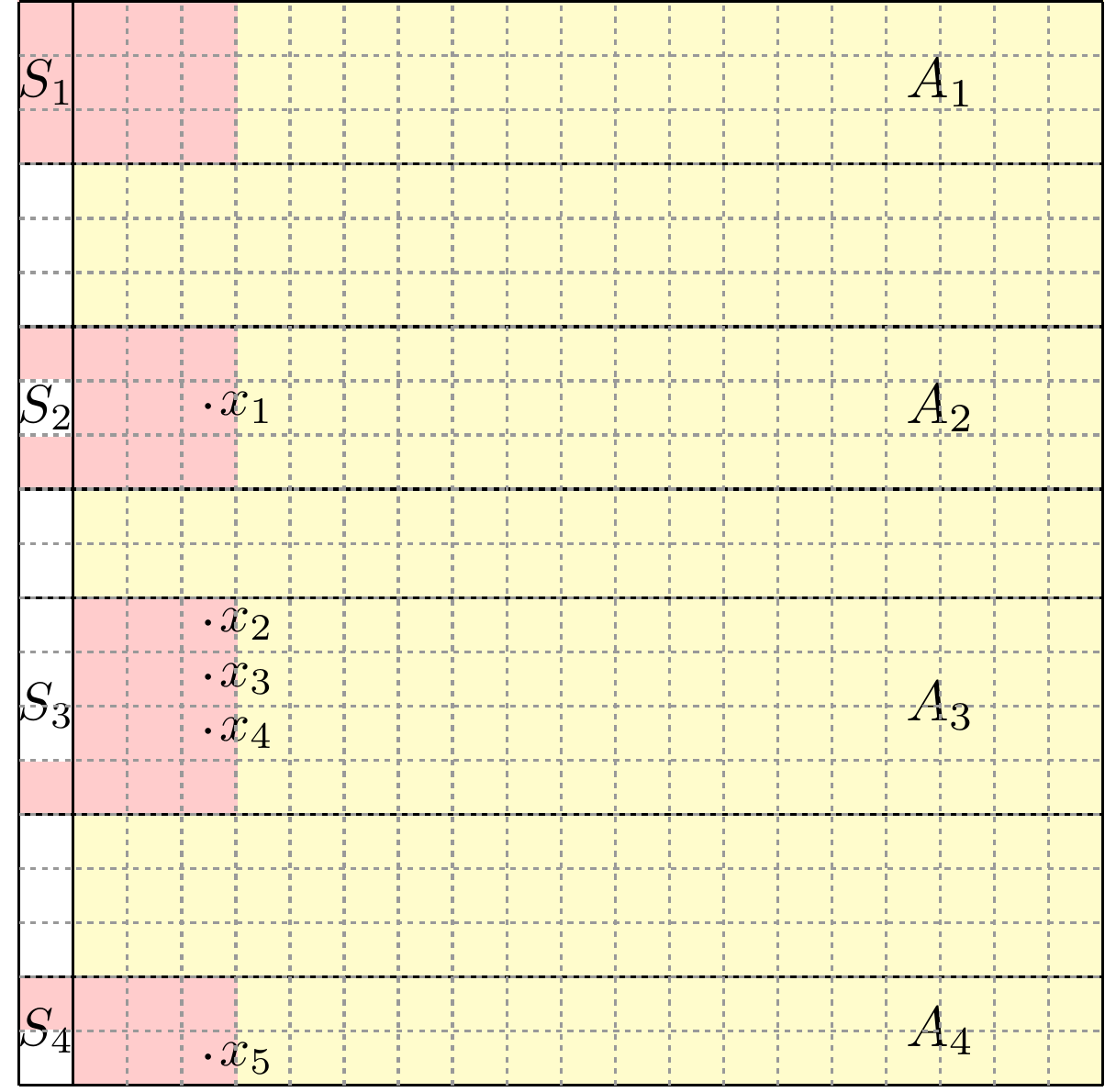}
\,\,\,\,\,
    \includegraphics[scale=0.27]{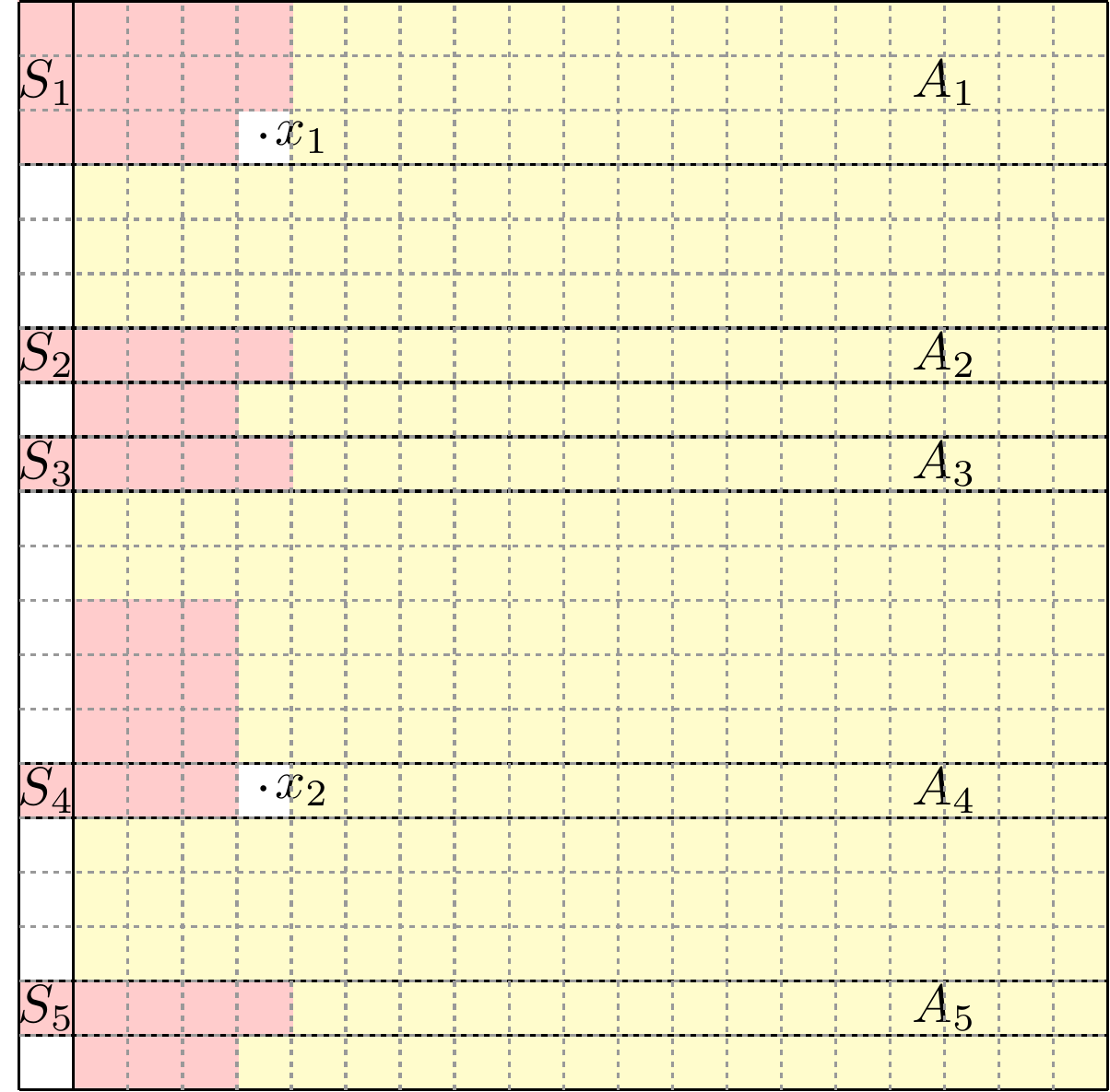} 
    \caption{On the left hand side, an example of polyomino $\mathcal{P}$ with vertical strips $S_1, \ldots ,S_4$ and the corresponding rows $A_1, \ldots ,A_4$. The red unit squares are part of $\mathcal{P}$ whereas the white ones are not. In yellow are displayed the unit squares in the rectangle that can belong to $\mathcal{P}$ or not, i.e. were not investigated yet.
    In the central figure, there is the polyomino $\mathcal{P}'$ obtained from $\mathcal{P}$ by a translation along the horizontal axes of the unit squares  $Q(x_1), \ldots, Q(x_5)$. On the right hand side, we have the polyomino $\mathcal{P}''$ obtained from the polyomino $\mathcal{P}'$ in the case that there is only one horizontal strip in all rows with an empty square $Q(x'_1), Q(x'_2)$ which are not in the polyomino. }
 \label{fig:no_rectangle}
\end{center}
    \end{figure}

\texttt{\large \underline{MAIN ALGORITHM}}
\begin{itemize}
\item[\texttt{\underline{Step 0}}.] {(Initialization). Set $ \mathcal{ \tilde P}:=\mathcal{P}$ and $t:=1$.}

    \item[\texttt{\underline{Step 1}}.] (Setup). Consider the smallest rectangle $\mathscr{R}_{a,b}$ containing $\mathcal{\tilde{P}}$ with $a\in \mathbb{N}$ the vertical side length and $b$ the horizontal side length. Order the columns of $\mathscr{R}_{a,b}$ according to the lexicographic order and analyze the $t$-th column of $\mathscr{R}_{a,b}$.
    
    Let $N \geq 1$ be the number of vertical strips in the $t$-th column of $\mathcal{\tilde{P}} \cap \mathscr{R}_{a,b}$ which we denote by $S_1,\ldots,S_N$. Construct $A_i$ as the set of rows that intersects $S_i \cap \mathcal{\tilde{P}}$ for $i=1, \ldots,N$, see Figure \ref{fig:no_rectangle}. 

    Set $k=t+1$. 
    
    \item[\texttt{\underline{Step 2}}.] (Investigating empty squares (holes) in $\mathscr{R}_{a,b}$). 
 Consider the $k$-th column and look for empty unit squares $Q(x) \not \subset \mathcal{\tilde{P}}$ in this column starting from the top to the bottom.  
    If there is at least one empty unit square $Q(x) \subset A_i \cap \mathscr{R}_{a,b}$ for some $i \in \{1,\ldots,N \}$ such that $ Q(x) \not \subset \mathcal{ \tilde P}$, see the left hand side in Figure \ref{fig:no_rectangle}, then go to Step 3.     
    If the $k$-th column has all unit squares present, then repeat Step 2 with $k=k+1$ until $k\leq b$.
    If no columns have empty squares go to Step 4.
    
    \item[\texttt{\underline{Step 3}}.] (Construction of a polyomino $\mathcal{P}'$ with smaller nonlocal perimeter in case of empty squares). \\
    Let $Q(x_1),\ldots, Q(x_m) \subset \bigcup_{i=1}^N A_i \cap \mathscr{R}_{a,b}$, $m \in \mathbb{N}$, be the empty squares in the $k$-th column such that \\
    $Q(x_1), \ldots , Q(x_m) \not \subset \mathcal{\tilde P}$, see Figure \ref{fig:no_rectangle}. 
    Denote by $W$ the union of all rows containing the empty squares $Q(x_1), \ldots , Q(x_m)$. 
    Construct a polyomino $\mathcal{P}'$ from $\mathcal{\tilde P}$ as follows.
    We move $m$ unit squares of $S_1, \ldots , S_N$ belonging to $W$ horizontally to fill the empty squares  $Q(x_1), \ldots, Q(x_m)$ respectively. See the left and central figure in Figure \ref{fig:no_rectangle}. Observe that, by construction, the lengths of the vertical strips in the k-th column in $\mathcal{P}'$ are equal to or greater than the lengths of the vertical strips in the first column of $\mathcal{\tilde P}$. For this reason, we obtain  by Lemma \ref{lem:reduction_strips} that
    \[
     \begin{cases}
  Per^V_{\lambda}(\mathcal{P'}) -     Per^V_{\lambda}(\mathcal{\tilde{P}})= 0 & \text{ if all strips have the same length}, \\
  Per^V_{\lambda}(\mathcal{P'}) -     Per^V_{\lambda}(\mathcal{\tilde{P}})> 0 & \text{ if the length of at least one strip in } \mathcal{P}' \text{ is larger than the one in } \mathcal{\tilde{P}}. 
    \end{cases}
    \]
    We analyze the rows of $W$ to compute the horizontal contribution.  
\begin{itemize}
    \item[3.1)] If there are at least two horizontal strips in one of these rows, then the horizontal nonlocal perimeter contribution is decreased along this row by applying Lemma \ref{reduction_distance}.  Thus, we have $\mathcal{P}'$ such that $Per_{\lambda}(\mathcal{P}')<Per_{\lambda}(\mathcal{\tilde P})$ and the procedure stops.
    
    \item[3.2)] If there is only one horizontal strip in each of these rows, then the horizontal nonlocal perimeter contribution is unchanged, i.e.
$Per^H_{\lambda}(\mathcal{P}') =  Per^H_{\lambda}(\mathcal{\tilde{P}})$. If the nonlocal perimeter is unchanged, then return to  Step 1, setting $\mathcal{\tilde{P}} = \mathcal{P}'$. Otherwise, we have $\mathcal{P}'$ such that $Per_{\lambda}(\mathcal{P}')<Per_{\lambda}(\mathcal{\tilde P})$ and the procedure stops.
\end{itemize}  

    \item[\texttt{\underline{Step 4}}.] (Rotation of $\mathcal{ \tilde P}$ in case of no empty squares). If $t>1$ go to Step 5. Note that in this case all the horizontal strips in $ A_i$ have the same length equal to $b-t+1$ for $i=1,\dots,N$. Then, starting from the north side of $\mathscr{R}_{a,b}$ denote the horizontal strips in the first row of $\mathscr{R}_{a,b}$ by $S'_{1}, \ldots , S'_{J}$, $J \in \mathbb{N}$, and define the columns $B_i$ that intersect $S'_{i}$ for each $i=1,\ldots,J$; see Figure \ref{fig:rotation}.
    If all the vertical strips in $B_i$ have the same length equal to $a$ for $i=1,\dots,J$, see the left hand side in Figure \ref{fig:chessboard_convex}, then go to Step 5. Otherwise, rotate the polyomino $\mathcal{\tilde{P}}$ by $-\pi/2$ and call it $\mathcal{P}'$ and repeat the same procedure from Step 1 setting $\mathcal{\tilde P}=\mathcal{P'} $. 

        \begin{figure}[htb!]
    \begin{center}
    \includegraphics[scale=0.3]{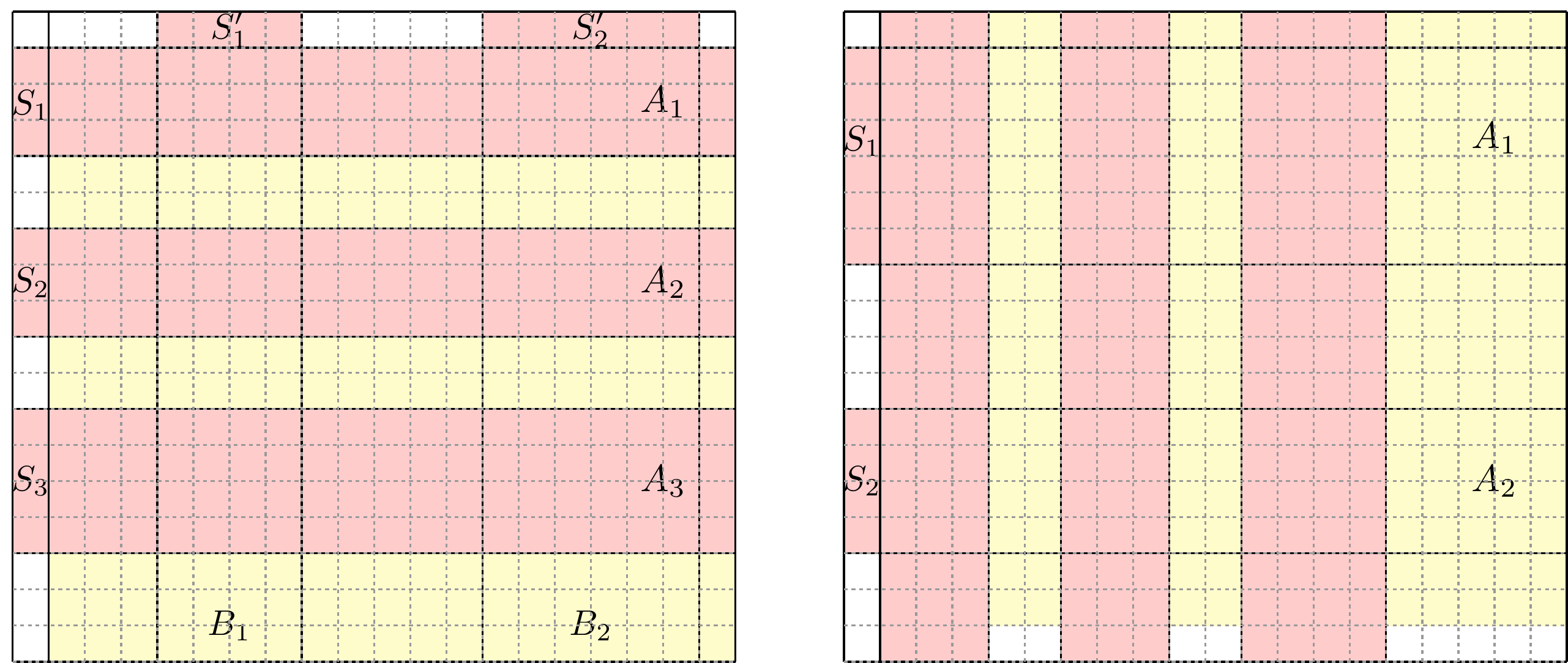} 
    \caption{On the left, an example of polyomino $\mathcal{P}$ with all horizontal strips of $A_i$ having the same length, for $i=1,\ldots,N$. On the right, the polyomino $\mathcal{P'}$ obtained from $\mathcal{P}$ by a rotation of $-\pi/2$.}
    \label{fig:rotation}
    \end{center}
    \end{figure}

    \item[\texttt{\underline{Step 5}}.] (Reducing the chessboard of $\mathcal{\tilde{P}}$ or \emph{convexification}). The polyomino $\mathcal{ \tilde P}$ is such that the horizontal strips in $A_i$ intersect the east side of $\mathscr{R}_{a,b}$ for every $i=1, \ldots ,N$ and the vertical strips in $B_j$ intersect the south side of $\mathscr{R}_{a,b}$ for $j=1, \ldots ,J$, see the left figure in Figure \ref{fig:chessboard_convex}. 
    \begin{figure}[htb!]
    \begin{center}
    \includegraphics[scale=0.3]{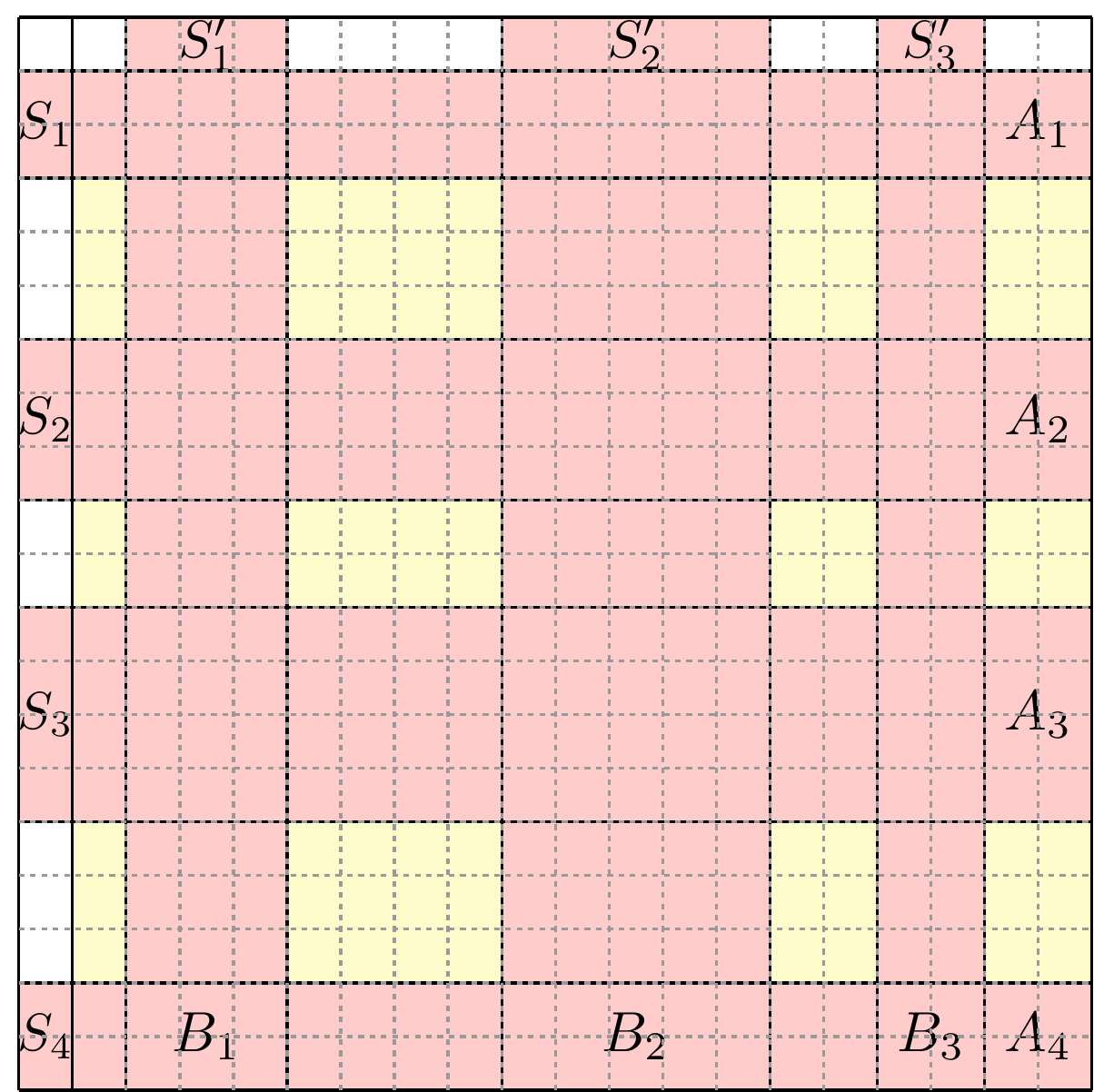} 
    \,\,\,\,\,\,\,\,\,\, 
    \includegraphics[scale=0.3]{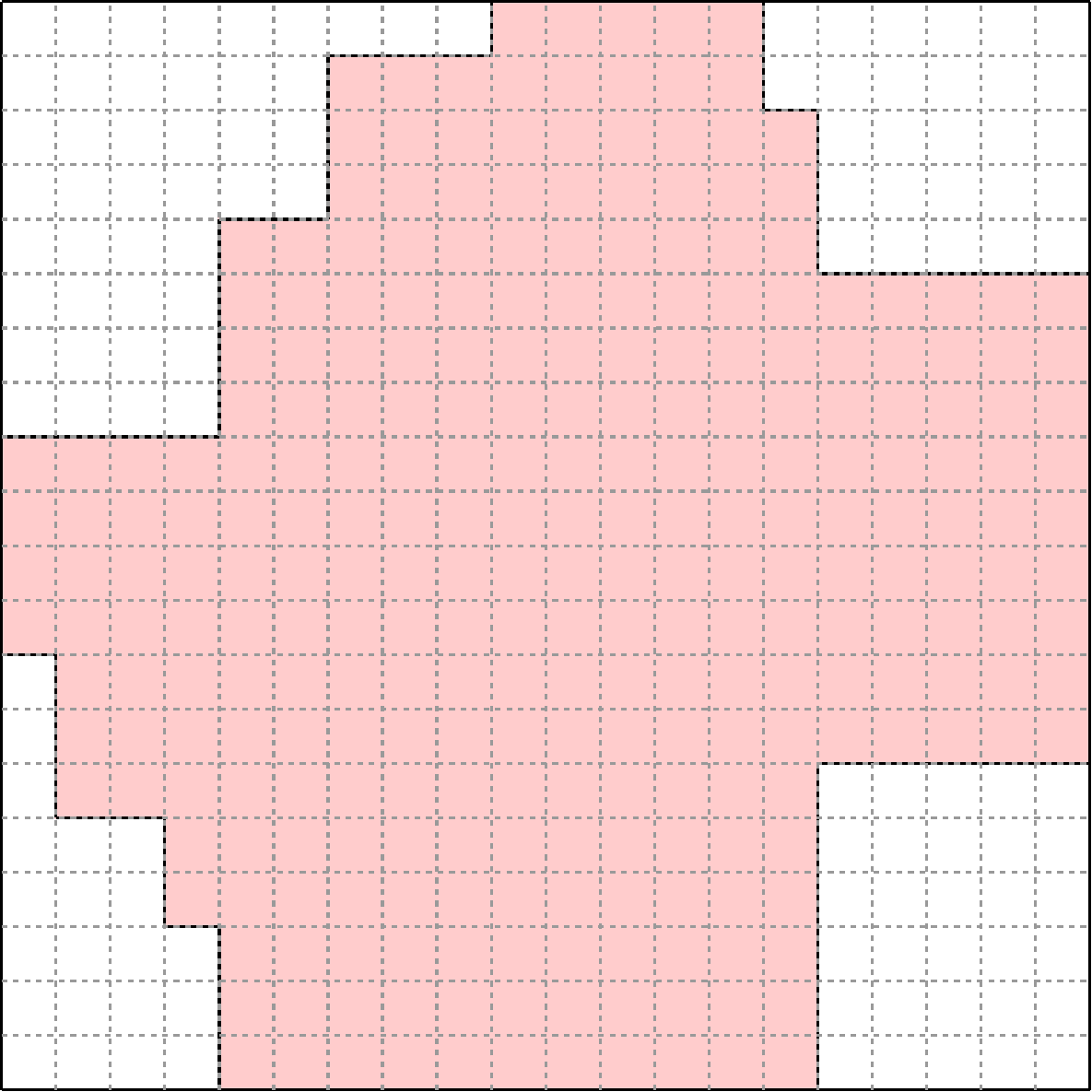}
    \caption{On the left, an example of polyomino $\mathcal{P}$ with  vertical strips $S_1,...,S_4$ and the corresponding sets $A_1,...,A_4$. The horizontal strips are denoted by $S'_1,...,S'_3$ and the corresponding sets by $B_1,...,B_3$. In red there are unit squares of $\mathcal{P}$. In white the unit squares which are not in $\mathcal{P}$. In yellow we see  the unit squares in the rectangle that are not analyzed yet and can belong to $\mathcal{P}$ or not.
    On the right we have an example of a cross-convex polyomino.}
    \label{fig:chessboard_convex}
    \end{center}
    \end{figure}
    \begin{figure}[htb]
     \begin{center}
    \includegraphics[scale=0.3]{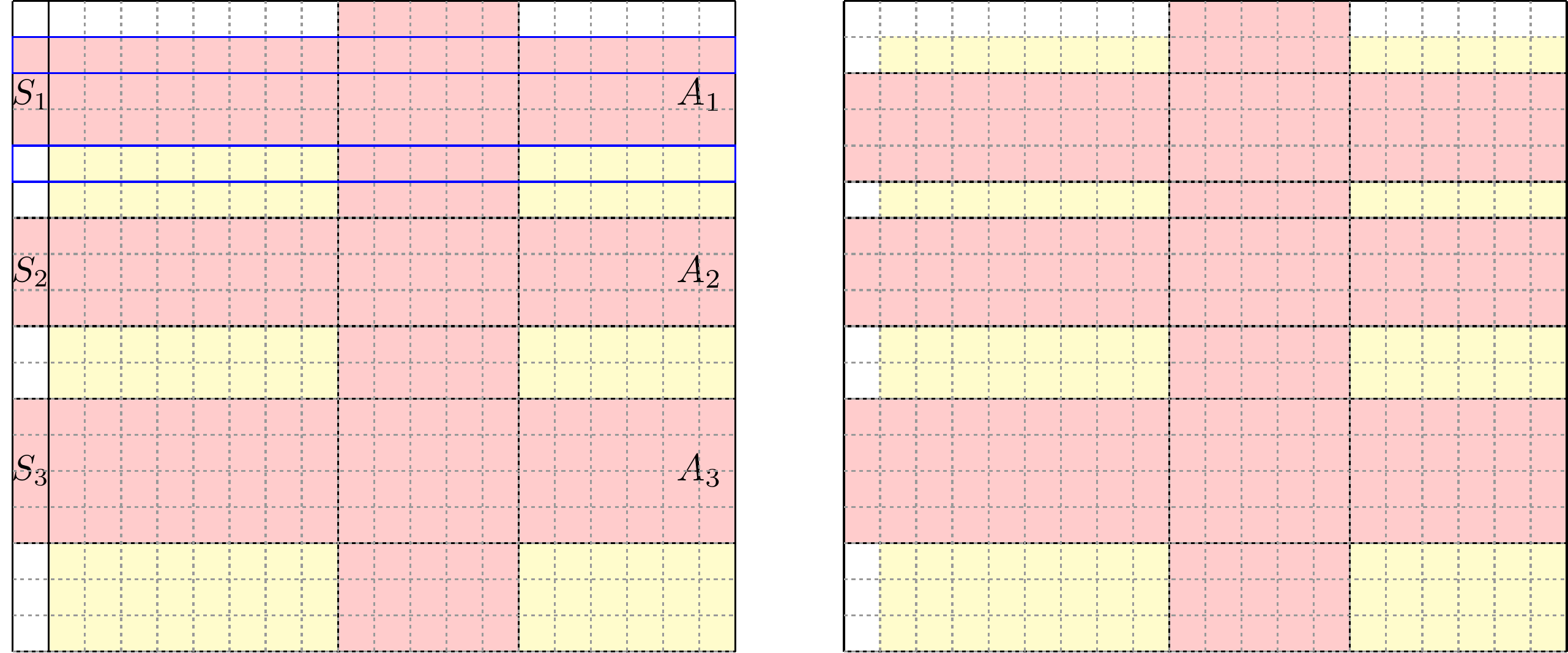} 
    \caption{On the left, there is an example of polyomino $\mathcal{P}$ with $N \neq 1$ and $J=1$. On the right, the polyomino $\mathcal{P}'$ is obtained from $\mathcal{P}$ as described in case (5.1). In particular, in red the unit squares of the polyominos are shown, in white the unit squares outside of the polyominos, and in yellow the unit squares of $\mathscr{R}_{a,b}$ that can belong to the polyominos or not. We highlight in blue, the rows that we exchange when we construct $\mathcal{P}'$ from $\mathcal{P}$.}
    \label{fig:chessboard_nocross}
    \end{center}
    \end{figure}
    We distinguish two cases:
    \begin{itemize}
        \item[5.1)] $N \neq 1$ or $J \neq 1$ or both $N,J \neq 1$.

        Assume without loss of generality that $N \neq 1$. Then, we construct $\mathcal{P}'$ by exchanging the first horizontal strip in $A_1$ intersecting $S_1$ with the horizontal strip in the first row that follows the set $A_1$, see Figure \ref{fig:chessboard_nocross}. 
     By Lemma \ref{reduction_distance}, we conclude that the nonlocal perimeter of $\mathcal{P}'$ is strictly smaller than that of $\mathcal{\tilde{P}}$, and the procedure stops.
    
    \item[5.2)] $N = 1$ and $J=1$. 
    We distinguish two cases:
    \begin{itemize}
    \item[5.2.a)] There is at least a column (or a row) of $\mathscr{R}_{a,b}$ that contains at least two vertical (or horizontal) strips, see the left figure in Figure \ref{fig:caseb}. Go to Step 1 with $t:=t+1$.
    
    \item[5.2.b)] All columns of $\mathscr{R}_{a,b}$ contain only one horizontal strip and all rows of $\mathscr{R}_{a,b}$ contain only one vertical strip, see the right figure in Figure \ref{fig:chessboard_convex}. By assumption and by the properties of the algorithm, $\mathcal{\tilde P}$ is not in $\mathscr{M}_n^{\text{ext}}$. However, in this case the polyomino $\mathcal{\tilde P}$ is cross-convex and then we conclude that it cannot be a minimum by applying \texttt{\underline{ALGORITHM FOR CROSS-CONVEX $\mathcal{P}$}} with starting polyomino $\mathcal{\tilde P}$, i.e.,  $\mathcal{C}=\mathcal{\tilde P}$.
    
    \end{itemize}
    \end{itemize}
\end{itemize}

\begin{figure}[htb!]
\begin{center}
    \includegraphics[scale=0.3]{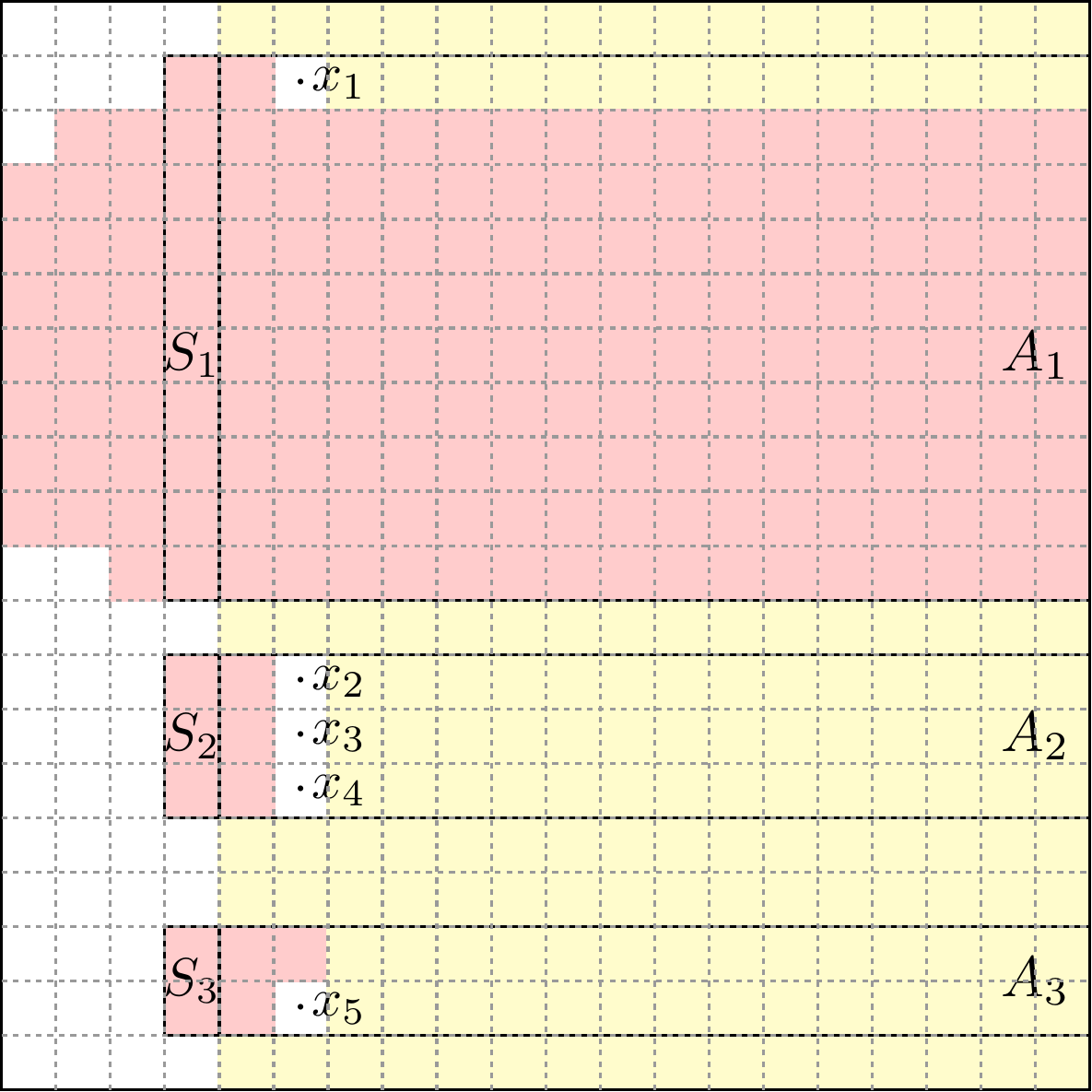} 
    \,\,\,\,\,\,\,\,\,\,
    \includegraphics[scale=0.3]{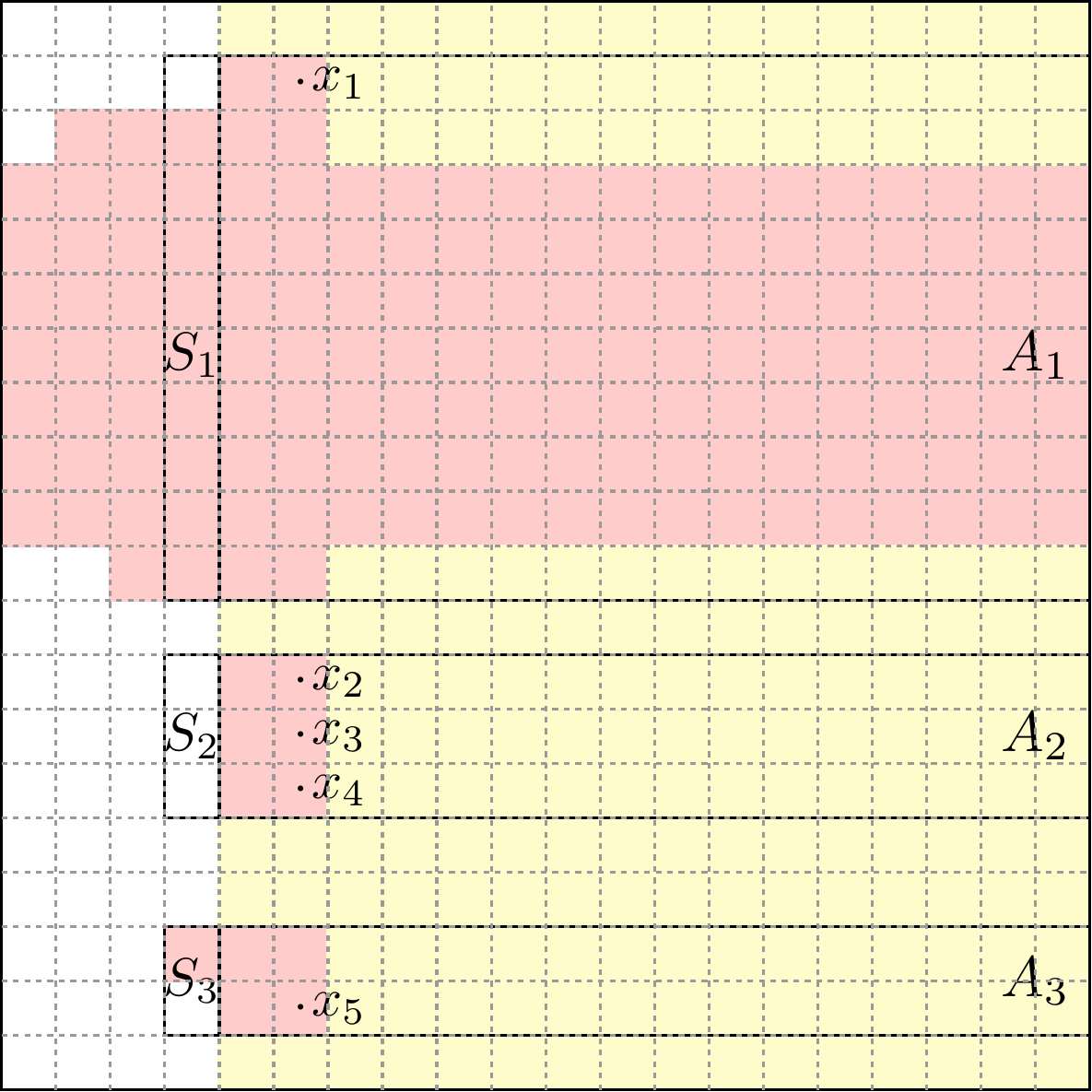} 
    \caption{On the left, there is an example of polyomino $\mathcal{P}$ as described in case (5.2.a). On the right, the polyomino $\mathcal{P}'$ obtained from $\mathcal{P}$ by applying the algorithm. }
 \label{fig:caseb}
\end{center}
    \end{figure}

The following algorithm
generalizes, in the nonlocal perimeter context, the transformations of polyominoes used in the proof of Theorem~2.2 in \cite{alonso1996three}. 
More specifically, for any cross-convex polyomino $\mathcal{C} \in \mathcal{M}_n \setminus \mathscr{M}_n^{\text{ext}}$, the algorithm finds a polyomino with the same area but with strictly smaller nonlocal perimeter, so that $\mathcal{C} \not \in \text{argmin}_{\mathcal{P} \in \mathcal{M}_{n}} \{Per_{\lambda}(\mathcal{P}) \}$.

\medskip
\texttt{\large\underline{ALGORITHM FOR CROSS-CONVEX $\mathcal{P}$}}\\
Given a cross-convex polyomino $\mathcal{C}$ not in $\mathscr{M}_n^{\text{ext}}$, the following algorithm finds first two polyominos with the same nonlocal perimeter of $\mathcal{C}$ but with different shape, and then it gives a polyomino with strictly smaller perimeter.
Before we start, we observe that if $\mathcal{C}$ is a cross-convex polyomino not in $\mathscr{M}_n^{\text{ext}}$, then $\mathcal{C}$ has at least two vertical (or horizontal) strips which are smaller than the others.
 
\begin{itemize}
 \item[\texttt{\underline{Step 0}}.] (Initialization). Set $\mathcal{D}:=\mathcal{C}$. Let $\mathscr{R}_{m_v, m_h}$ be the smallest rectangle containing $\mathcal{D}$ and let $m_v$, $m_h$ be the lengths of its columns and its rows respectively. Set $k:=1$ and $t:=1$.
 \item[\texttt{\underline{Step 1}}.] (Construction of a polyomino $\mathcal{D}'$ with the same nonlocal perimeter). 
We consider the vertical strip $S_k$ in the $k$-th column of $\mathcal{D}$. 
Let $A_k$ be the set of the rows that intersect $S_k$. 
If the horizontal strips in $A_k$ have not the same length, then there exist some empty squares $Q(x_1), \ldots, Q(x_N) \in A_k \cap \mathscr{R}_{m_v,m_h}$, $N \geq 1$. 
We consider the unit squares of $S_k$ in the same rows of the squares $Q(x_1), \ldots , Q(x_N)$ and we move each of them from its position to fill the empty squares $Q(x_i)$, for $i=1, \ldots, N$.
Set $k:=k+1$ and go to Step 2.

\item[\texttt{\underline{Step 2}}.] If $k\leq m_v$, then return to Step 1. Otherwise, call the obtained polyomino $\mathcal{D}'$ and go to Step 3. See the second panel in the Figure \ref{fig:convex3} for an example of $\mathcal{D}'$.

\item[\texttt{\underline{Step 3}}.](Construction of a polyomino $\mathcal{D}''$ with the same nonlocal perimeter).
We consider the $t$-th horizontal strip $S'_t$ in the $t$-th row of $\mathcal{D'}$. 
Let $B_t$ be the set of the columns that intersect $S'_t$. 
If the vertical strips in $B_t$ have not the same length, then there exist some empty squares $Q(x_1'), \ldots , Q(x'_{N'}) \subset B_t \cap \mathscr{R}_{m_v,m_h}$, $N' \geq 1$. 
We consider the unit squares of $S'_t$ in the same columns of the squares $Q(x_1'), \ldots , Q(x'_{N'})$ and we move each of them from its position to $Q(x'_i)$ for $i=1,\ldots ,N'$.
Set $t:=t+1$ and go to Step 4.

\item[\texttt{\underline{Step 4}}.] If $t\leq m_h$, then return to Step 3. Otherwise, call the obtained polyomino $\mathcal{D}''$ and go to Step 5. See the third panel in the Figure \ref{fig:convex3} for an example of $\mathcal{D}''$.

    \begin{figure}[htb!]
\begin{center}
    \includegraphics[scale=0.23]{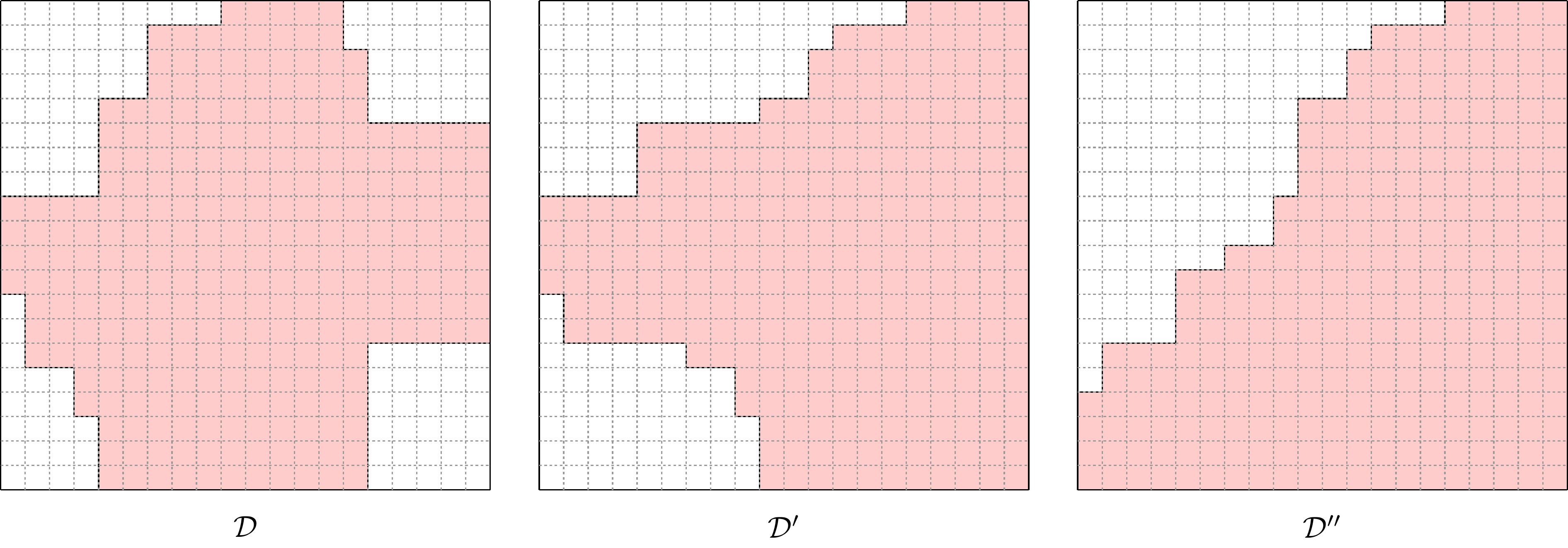} 
    \caption{On the left, an example of convex polyomino $\mathcal{D}$. In the center and on the right, the polyominos $\mathcal{D}'$ and $\mathcal{D}''$ obtained by applying the described procedure to $\mathcal{D}$. }
 \label{fig:convex3}
\end{center}
    \end{figure}
    \begin{figure}[htb]
\begin{center}
    \includegraphics[scale=0.3]{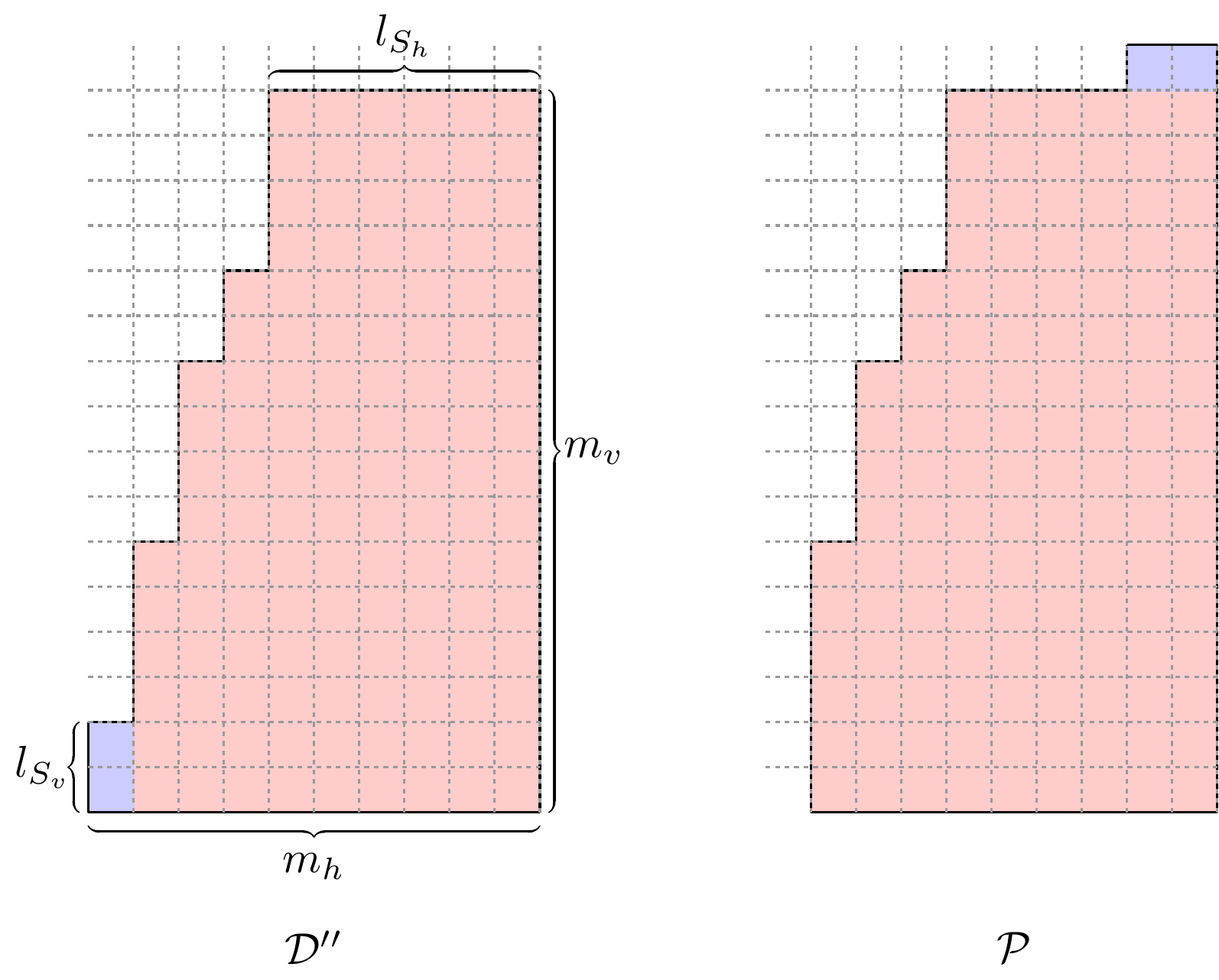} 
    \,\,\,\,\,\,\,\,\,\,\,\,\,
    \includegraphics[scale=0.3]{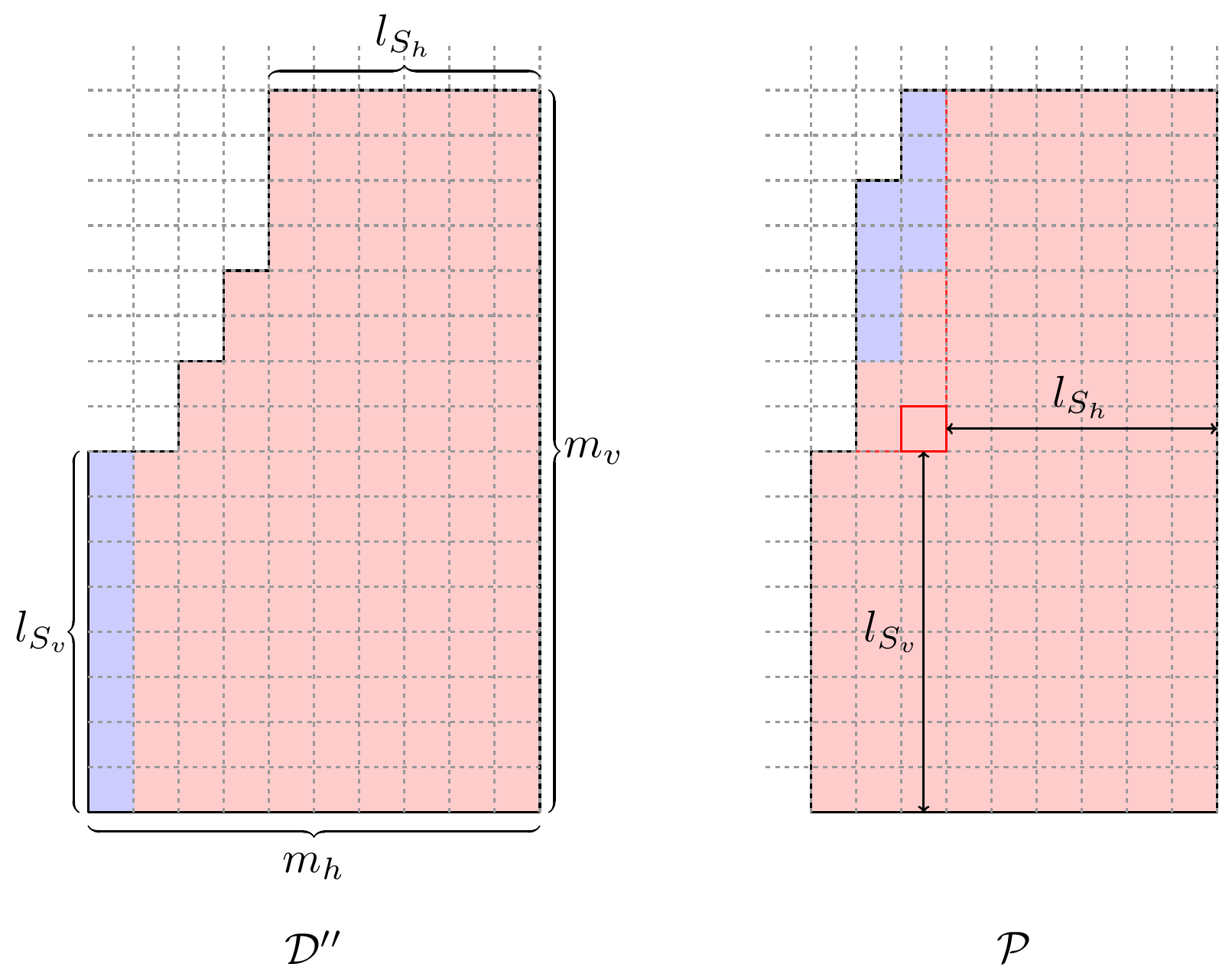} 
    \caption{On the left, an example of a polyomino $\mathcal{D''}$ with $l_{S_v} \leq l_{S_h}$ and the polyomino $\mathcal{P}$ obtained from $\mathcal{D}''$ with smaller nonlocal perimeter. On the right, an example of a polyomino $\mathcal{D''}$ with $l_{S_h} < l_{S_v}$ and the polyomino $\mathcal{P}$ obtained from $\mathcal{D''}$ with smaller perimeter. In the last picture, the red square indicates the position with minimal lengths $l_{S_h}, l_{S_v}$ of the strips where it could be possible to move the unit squares of the vertical strip with length $l_{S_v}$ of $\mathcal{D''}$. }
 \label{fig:convex_i_ii}
\end{center}
    \end{figure}
\item[\texttt{\underline{Step 5}}.] 
Suppose without loss of generality that $m_v \leq m_h$ and let $l_{S_h}, l_{S_v}$ be the lengths of the first horizontal and vertical strips of $\mathcal{D}''$. 
If $l_{S_h} \geq l_{S_v}$ go to Step 6, otherwise go to Step 7.

\item[\texttt{\underline{Step 6}}.] (Construction 1 of a polyomino $\mathcal{P}$ with smaller nonlocal perimeter).
We remove the first vertical strip from $\mathcal{D}''$ and we attach a horizontal strip of length $l_{S_v}$ on the first horizontal strip of $\mathcal{D}''$, see the left figure of Figure \ref{fig:convex_i_ii}. 
In this way, we obtain a polyomino $\mathcal{P}$ belonging to $\mathcal{M}_n$ with smaller nonlocal perimeter than that of $\mathcal{D}$. Indeed, 
\begin{align}
    Per_{\lambda}(\mathcal{P})& = Per_{\lambda}(\mathcal{D})-\sum_{i=1}^{l_{S_v}} \zeta(\lambda,i)- l_{S_v}  \zeta(\lambda, m_v) + \sum_{i=1}^{l_{S_v}} \zeta(\lambda, i) + l_{S_v} \zeta(\lambda, m_h+1) \notag \\
    &\leq Per_{\lambda}(\mathcal{D})- l_{S_v}\sum_{j= m_v}^{m_h} \frac{1}{j^\lambda} \leq Per_{\lambda}(\mathcal{D})- \frac{l_{S_v}}{m_v^\lambda}< Per_{\lambda}(\mathcal{D}).
\end{align}

\item[\texttt{\underline{Step 7}}.] (Construction 2 of a polyomino $\mathcal{P}$ with smaller nonlocal perimeter). We remove the first horizontal strip from $\mathcal{D}''$ and we attach  $l_{S_v}$ unit squares as in the right figure of Figure \ref{fig:convex_i_ii}. In particular, we order the columns according to the distance from the horizontal strip $S_h$ and we fill them starting from the first one until the length of this vertical strip is $m_v$.
In this way, we obtain a polyomino $\mathcal{P}$ belonging to $\mathcal{M}_n$ with smaller nonlocal perimeter than that of $\mathcal{D}$. Indeed
    \begin{align}
    Per_{\lambda}(\mathcal{P})& \leq Per_{\lambda}(\mathcal{D})-\sum_{i=1}^{l_{S_h}} \zeta(\lambda, i) -  l_{S_h} \zeta(\lambda, m_h) + l_{S_h} \zeta(\lambda, l_{S_v}+1) +l_{S_h} \zeta(\lambda, l_{S_h}+1),
\end{align}
where the inequality is obtained by estimating the value of the lengths with the smallest value that there can be, see the red square in the right panel of Figure \ref{fig:convex_i_ii}. Moreover, recalling $l_{S_h} \geq l_{S_v}+1$ and by using \eqref{eq:forA=1}, we obtain 
    \begin{align}
    Per_{\lambda}(\mathcal{P})& 
     \leq Per_{\lambda}(\mathcal{D})- l_{S_h} \zeta(\lambda)+\sum_{j=1}^{l_{S_h}-1} \frac{l_{S_h}-j}{j^\lambda}
    +l_{S_h} \left (
\zeta(\lambda, l_{S_h}+2) + \zeta(\lambda, l_{S_h}+1)     - \zeta(\lambda, m_h) \right ) \notag \\
    & < Per_{\lambda}(\mathcal{D})-\sum_{j=1}^{l_{S_h}} \frac{1}{j^{\lambda-1}}
    +l_{S_h} \zeta(\lambda, l_{S_h}+2) < Per_{\lambda}(\mathcal{D}),
\end{align}
where the last inequality follows by the positivity of $f(x, \lambda)$ defined by
\begin{align}
    f(x, \lambda):=\sum_{j=1}^{x} \frac{1}{j^{\lambda-1}}
    - x \zeta(\lambda, x+2).
\end{align}
Indeed, it is easy to see that $x\mapsto f(x,\lambda)$ is increasing in $x$. For $x=1, \lambda > 1.8$ we have that
\begin{align}
    f(1,\lambda)=1- \zeta(\lambda, 3)=2+\frac{1}{2^{\lambda}} -\zeta(\lambda) >0,
\end{align} 
which proves the claim.
\end{itemize}

\subsection{Proof of Proposition~\ref{prop:rectangle_square}}
\label{proof:prop3.2}

\noindent
Assume $n=l^2=a\cdot b$. The nonlocal perimeter of the square (resp. the rectangle) can be written as sum of the contributions along all vertical and horizontal strips. By Equation \eqref{eq:hor+ver}, we have: 
\begin{align}\label{eq:square_rectangle_GAP}
   &Per_{\lambda}(\mathscr{Q}_l)=4l \sum_{i=1}^l \zeta(\lambda, i)  \qquad \text{ and } \qquad
   Per_{\lambda}(\mathscr{R}_{a,b})=2a \sum_{i=1}^b \zeta(\lambda, i)  +2b \sum_{i=1}^a \zeta(\lambda, i),  
\end{align}
where $2l$ is the number of vertical and horizontal strips of the square, and $a$ (resp. $b$) is the number of the horizontal (resp. vertical) strips of $\mathscr{R}_{a,b}$. 
By Equations \eqref{eq:forA=1} and \eqref{eq:square_rectangle_GAP}, we get:
\begin{align}
  &\frac{Per_{\lambda}(\mathscr{Q}_l)}{2}=2l^2\zeta(\lambda) - 2l\sum_{k=1}^{l-1} \frac{l-k}{k^\lambda}, \label{eq:semiper_square}\\
  &\frac{Per_{\lambda}(\mathscr{R}_{a,b})}{2}=2ab\zeta(\lambda) - a\sum_{k=1}^{b-1} \frac{b-k}{k^\lambda} - b\sum_{k=1}^{a-1} \frac{a-k}{k^\lambda}. \label{eq:semiper_rectangle}
\end{align}
Thus, recalling $l^2=ab$, we obtain:
\begin{align}
\Delta := \frac{Per_{\lambda}(\mathscr{R}_{a,b})-Per_{\lambda}(\mathscr{Q}_l)}{2}&=2l \sum_{k=1}^{l-1} \frac{l-k}{k^{\lambda}} -a \sum_{k=1}^{b-1} \frac{b-k}{k^{\lambda}} -b \sum_{k=1}^{a-1} \frac{a-k}{k^{\lambda}}. \label{eq:GAP_difference_QRAB} 
\end{align}
We will first treat the cases $l=2,3$. Note that for $l=2$ (resp. $l=3$) we necessarily have that $a=1$ and $b=4$ (resp. $b=9$). Let first consider $l=2$, then for $\lambda > 1.8$ we have:
\[
\Delta = 4 - \left( 3 + \frac{2}{2^{\lambda}} + \frac{1}{3^{\lambda}}\right) = 1-\frac{1}{2^{\lambda-1}} -\frac{1}{3^{\lambda-1}} > 0.
\]
When $l=3$ we have that for any $\lambda \geq 1.3$:
\[
\Delta = 6 \left (2+\frac{1}{2^{\lambda}} \right ) - \left ( 8 + \frac{7}{2^{\lambda}} + \frac{6}{3^{\lambda}} + \ldots + \frac{1}{8^{\lambda}}\right )>0.
\]

\medskip

In the remaining cases $l\geq 4$, we can write the difference $\Delta$ as the sum of two functions $F_1(a,l)$ and $F_2(a,l)$, i.e.: 
\begin{equation}\label{eq:delta}
\Delta = F_1(a,l) + F_2(a,l,b),
\end{equation}
where
\begin{align}
&F_1(a,l)=\Big (a+\frac{l^2}{a}-2l \Big )\sum_{k=1}^{a-1} \frac{1}{k^{\lambda-1}}-2(l-a)\sum_{k=a}^{l-1} \frac{1}{k^{\lambda-1}} +\frac{l^2-a^2}{a^\lambda}, \label{def:F1}\\
&F_2(a,l,b)=\sum_{k=a+1}^{l-1} \frac{l^2-ak}{k^{\lambda}}-\sum_{k=l}^{b-1}\frac{l^2-ak}{k^{\lambda}} .\label{def:F2}
\end{align}
We then conclude the claim once can we show that one function is strictly positive and one non-negative.

\noindent
\textbf{Step 1:}
We will first show that $l \mapsto F_1(a,l)$ is increasing in $l$ and in particular that $F_1(a,l)>0$. We write:
\begin{equation}\label{eq:rest}
\begin{split}
F_1(a,l+1) &=  F_1(a,l) + \left (\frac{2(l-a)+1}{a} \right ) \sum_{k=1}^{a-1} \frac{1}{k^{\lambda-1}} +\frac{2l+1}{a^{\lambda}}-\frac{2(l-a)}{l^{\lambda-1}} -2\sum_{k=a}^l \frac{1}{k^{\lambda-1}} \\
& := F_1(a,l) + Rest(a,l).
\end{split}
\end{equation}
The function is increasing in $l$, if the rest-term $Rest(a,l)>0$ for all $l \geq a + \sqrt{a} =: a^*$. The lower bound on $l$ can be deduced from the following argument. Suppose by contradiction that $l < \lfloor a^* \rfloor$. We set $l=a+k$ with $k \in \mathbb{N}$ such that $k<\sqrt{a}$. By assumption, we have $a, \, b, \, l \in \mathbb{N}$ such that $ab=l^2$ and $a+b=2l +\alpha$ with $\alpha \in \mathbb{N}$ by \cite{alonso1996three}. Thus, we obtain
$(a+k)^2=a(2(a+k)+\alpha-a)$, i.e., $k^2=\alpha a$ and this is a contradiction since indeed $a>k^2=\alpha a \geq a$. 

We compute further
\begin{equation*}
\begin{split}
Rest(a,l+1) -Rest(a,l) = \frac{2}{a}\sum_{k=1}^{a}\frac{1}{k^{\lambda-1}}+\frac{2(l-a)}{l^{\lambda-1}}-\frac{2(l-a)}{(l+1)^{\lambda-1}} - \frac{4}{(l+1)^{\lambda-1}}.
\end{split}
\end{equation*}
Moreover, we note that 
\begin{equation}\label{eq:exp1}
\begin{split}
\frac{2}{a}\sum_{k=1}^{a}\frac{1}{k^{\lambda-1}}+\frac{2(l-a)}{l^{\lambda-1}}-\frac{2(l-a)}{(l+1)^{\lambda-1}} - \frac{4}{(l+1)^{\lambda-1}} & > \frac{2}{a}\sum_{k=1}^{a}\frac{1}{k^{\lambda-1}} - \frac{4}{(l+1)^{\lambda-1}}.
\end{split}
\end{equation}
Let $a\in \{1,2,3,4\}$ we can prove by hand that the last expression is positive for all $\lambda \geq 1.8$ using $l \geq a+\sqrt{a}$. 
Then, suppose that $a\geq 5$. It follows that $l \geq a+\sqrt{a} \geq  a+2$. Thus, we can bound the r.h.s. of Equation \eqref{eq:exp1} by
\begin{equation*}
\begin{split}
\frac{2}{a}\sum_{k=1}^{a}\frac{1}{k^{\lambda-1}} - \frac{4}{(l+1)^{\lambda-1}} &>  \frac{2}{a}
\int_1^{a+1} \frac{1}{x^{\lambda-1}} dx - \frac{4}{(a + 3)^{\lambda-1}}
= \frac{2}{a(2-\lambda)}((a+1)^{2-\lambda} -1) - \frac{4}{(a + 3)^{\lambda-1}}.
\end{split}
\end{equation*}
Note that, assuming first that  $1.8 \leq \lambda < 2$, we can bound
\begin{equation}
((a+1)^{2-\lambda}-1)(a+ 3)^{\lambda-1} > 2a(2-\lambda).
\end{equation}
The r.h.s. and the l.h.s. of the inequality are increasing functions in $a$ and the slope of the l.h.s. is bigger than the slope of the r.h.s. The case $\lambda \geq 2$ follows from similar considerations. We have proven that $Rest(a,l)$ is an increasing function of $l$ for each fixed $a$ and $l \geq a + 3$. It remains to show that $Rest(a, a + 3) > 0$. We have that:

\begin{equation*}
\begin{split}
Rest(a, a + 3) &= \frac{7}{a} \sum_{k=1}^{a-1} \frac{1}{k^{\lambda-1}} +\frac{2a + 7}{a^{\lambda}}-\frac{6}{(a + 3)^{\lambda-1}} -2\sum_{k=a}^{a + 3} \frac{1}{k^{\lambda-1}} \\
& = \frac{7}{a} \sum_{k=1}^{a-1} \frac{1}{k^{\lambda-1}}+\frac{7}{a^{\lambda}} -\frac{2}{(a+1)^{\lambda}} -\frac{2}{(a+2)^{\lambda}}-\frac{2a}{(a+1)^{\lambda}}-\frac{2a}{(a+2)^{\lambda}} - \frac{8a}{(a+3)^{\lambda}} -\frac{24}{(a+3)^{\lambda}},
\end{split}
\end{equation*}
which is greater than
\begin{equation*}
\frac{7}{a} \sum_{k=1}^{a-1} \frac{1}{k^{\lambda-1}} - \frac{14 a}{(a+1)^{\lambda}} -\frac{21}{(a+3)^{\lambda}}.
\end{equation*}
Note that trivially $\frac{7}{a} > \frac{21}{(a+3)^{\lambda}}$, it remains to show that 
\begin{equation} \label{eq:LBF1}
\sum_{k=2}^{a-1} \frac{1}{k^{\lambda-1}} - \frac{2 a^2}{(a+1)^{\lambda}} > 0,
\end{equation}
which is equivalent to showing (for $\lambda < 2$)

\begin{equation*}
(a+1)^{\lambda}(a^{2-\lambda} -2^{2-\lambda}) > 2(2-\lambda) a^2,
\end{equation*}
or
\begin{equation*}
\left( \frac{2}{a} \right ) ^{2-\lambda} \geq 1-\left(\frac{a}{a+1} \right)^{\lambda} 2 (2-\lambda) 
\end{equation*}
which is true for $a\geq 4$ and $1.8 \leq \lambda < 2$. The other case for $\lambda \geq 2$ follows in a similar manner. We have proven that $F_1(a,l)$ is an increasing function in $l$ for all $l\geq a+2$. Let us prove that $F_1(a,l)>0$ by showing that $F_1(a,a+2)>0$. First, we compute by hand $F_1(1,3)$. 
\begin{align*}
    F_1(1,3) & = 4 - \frac{8}{2^{\lambda}}>0,
\end{align*}
since $\lambda>1.8$.
Next, considering first $\lambda<2$, we have that:

\begin{equation*}
\begin{split}
F_1(a,a+2) & = \frac{4}{a} \sum_{k=1}^{a-1}\frac{1}{k^{\lambda-1}} + \frac{4}{a^{\lambda}} - \frac{4}{(a+1)^{\lambda-1}} 
\geq  \frac{4}{a} \sum_{k=1}^{a-1}\frac{1}{k^{\lambda-1}} - \frac{4}{a^{\lambda-1}} 
\geq \frac{4}{a} \int_{1}^{a}\frac{1}{k^{\lambda-1}} dk+\frac{4}{a^{\lambda}}-\frac{4}{a^{\lambda-1}} \notag \\
&=\frac{4}{(2-\lambda)a} (a^{2-\lambda}-1)- \frac{4(a-1)}{a^{\lambda}} 
=\frac{4}{a} \left (\frac{a^{2-\lambda}-1}{2-\lambda} - \frac{a-1}{a^{\lambda-1}} \right ).
\end{split}
\end{equation*}
Thus, it remains to show that
\begin{align*}
    \frac{a^{2-\lambda}-1}{2-\lambda} - \frac{a-1}{a^{\lambda-1}}>0.
\end{align*}
Assuming $1.8 \leq \lambda < 2$, it is equivalent to require
\begin{align*}
  1- a^{\lambda-2}  > \left (1-\frac{1}{a} \right) (2-\lambda),
\end{align*}
which is true for $a\geq 2$. The other case $\lambda \geq 2$ follows in a similar way.

\noindent
\textbf{Step 2:} In a second step, we will prove that $a\mapsto F_2(a,l,\lfloor l^2/a\rfloor)$ is decreasing in $a$ and in particular that $F_2(a,l,\lfloor l^2/a\rfloor)\geq 0$. First note that
\begin{equation*}
F_2(a,l,\lfloor l^2/a \rfloor) \geq F^{Int}_2(a,l) :=  \int_{a}^{l-1} \frac{l^2-ak}{k^{\lambda}} d k -\int_{l}^{\lfloor l^2/a \rfloor} \frac{l^2-ak}{k^{\lambda}} d k 
\geq \int_{a}^{l-1} \frac{l^2-ak}{k^{\lambda}} d k -\int_{l}^{l^2/a } \frac{l^2-ak}{k^{\lambda}} d k. 
\end{equation*}
The map $a\mapsto \frac{l^2-ak}{k^{\lambda}}$ is decreasing in $a$. Moreover, if $a$ increases, then the area of integration shrinks faster in $a$ than the area of integration of the second integral $l/a(l-a)$. Therefore, the first term is dominant and the function is decreasing in $a$. Since $a\leq l-3$, it remains to show that $F^{Int}_2(l-3,l)\geq 0$. Note that $F^{Int}_2(l-3,l)$ is decreasing in $l$ by the same arguments as before. Taking $l$ to infinity we obtain that $F_2^{Int}(l-3,l)$ converges to 0 from above.

\subsection{Proof of Proposition \ref{prop:rectangle_strip_attached1}}
\label{prop:3.4}
First we note that, by using Proposition \ref{prop:no_rectangle}, we have that 
$$\text{argmin}_{\mathcal{P} \in \mathcal{M}_{n}} \{Per_{\lambda}(\mathcal{P}) \}\subset \text{argmin}_{\mathcal{P} \in \mathscr{M}_n^{\text{ext}}} \{Per_{\lambda}(\mathcal{P}) \}.$$ We distinguish two cases: in Case (i) we suppose that the square has no protuberance, i.e. $k_1=0$, instead in Case (ii) we assume $k_1 \neq 0$. 
\paragraph{Case (i).} Let $k_1=0$. We will prove that for $n=l^2$ the nonlocal perimeter of a square is smaller than the nonlocal perimeter of a rectangle with a $k_2$-protuberance and the same area $n = ab+k_2$. 
Suppose first that the protuberance is attached along the shorter side. Thus, $k_2 \leq a-1$.
By using Equations \eqref{eq:semiper_square} and \eqref{eq:semiper_rectangle}, we have:
\begin{equation*}
 Per_{\lambda}(\mathscr{R}^{k_2}_{a,b})- Per_{\lambda}(\mathscr{Q}_l)= (Per_{\lambda}(\mathscr{R}_{a,b})- Per_{\lambda}(\mathscr{Q}_l)) + 2{k_2} \zeta(\lambda, b+1) + 2\sum_{i=1}^{k_2} \zeta(\lambda, i).
\end{equation*}
From Equations \eqref{eq:GAP_difference_QRAB}, \eqref{def:F1} and \eqref{def:F2}, we obtain that the difference can be rewritten as:
\begin{equation*}
\begin{split}
 Per_{\lambda}(\mathscr{R}^{k_2}_{a,b})- Per_{\lambda}(\mathscr{Q}_l) &= 2F_1(a,l)+2 F_2(a,l,b)+\frac{2{k_2}}{a}\sum_{r=1}^{a-1}\frac{a-r}{r^{\lambda}}- 2\sum_{r=1}^{k_2-1}\frac{k_2-r}{r^\lambda}-\frac{2k_2}{b^\lambda}.
\end{split}
\end{equation*}

From the proof of Proposition \ref{prop:rectangle_square} we have $F_1(a,l)>0$. Moreover, we can prove that $F_2(a,l,b)\geq0$ with $b= \lfloor (l^2-k_2)/a \rfloor$ by arguing as in Step 2 of Proposition \ref{prop:rectangle_square}. Thus, we obtain
\begin{align}\label{eq:prop3.4.1}
 Per_{\lambda}(\mathscr{R}^{k_2}_{a,b})- Per_{\lambda}(\mathscr{Q}_l) & \geq  \underbrace{2F_1(a,l)+ 2F_2(a,l, \lfloor (l^2-k_2)/a \rfloor)}_{>0}
 +\underbrace{\frac{2{k_2}}{a}\sum_{r=1}^{a-1}\frac{a-r}{r^{\lambda}}- 2\sum_{r=1}^{k_2-1}\frac{k_2-r}{r^\lambda}-\frac{2k_2}{b^\lambda}}_{(*)}.
\end{align}
In order to show that $(*)>0$, note that the function
\begin{align*}
    x \to \frac{c}{x} \sum_{r=2}^{x-1} \frac{x-r}{r^\lambda}
\end{align*}
is increasing in $x$, where $c >0$ is some constant. Then, recalling $k_2 \leq a-1$ we obtain
\begin{align*}
    (*) \geq \frac{2}{k_2-1} \sum_{r=1}^{k_2-1}\frac{k_2-r}{r^\lambda}- \frac{2k_2}{b^\lambda} \geq 2- \frac{2k_2}{b^\lambda}>0,
\end{align*}
where the last inequality follows from $\lambda>1$.

Next, assume that the protuberance is attached along the longer side, thus $k_2 \leq b-1$ and call the corresponding polyomino $\overline{\mathscr{R}}^{k_2}_{a,b}$. Observe that, if $k_2 \leq a$, then the rectangle $\mathscr{R}^{k_2}_{a,b}$ with the protuberance $k_2$ attached along the shorter side has smaller nonlocal perimeter than $\overline{\mathscr{R}}^{k_2}_{a,b}$. Indeed,  we have
\begin{align*}
    Per_{\lambda}(\overline{\mathscr{R}}^{k_2}_{a,b})- Per_{\lambda}(\mathscr{R}^{k_2}_{a,b})= 2k_2 \left ( \zeta(\lambda, a+1) - \zeta(\lambda, b+1)\right ) = 2k_2 \sum_{r=a+1}^{b} \frac{1}{r^\lambda} >0,
\end{align*}
hence we conclude in this case
\[
    Per_{\lambda}(\overline{\mathscr{R}}^{k_2}_{a,b})- Per_{\lambda}(\mathscr{Q}_l) = \left ( Per_{\lambda}(\overline{\mathscr{R}}^{k_2}_{a,b})- Per_{\lambda}(\mathscr{R}^{k_2}_{a,b})\right ) + Per_{\lambda}(\mathscr{R}^{k_2}_{a,b}) - Per_{\lambda}(\mathscr{Q}_l) > 0.
\]
Suppose now that $a \leq k_2 \leq b-1$. By analogous arguments as in the bound in Equation \eqref{eq:prop3.4.1} and using Equations \eqref{eq:semiper_square} and \eqref{eq:semiper_rectangle}, we can again write 
\begin{equation*}
\begin{split}
 Per_{\lambda}(\overline{\mathscr{R}}^{k_2}_{a,b})- Per_{\lambda}(\mathscr{Q}_l) &= (Per_{\lambda}(\mathscr{R}_{a,b})- Per_{\lambda}(\mathscr{Q}_l)) + 2{k_2} \zeta(\lambda, a+1) +2\sum_{i=1}^{k_2}\zeta(\lambda, i) \\
 & = Per_{\lambda}(\mathscr{R}^{k_2}_{a,b})- Per_{\lambda}(\mathscr{Q}_l) + 2k_2 \sum_{r=a+1}^{b} \frac{1}{r^\lambda} >0.
\end{split}
\end{equation*}

Finally we conclude that for the first case, when $n=l^2$,
 $\text{argmin}_{\mathcal{P} \in \mathcal{M}_{n}} \{Per_{\lambda}(\mathcal{P}) \}= \{ \mathscr{Q}_l \}$.

\paragraph{Cases (ii).} Let $k_1 \neq 0$ and consider a square with a $k_1$-protuberance attached along one of the four sides $\mathscr{Q}_l^{k_1}$ and area $n=l^2+k_1$.
We observe that the classical perimeter of $\mathscr{Q}_l^{k_1}$ is $per(\mathscr{Q}_l^{k_1})=2(2l+1)$ and we will prove the following two results.
\begin{itemize}
\item[(ii.1)] Let $0 \leq k_2 \leq b-1$ and consider $\mathscr{R}_{a,b}^{k_2} \in \mathcal{M}_n$ with classical perimeter greater than the perimeter of $\mathscr{Q}_l^{k_1}$, i.e. $2(a+b+1)=2(2l+1+C)$ for some $C \geq 1$, $C \in \mathbb{N}$. Then we have that,
\[
Per_{\lambda}(\mathscr{R}_{a,b}^{k_2}) > Per_{\lambda} (\mathscr{Q}_l^{k_1}).
\]

\item[(ii.2)] Let $0 \leq k_2 \leq b-1$ and consider a rectangle $\overline{\mathscr{R}}_{a,b}^{k_2} \in \mathcal{M}_n$ with a protuberance attached along the longer side and classical perimeter greater than the perimeter of $\mathscr{Q}_l^{k_1}$, i.e. $2(a+b+1)=2(2l+1+C)$ for some $C \geq 1$, $C \in \mathbb{N}$. Then we have that,
\[
Per_{\lambda}(\overline{\mathscr{R}}_{a,b}^{k_2}) > Per_{\lambda} (\mathscr{Q}_l^{k_1}).
\]
\end{itemize}
Recall that for general $a,b,k \in \mathbb{N}$,
\begin{equation}\label{eq:rect_prot}
 Per_{\lambda}(\mathscr{R}_{a,b}^{k}) = 2a \sum_{i=1}^{b} \zeta(\lambda, i) + 2b \sum_{i=1}^a \zeta(\lambda, i) + 2k\zeta(\lambda, b+1) + 2\sum_{i=1}^{k} \zeta(\lambda, i).
\end{equation}
\subparagraph{Case (ii.1).} Suppose that $a<b$ and $k_2 \geq 0$. 

 The difference of perimeters can be written as
\begin{align}\label{eq:difference_square_rectangle_proof_k2}
   & Per_{\lambda}(\mathscr{R}_{a,b}^{k_2})-Per_{\lambda}(\mathscr{Q}_l^{k_1})\nonumber \\
    =&4ab\zeta(\lambda) - 2a\sum_{r=1}^{b-1} \frac{b-r}{r^\lambda} - 2b\sum_{r=1}^{a-1} \frac{a-r}{r^\lambda}
    +2\sum_{i=1}^{k_2} \zeta(\lambda, i) +2k_2 \zeta(\lambda, b+1) -4l^2\zeta(\lambda) 
    + 4l\sum_{r=1}^{l-1} \frac{l-r}{r^\lambda}   -2\sum_{i=1}^{k_1}\zeta(\lambda, i)\\
    &-2k_1 \zeta(\lambda, l+1). \nonumber
\end{align}

We will treat two cases: $k_1 - k_2 \leq a$ resp. $k_1-k_2>a$.

\begin{itemize}
    \item[(ii.1.a)] Suppose $k_1-k_2 \leq a$. We will prove that this condition is equivalent to require $l \geq a+\sqrt{a}$. Indeed, assume first $k_1-k_2 \leq 0$. 
    Without loss of generality suppose that $a<b$ and we write $a=l-\alpha$ and $b=l+\beta$ for some $\alpha, \beta \in \mathbb{N}$.  Thus, for some positive integer $C \geq 1$, we have the following relations between the area and the classical perimeter
   \begin{align}\label{condition_k1_k2}
         \begin{cases}
             a+b+1=2l+1+C  \\
             ab+k_2=l^2+k_1
         \end{cases}
         \iff 
         \begin{cases}
             \beta=\alpha+C  \\
             (l-\alpha)(l+\alpha+C)+k_2=l^2+k_1
         \end{cases}
         \iff 
         \begin{cases}
             \beta=\alpha+C \\
             -\alpha^2+C(l-\alpha)=k_1-k_2.
         \end{cases}
    \end{align}
    
In this case, we have
\begin{align}\label{lsqrta}
    a \geq k_1-k_2 \iff a \geq -\alpha^2+C(l-\alpha) \iff (l-a)^2 \geq (C-1) a \iff l \geq a+ \sqrt{(C-1)a}.
\end{align}
Moreover, we observe that the case $C=1$ cannot be achieved  assuming $k_1 \geq k_2$. Indeed, let us assume $C=1$, by \eqref{condition_k1_k2} we obtain $k_1-k_2=-\alpha^2 < 0$, since $\alpha \geq 1$. Then $k_1 < k_2$, and this is a contradiction since by assumption we have $k_1 \geq k_2$. Thus, by \eqref{lsqrta}, $l \geq a+\sqrt{a}$.

Using \eqref{eq:difference_square_rectangle_proof_k2}, we will write for the difference of nonlocal perimeters
    \begin{align*}
        Per_{\lambda}(\mathscr{R}_{a,b}^{k_2})&-Per_{\lambda}(\mathscr{Q}_l^{k_1}) \\
       &=2 \tilde{F}_1(a,l, k_1)+ 2 \left ( \sum_{r=a+1}^{l-1} \frac{l^2-ar}{r^\lambda} - \sum_{r=l}^{b-1} \frac{l^2+k_1-ar}{r^\lambda} \right )
        +\frac{k_2}{a} \sum_{r=1}^{a-1} \frac{a-r}{r^\lambda} 
    - \frac{k_2}{b^\lambda}-\sum_{r=1}^{k_2-1} \frac{k_2-r}{r^\lambda} \notag \\
    &\geq 2 \tilde{F}_1(a,l, k_1)+ 2\tilde{F}_2(a,l,k_1) + \underbrace{\frac{k_2}{a} \sum_{r=1}^{a-1} \frac{a-r}{r^\lambda} 
    - \frac{k_2}{b^\lambda}-\sum_{r=1}^{k_2-1} \frac{k_2-r}{r^\lambda}}_{\textnormal{(Eq.1.a)}}.
    \end{align*}
    where
    \begin{align}
    & \tilde{F}_1(a,l,k_1)= 
    F_1(a,l)-\frac{k_1}{a} \sum_{r=1}^{a-1} \frac{a-r}{r^\lambda} 
    + \frac{k_1}{l^\lambda}+\sum_{r=1}^{k_1-1} \frac{k_1-r}{r^\lambda}, \label{def:F_tilde1} \\
& \tilde{F}_2(a,l,k_1)= \sum_{r=a+1}^{l-1} \frac{l^2-ar}{r^{\lambda}}-\sum_{r=l}^{\lfloor\frac{l^2+k_1}{a}-1\rfloor} \frac{l^2+k_1-ar}{r^{\lambda}}. \label{def:F_tilde2}
    \end{align}

Indeed, it will be easy to see that (Eq.1.a) $>0$, since first of all the function 
\begin{align}\label{function_k2_part}
    x \to \sum_{r=3}^{\frac{k_2}{x}-1}\frac{k_2-rx}{r^\lambda}
\end{align}
is decreasing for $\lambda>1$. Thus, since $1 \leq k_2 <a$, we obtain
\begin{align}\label{positivity_lastparte_k2}
 (\textnormal{Eq.1.a} )
& \geq - \frac{k_2}{b^\lambda}-\frac{k_2}{a}+1- \frac{k_2}{2^{\lambda-1} a}+\frac{1}{2^{\lambda-1}} \nonumber\\
& \geq 1+\frac{1}{2^{\lambda-1}} - k_2 \left ( \frac{1}{k_2+1} + \frac{1}{(k_2+1)^\lambda} + \frac{1}{2^{\lambda-1} (k_2+1)} \right ) \notag \\
& \geq 1+\frac{1}{2^{\lambda-1}} - \left ( \frac{1}{2} + \frac{1}{2^\lambda} + \frac{1}{2^{\lambda}} \right )=\frac{1}{2}>0,
\end{align}
where we used $a,b \geq k_2+1$.
It remains to show that
\[
\tilde{F}_1(a,l, k_1)+  \tilde{F}_2(a,l,k_1) > 0.
\]
 First, we will show that $\tilde{F}_2(a,l,k_1)>0$ by proving that the function $k_1\mapsto \tilde{F}_2(a,l,k_1)$ is decreasing in $k_1$ and evaluating $\tilde{F}_2(a,l,l-1)$. Trivially, we can write
\begin{align}\label{F2tilde_decreasing}
    \tilde{F}_2(a,l,k_1+1)-\tilde{F}_2(a,l,k_1)&=
    \sum_{r=l}^{\lfloor\frac{l^2+k_1}{a}-1\rfloor} \frac{l^2+k_1-ar}{r^{\lambda}}
    -\sum_{r=l}^{\lfloor\frac{l^2+k_1+1}{a}-1\rfloor} \frac{l^2+k_1+1-ar}{r^{\lambda}} \notag \\
    &=-\sum_{r=\lfloor\frac{l^2+k_1}{a}\rfloor}^{\lfloor\frac{l^2+k_1+1}{a}-1\rfloor} \frac{l^2+k_1-ar}{r^{\lambda}}
    -\sum_{r=l}^{\lfloor\frac{l^2+k_1+1}{a}-1\rfloor} \frac{k_1}{r^{\lambda}} <0.
\end{align}
Then, we can consider the function $\tilde{F}_2(a,l,l-1)$, i.e., 
\begin{align}\label{F2tilde_positive}
   \tilde{F}_2(a,l,l-1) &= \sum_{r=a+1}^{l-1} \frac{l^2-ar}{r^{\lambda}}-\sum_{r=l}^{\lfloor\frac{l^2+l-1}{a}-1\rfloor} \frac{l^2+l-1-ar}{r^{\lambda}}
   =\sum_{r=a+1}^{l-1} \frac{l^2-ar}{r^{\lambda}}-\sum_{r=l-\lfloor\frac{l-1}{a}\rfloor}^{\lfloor\frac{l^2}{a}-1\rfloor} \frac{l^2-ar}{(r+\frac{l-1}{a})^{\lambda}} \notag \\
   & \geq \sum_{r=a+1}^{l-1} \frac{l^2-ar}{r^{\lambda}}-\sum_{r=l-\lfloor\frac{l-1}{a}\rfloor}^{\lfloor\frac{l^2}{a}-1\rfloor} \frac{l^2-ar}{r^{\lambda}} >0
\end{align}
where the last inequality follows from arguing as in the Step 2 of the proof of Proposition \ref{prop:rectangle_square}.

We will conclude the proof of the Case (ii.1.a) 
by showing that $\tilde{F}_1(a,l, k_1) >0$.
First, we  show that $l \mapsto \tilde{F}_1(a,l, k_1)$ is increasing in $l$. We write
\begin{equation*}
\begin{split}
\tilde{F}_1(a,l+1, k_1) &=  F_1(a,l+1) + \frac{k_1}{(l+1)^\lambda}-\frac{k_1}{a} \sum_{r=1}^{a-1} \frac{a-r}{r^\lambda} 
   +\sum_{r=1}^{k_1-1} \frac{k_1-r}{r^\lambda} \\
& = F_1(a,l) + Rest(a,l)  + \frac{k_1}{(l+1)^\lambda}-\frac{k_1}{a} \sum_{r=1}^{a-1} \frac{a-r}{r^\lambda} 
   +\sum_{r=1}^{k_1-1} \frac{k_1-r}{r^\lambda} \\
& =  \tilde{F}_1(a,l, k_1) - \frac{k_1}{l^\lambda}+ \frac{k_1}{(l+1)^\lambda} + Rest(a,l).
\end{split}
\end{equation*}
Let us call
\[
\overline{Rest}(a,l) = Rest(a,l)  - \frac{k_1}{l^\lambda}+ \frac{k_1}{(l+1)^\lambda}
\]
where $Rest(a,l)$ is defined in \eqref{eq:rest} and show that $\overline{Rest}(a,l)>0$. The function $l \mapsto \overline{Rest}(a,l)$ is increasing in $l$, if the term $\overline{Rest}(a,l)>0$ for all $l \geq a+\sqrt{a}$. Indeed,
\begin{align}
    \overline{Rest}(a,l+1)-\overline{Rest}(a,l) &=Rest(a,l+1)-Rest(a,l)- \frac{k_1}{(l+1)^\lambda}+ \frac{k_1}{(l+2)^\lambda} + \frac{k_1}{l^\lambda}- \frac{k_1}{(l+1)^\lambda} \notag \\
    & \geq k_1 \left ( \frac{1}{l^\lambda} + \frac{1}{(l+2)^\lambda}-\frac{2}{(l+1)^\lambda} \right )>0,
\end{align}
where we obtain the first inequality as in the proof of Proposition \ref{prop:rectangle_square}, since by assumption $l \geq a+\sqrt{a}$. 
    \item[(ii.1.b)] Assume $k_1-k_2 > a$. 
     We observe that if $k_2=0$, then $\mathscr{R}_{a,b}^{k_2}=\mathscr{R}_{a,b}$ and its perimeter is $2(a+b)$. Thus, for some positive integer $C \geq 2$, we have

\begin{align}\label{eq:condition_alpha_beta_k}
        \begin{cases}
            a+b=2l+1+C \\
            ab=l^2+k_1
        \end{cases}
        \iff 
       \begin{cases}
            \beta=\alpha+1+C \\
            (l-\alpha)(l+\alpha+1+C)=l^2+k_1
        \end{cases}
        \iff 
        \begin{cases}
            \beta=\alpha+1+C \\
            -\alpha^2+(1+C)(l-\alpha)=k_1.
        \end{cases}
   \end{align}

By using \eqref{eq:condition_alpha_beta_k} and \eqref{eq:difference_square_rectangle_proof_k2} with $k_2=0$, we obtain
\begin{align}\label{eq:difference_square_rectangle_proof_k}
   Per_{\lambda}(\mathscr{R}_{a,b})&-Per_{\lambda}(\mathscr{Q}_l^{k_1}) \nonumber\\
    =&4ab\zeta(\lambda) - 2a\sum_{r=1}^{b-1} \frac{b-r}{r^\lambda} - 2b\sum_{r=1}^{a-1} \frac{a-r}{r^\lambda}
    -4l^2\zeta(\lambda) + 4l\sum_{r=1}^{l-1} \frac{l-r}{r^\lambda}-2\sum_{i=1}^{k_1}\zeta(\lambda, i)-2k_1 \zeta(\lambda, l+1) \notag \\
    =& 4 \left (-\alpha^2+(1+C)(l-\alpha) \right )\zeta(\lambda) - 2(l-\alpha) \sum_{r=1}^{l+\alpha+C} \frac{l+\alpha+C+1-r}{r^\lambda} \notag \\
    &- 2(l+\alpha+C+1) \sum_{r=1}^{l-\alpha-1} \frac{l-\alpha-r}{r^\lambda}
    + 4l\sum_{r=1}^{l-1} \frac{l-r}{r^\lambda} -2\sum_{i=1}^{-\alpha^2+(1+C)(l-\alpha)}\sum_{r=1}^{\infty} \frac{1}{r^\lambda} \notag \\
    &-2(-\alpha^2+(1+C)(l-\alpha)) \zeta(\lambda, l+1)=2\Delta(l,\alpha,C,\lambda),
\end{align}
where
\begin{align}\label{def_DeltalCalpha}
&\Delta(l,\alpha,C,\lambda):= \nonumber\\
& \left (-\alpha^2+(1+C)(l-\alpha) \right ) \sum_{r=1}^{l} \frac{1}{r^\lambda} - (l-\alpha) \sum_{r=1}^{l+\alpha+C} \frac{l+\alpha+C+1-r}{r^\lambda} - (l+\alpha+C+1) \sum_{r=1}^{l-\alpha-1} \frac{l-\alpha-r}{r^\lambda}, \notag \\
    &+ 2l\sum_{r=1}^{l-1} \frac{l-r}{r^\lambda}
    +\sum_{r=1}^{-\alpha^2+(1+C)(l-\alpha)-1}\frac{-\alpha^2+(1+C)(l-\alpha)-r}{r^\lambda}.
\end{align}

   Let us now consider \eqref{condition_k1_k2} and \eqref{eq:difference_square_rectangle_proof_k2}, so that we can write further for general $k_2$:
\begin{align}\label{difference_nonlocal_deltatilde}
    &Per_{\lambda}(\mathscr{R}_{a,b}^{k_2})-Per_{\lambda}(\mathscr{Q}_l^{k_1})\\
    &= 2\Delta(l,\alpha,C-1,\lambda) 
-2 k_2 \sum_{r=l+1}^{l+\alpha+C} \frac{1}{r^\lambda}+ 2\left (\underbrace{\sum_{r=1}^{k_1-1} \frac{k_1-r}{r^\lambda}-\sum_{r=1}^{k_2-1} \frac{k_2-r}{r^\lambda}-\sum_{r=1}^{k_1-k_2-1} \frac{k_1-k_2-r}{r^\lambda}}_{\geq 0}
    \right ).
\end{align}
It is easy to see that the term in the bracket is strictly positive. 

Observe, that 
\begin{align}\label{dif_nonlocal_tildedelta}
     \Delta(l,\alpha,C-1,\lambda) - k_2 \sum_{r=l+1}^{l+\alpha+C} \frac{1}{r^\lambda}
    \geq \tilde \Delta (l, \alpha, C, \lambda),
\end{align}
where
\begin{align}\label{def:deltatilde}
\tilde \Delta (l, \alpha, C, \lambda):= \Delta(l,\alpha,C-1,\lambda) - (l-\alpha) \sum_{r=l+1}^{l+\alpha+C} \frac{1}{r^\lambda}.
\end{align}

In the sequel we will prove by induction that $\tilde \Delta (l, \alpha, C, \lambda)>0$ for each $C\geq 2$.

\paragraph{1st step.} We prove that $\tilde \Delta (l, \alpha, 2, \lambda) > 0$. Let us write
\begin{align}
    \tilde \Delta (l, \alpha, 2, \lambda)=&\Delta (l, \alpha, 1, \lambda)- (l-\alpha) \sum_{r=l+1}^{l+\alpha+2} \frac{1}{r^\lambda} \notag \\
    = &\left (-\alpha^2+2(l-\alpha) \right ) \sum_{r=1}^{l} \frac{1}{r^\lambda} - (l-\alpha) \sum_{r=1}^{l+\alpha+1} \frac{l+\alpha+2-r}{r^\lambda}- (l+\alpha+2) \sum_{r=1}^{l-\alpha-1} \frac{l-\alpha-r}{r^\lambda}  \notag \\
    &+ 2l\sum_{r=1}^{l-1} \frac{l-r}{r^\lambda}+\sum_{r=1}^{-\alpha^2+2(l-\alpha)-1}\frac{-\alpha^2+2(l-\alpha)-r}{r^\lambda}
    - (l-\alpha) \sum_{r=l+1}^{l+\alpha+2} \frac{1}{r^\lambda}.
\end{align}

First we will show that $\tilde \Delta (l, \alpha, 2, \lambda)$ is an increasing function of $l$ and then evaluate the expression at the smallest value of $l$. The difference is equal to
\begin{align}
   \tilde \Delta (l+1, \alpha, 2, \lambda)&-\tilde \Delta (l, \alpha, 2, \lambda) \nonumber \\
&=
    \Delta (l+1, \alpha, 1, \lambda)
    - \Delta (l, \alpha, 1, \lambda)
    - \sum_{r=l+2}^{l+\alpha+3} \frac{1}{r^\lambda}
    +(l-\alpha) \left ( \frac{1}{(l+1)^\lambda} -\frac{1}{(l+\alpha+2)^\lambda} \right ) \notag \\
    & \geq \Delta (l+1, \alpha, 1, \lambda)
    - \Delta (l, \alpha, 1, \lambda)
    - \sum_{r=l+2}^{l+\alpha+3} \frac{1}{r^\lambda} \label{eq:diff_DT_l}. 
\end{align}

Recalling that $k_1-k_2> a$, then by condition \eqref{condition_k1_k2}, we have $-\alpha^2+2(l-\alpha) > l-\alpha$. Thus, we obtain
\begin{align}
\eqref{eq:diff_DT_l} = & \frac{-\alpha^2+2(l-\alpha) }{(l+1)^\lambda} +\frac{1}{(-\alpha^2+2(l-\alpha)+1)^\lambda}
+2\sum_{r=l-\alpha+1}^{-\alpha^2+2(l-\alpha)}\frac{1}{r^\lambda} 
-\sum_{r=l-\alpha+1}^{l+\alpha+2} \frac{2l+3-r}{r^\lambda} \notag \\
  & + \sum_{r=l-\alpha+1}^{l} \frac{4l+4-r}{r^\lambda}
  - \sum_{r=l+2}^{l+\alpha+3} \frac{1}{r^\lambda} \notag \\
   =& \frac{-\alpha^2+2(l-\alpha) }{(l+1)^\lambda} +\frac{1}{(-\alpha^2+2(l-\alpha)+1)^\lambda}
+(2l+3)\sum_{r=l-\alpha+1}^{-\alpha^2+2(l-\alpha)}\frac{1}{r^\lambda} 
-\sum_{r=l+1}^{l+\alpha+2} \frac{2l+3-r}{r^\lambda}.  \label{eq:DT_l1} 
  \end{align}
  Then by reshuffling and by recalling that the assumption $k_1-k_2>a$ is equivalent to $l+1<-\alpha^2+2l-\alpha+1$, we have

  \begin{align}
 \eqref{eq:DT_l1} =& \frac{-\alpha^2+2(l-\alpha) }{(l+1)^\lambda} +\frac{1}{(-\alpha^2+2(l-\alpha)+1)^\lambda} 
+\sum_{r=l+1}^{l+\alpha+2} \frac{1}{r^{\lambda-1}} -2 \sum_{r=-\alpha^2+2l-\alpha+1}^{l+\alpha+2}\frac{1}{r^\lambda} \nonumber \\
& - (2l+1) \left ( \frac{1}{(l+\alpha+1)^\lambda}+\frac{1}{(l+\alpha+2)^{\lambda}} \right ) 
- \sum_{r=l+2}^{l+\alpha+3} \frac{1}{r^\lambda} \notag \\%% 
& +(2l+3) \underbrace{\left (\sum_{r=l-\alpha+1}^{-\alpha^2+2(l-\alpha)}\frac{1}{r^\lambda} - \sum_{r=l+1}^{-\alpha^2+2l-\alpha}\frac{1}{r^\lambda} \right)}_{>0} + (2l+1) \underbrace{\left (\sum_{r=-\alpha^2+2(l-\alpha)+1}^{l} \frac{1}{r^\lambda} - \sum_{r=-\alpha^2+2l-\alpha+1}^{l+\alpha} \frac{1}{r^\lambda}  \right )}_{>0} \nonumber\\
& \geq \frac{-\alpha^2+2(l-\alpha) }{(l+1)^\lambda} +\frac{1}{(-\alpha^2+2(l-\alpha)+1)^\lambda} 
+\sum_{r=l+1}^{l+\alpha+2} \frac{1}{r^{\lambda-1}} -2 \sum_{r=-\alpha^2+2l-\alpha+1}^{l+\alpha+2}\frac{1}{r^\lambda} \notag \\
& - (2l+1) \left ( \frac{1}{(l+\alpha+1)^\lambda}+\frac{1}{(l+\alpha+2)^{\lambda}} \right ) - \sum_{r=l+2}^{l+\alpha+3} \frac{1}{r^\lambda}, \label{eq:hallo}
\end{align}
where the inequality follows by noting that the sums in the brackets contain the same number of terms, and the first sum is greater than the second one in each bracket. We observe that the first two remaining terms are positive and we can lower bound as follows
\begin{align}\label{last_equation_1b}
 \eqref{eq:hallo}=&\sum_{r=l+1}^{l+\alpha+2} \frac{1}{r^{\lambda-1}} -2 \sum_{r=-\alpha^2+2l-\alpha+1}^{l+\alpha+2}\frac{1}{r^\lambda}
 - (2l+1) \left ( \frac{1}{(l+\alpha+1)^\lambda}+\frac{1}{(l+\alpha+2)^{\lambda}} \right ) - \sum_{r=l+2}^{l+\alpha+3} \frac{1}{r^\lambda}\notag \\
 \geq& \sum_{r=l+1}^{l+\alpha-2} \frac{1}{r^{\lambda-1}} -2 \sum_{r=-\alpha^2+2l-\alpha+1}^{l+\alpha-2}\frac{1}{r^\lambda} - \sum_{r=l+2}^{l+\alpha+3} \frac{1}{r^\lambda} = \sum_{r=l+1}^{-\alpha^2+2l-\alpha} \frac{1}{r^{\lambda-1}} +\sum_{r=-\alpha^2+2l-\alpha+1}^{l+\alpha-2} \frac{r-2}{r^{\lambda}} - \sum_{r=l+2}^{l+\alpha+3} \frac{1}{r^\lambda}. 
\end{align}
If $\alpha=1$ it easy to see that the right-hand side of Equation $\eqref{last_equation_1b}$ is positive, indeed the second sum there exists if and only if $\alpha\geq 2$, since $l\geq 1$. Suppose now that $\alpha \geq 2$ and write
    \begin{align}
        \eqref{last_equation_1b} 
    & =\frac{l+1}{(l+1)^{\lambda}} + \underbrace{\sum_{r=l+2}^{-\alpha^2+2l-\alpha} \left ( \frac{1}{r^{\lambda-1}} - \frac{1}{r^\lambda} \right )}_{>0} + \underbrace{\sum_{r=-\alpha^2+2l-\alpha+1}^{l+\alpha-2} \frac{r-3}{r^{\lambda}} }_{(\textnormal{Eq.In0})}
    \notag \\
    & - \left ( \frac{1}{(l+\alpha-1)^\lambda}+\frac{1}{(l+\alpha)^\lambda} + \frac{1}{(l+\alpha+1)^\lambda}+\frac{1}{(l+\alpha+2)^\lambda} \right ) \notag \\
 & \geq \frac{l+1}{(l+1)^{\lambda}} +\frac{-\alpha^2+2l-\alpha-2}{(-\alpha^2+2l-\alpha+1)^{\lambda}} - \left (\frac{1}{(l+\alpha-1)^\lambda}
 +  \underbrace{\frac{1}{(l+\alpha)^\lambda} + \frac{1}{(l+\alpha+1)^\lambda}+\frac{1}{(l+\alpha+2)^\lambda}}_{(\textnormal{Eq.In1})} \right ),
\end{align}
where the inequality holds because the first sum is positive and we bounded the sum $(\textnormal{Eq.In0})$ with only the first term. 
Trivially, we can bound
\[
\frac{l+1}{(l+1)^{\lambda}} - (\textnormal{Eq.In1}) \geq \frac{l-2}{(l+1)^{\lambda}},
\]
so that we find
\begin{align}
    \tilde \Delta (l+1, \alpha, 2, \lambda)-\tilde \Delta (l, \alpha, 2, \lambda) 
    & \geq \underbrace{\frac{l-2}{(l+1)^{\lambda}} - \frac{1}{(l+\alpha-1)^\lambda}}_{>0} +\frac{-\alpha^2+2l-\alpha-2}{(-\alpha^2+2l-\alpha+1)^{\lambda}} > 0,
\end{align}
which proves the claim for that $\tilde \Delta (l, \alpha, 2, \lambda)$ is increasing in $l$. Moreover, $k_1-k_2 >a$ implies $l>\alpha^2+\alpha$, indeed
we obtain
\begin{align}\label{conditionkalpha}
    k_1-k_2>a \iff -\alpha^2+2(l-\alpha)>l-\alpha \iff l >\alpha^2+\alpha.
\end{align}
We conclude the induction step by proving $\tilde \Delta (\alpha^2+\alpha, \alpha, 2, \lambda)>0$, which follows from  
Lemma \ref{lem:increasing_alpha_2}, since $l\geq 2$.

 \paragraph{n-th step.} We want to prove that $\tilde \Delta \left (l,\alpha,C+1,\lambda \right)>0$ by using the induction hypothesis, i.e., $\tilde \Delta \left (l,\alpha,C,\lambda \right)>0$. In the following, we simplify the notation by using $\tilde \Delta(C)$ for $\tilde \Delta \left (l,\alpha,C,\lambda \right)$.
 \begin{align}
     \tilde \Delta \left (C+1 \right )-\tilde \Delta \left (C\right) &=\Delta \left ( C \right)-\Delta \left (C-1\right)-\frac{l-\alpha}{(l+\alpha+C+1)^\lambda}. 
 \end{align}

Consider first $l-\alpha=1$, then the difference is positive since
\begin{align}
     \tilde \Delta \left (C+1\right)-\tilde \Delta \left (C\right) 
    &=- \sum_{r=l+1}^{2\alpha + C+2} \frac{1}{r^\lambda}
    +\sum_{r=1}^{-\alpha^2+1+C}\frac{1}{r^\lambda} 
    -\frac{1}{(2\alpha + C+2)^\lambda}
    \notag \\
    &\geq - \zeta(\lambda, l)
    + \underbrace{\sum_{r=1}^{-\alpha^2+1+C}\frac{1}{r^\lambda}}_{> 1}
    -\frac{1}{(2\alpha + C+2)^\lambda} >0,
 \end{align}
 and $\tilde \Delta \left (C\right) >0$ by induction assumption.
From now on, we take $l-\alpha \geq 2$. Recall that $-\alpha^2+C(l-\alpha) =k_1-k_2 \geq l-\alpha+1 \geq 3$. Then,

\begin{align}
    \tilde \Delta \left (C+1\right) = &\tilde \Delta \left (C\right)+\Delta \left (C \right)-\Delta \left (C-1\right)-\frac{l-\alpha}{(l+\alpha+C+1)^\lambda} \notag \\
    =& \tilde \Delta \left (C\right)+(l-\alpha) \sum_{r=1}^l \frac{1}{r^\lambda} 
   - \sum_{r=1}^{l-\alpha-1} \frac{l-\alpha-r}{r^\lambda} 
   -(l-\alpha) \sum_{r=1}^{l+\alpha+C} \frac{1}{r^\lambda} +(l-\alpha) \sum_{r=1}^{-\alpha^2+C(l-\alpha)-1} \frac{1}{r^\lambda}\notag \\
   &+\sum_{r=-\alpha^2+C(l-\alpha)}^{-\alpha^2+(1+C)(l-\alpha)-1} \frac{-\alpha^2+(1+C)(l-\alpha)-r}{r^\lambda}
   -\frac{l-\alpha}{(l+\alpha+C+1)^\lambda}.
\end{align}
By induction hypothesis $\tilde \Delta \left (C\right) >0$ and we can further bound 
\begin{align}
    \tilde \Delta \left (C+1\right) &\geq  
    (l-\alpha)\sum_{r=l-\alpha}^{-\alpha^2+C(l-\alpha)-1}  \underbrace{\left ( \frac{1}{r^\lambda} -\frac{1}{(r+\alpha+1)^\lambda} \right )}_{>0}
   -(l-\alpha) \sum_{r=-\alpha^2+C(l-\alpha)}^{l+C-1} \frac{1}{(r+\alpha+1)^\lambda} \notag \\
   &+\sum_{r=1}^{l-\alpha-1} \frac{1}{r^{\lambda-1}}
   +\sum_{r=-\alpha^2+C(l-\alpha)}^{-\alpha^2+(1+C)(l-\alpha)-1} \frac{-\alpha^2+(1+C)(l-\alpha)-r}{r^\lambda} 
   -\frac{l-\alpha}{(l+\alpha+C+1)^\lambda}. \label{eq:kaput}
\end{align}

The first sum in the equation above is positive, hence we will drop it, we reshuffle the terms in the sums and get
\begin{align}
     \eqref{eq:kaput} &\geq 
      -(l-\alpha) \sum_{r=-\alpha^2+C(l-\alpha)}^{l+C-1} \frac{1}{(r+\alpha+1)^\lambda} 
     +(l-\alpha)\sum_{r=-\alpha^2+C(l-\alpha)}^{-\alpha^2+(1+C)(l-\alpha)-1} \frac{1}{r^\lambda}\notag \\
     & +\sum_{r=-\alpha^2+C(l-\alpha)}^{-\alpha^2+(1+C)(l-\alpha)-1} \frac{-\alpha^2+C(l-\alpha)}{r^\lambda} 
     -\frac{l-\alpha}{(l+\alpha+C+1)^\lambda}. \label{eq:kaput1}
\end{align}

To simplify the notation we will write from now on $k_1-k_2:=k$ and use $k=-\alpha^2+C(l-\alpha)$,  $k \geq 3$, and $\lambda>1.8$:
\begin{align}
  \eqref{eq:kaput1}
    \geq & -(l-\alpha) \sum_{r=k}^{l+C-1} \frac{1}{(r+\alpha+1)^\lambda}+(l-\alpha) \sum_{r=k}^{k+l-\alpha} \frac{1}{r^\lambda} +\sum_{r=k}^{k+l-\alpha} \frac{-\alpha^2+C(l-\alpha)}{r^\lambda} -\frac{l-\alpha}{(l+\alpha+C+1)^\lambda}  \notag \\
    = & - \underbrace{(l-\alpha) \sum_{r=1}^{l+C-k} \frac{1}{(r+\alpha+k)^\lambda}}_{(\textnormal{Eq.S1})}+(l-\alpha) \sum_{r=1}^{l-\alpha+1} \frac{1}{(r+k-1)^\lambda} \notag \\
    &+\sum_{r=1}^{l-\alpha+1} \frac{-\alpha^2+C(l-\alpha)}{(r+k-1)^\lambda} -\frac{l-\alpha}{(l+\alpha+C+1)^\lambda}. \label{eq:kaput2}
\end{align}
By the choice of $C, l, \alpha, a, k$, we have that:
\begin{align*}
    2 \sum_{r=1}^{l-\alpha-1} \frac{1}{r^{\lambda}} > \sum_{r=1}^{2(l-\alpha-1)} \frac{1}{r^{\lambda}} \geq \sum_{r=1}^{l+C-(k_1-k_2)} \frac{1}{r^{\lambda}},
\end{align*}
and hence
\[
(\textnormal{Eq.S1}) \leq 2(l-\alpha) \sum_{r=1}^{l-\alpha-1} \frac{1}{(r+\alpha+k)^\lambda}.
\]
Then we reshuffle the sums in the following way:

\begin{align*}
-2(l-\alpha) &\sum_{r=1}^{l-\alpha-1} \frac{1}{(r+\alpha+k)^\lambda}+(l-\alpha) \sum_{r=1}^{l-\alpha+1} \frac{1}{(r+k-1)^\lambda} +\sum_{r=1}^{l-\alpha+1} \frac{-\alpha^2+C(l-\alpha)}{(r+k-1)^\lambda}\\ &= 
   \left (\frac{(-\alpha^2+C(l-\alpha))}{(l-\alpha+k-1)^\lambda}+ \frac{(-\alpha^2+C(l-\alpha))}{(l-\alpha+k-2)^\lambda} \right )
     +\left ( \frac{(l-\alpha)}{(l-\alpha+k-1)^\lambda} +\frac{(l-\alpha)}{(l-\alpha+k-2)^\lambda} \right ) \nonumber \\
&+ \left (\underbrace{ \sum_{r=1}^{l-\alpha-1} \frac{(l-\alpha)}{(r+k-1)^\lambda} - \sum_{r=1}^{l-\alpha-1} \frac{(l-\alpha)}{(r+\alpha+k)^\lambda}}_{>0}\right ) 
     +\left (\underbrace{\sum_{r=1}^{l-\alpha-1} \frac{-\alpha^2+C(l-\alpha)}{(r+k-1)^\lambda}
     -  \sum_{r=1}^{l-\alpha-1} \frac{(l-\alpha)}{(r+\alpha+k)^\lambda}}_{>0} \right ).
     \notag 
\end{align*}
Ignoring the positive contributions from the last equation, and using the bound on (Eq.S1) we can lower bound $\tilde \Delta \left (C+1\right)$ further by

\begin{align}
   \eqref{eq:kaput2} &\geq 
      (l-\alpha)\left ( \frac{1}{(l-\alpha+k-1)^\lambda} +\frac{1}{(l-\alpha+k-2)^\lambda} -\frac{1}{(l+\alpha+C+1)^\lambda}\right ) \notag \\
     \geq & (l-\alpha)\left ( \frac{2}{(l-\alpha+k-2)^\lambda} -\frac{1}{(l+\alpha+C+1)^\lambda}\right ) > 0,
\end{align}
since $-\alpha^2+C(l-\alpha)=k >0$ and $l-\alpha+k-2 \leq l+\alpha+C+1$ which yields the conclusion for the induction step.
\end{itemize}

\subparagraph{Case (ii.2).} Recall that in this case the protuberance of the rectangle $\overline{\mathscr{R}}_{a,b}^{k_2}$ is attached along the longer side $b$.
By assumption, the classical perimeter of $\mathscr{Q}_l^{k_1}$ is smaller than the classical perimeter of all rectangles $\overline{\mathscr{R}}_{a,b}^{k_2}$, i.e. we have that $2(a+b+1)=2(2l+1+C)$ for some $C \geq 1$.  

We distinguish three cases: $k_2\in \{0,\ldots, a-1\}$, $k_2 \in \{a,\ldots, k_1\}$, and $k_2>\max\{k_1,a\}$.
\begin{itemize}
\item[(ii.2.a)] $k_2 < a$. In this case, we consider the rectangle $\mathscr{R}_{a,b}^{k_2} \in \mathcal{M}_n$ with the protuberance attached along the shorter side $a$ and we will show that its nonlocal perimeter is smaller than the nonlocal perimeter of $\overline{\mathscr{R}}_{a,b}^{k_2} \in \mathcal{M}_n$. Indeed,
    \begin{align}\label{eq:waschen}
        Per_\lambda(\overline{\mathscr{R}}_{a,b}^{k_2})-Per_\lambda(\mathscr{R}_{a,b}^{k_2})= k_2 \left ( \zeta(\lambda, a+1) - \zeta(\lambda, b+1)\right )=k_2 \sum_{r=a+1}^{b} \frac{1}{r^\lambda} >0.
    \end{align}
The claim follows from 
\[
Per_\lambda(\overline{\mathscr{R}}_{a,b}^{k_2}) - Per_\lambda(\mathscr{Q}_l^{k_1}) = (\underbrace{Per_\lambda(\overline{\mathscr{R}}_{a,b}^{k_2}) -Per_\lambda(\mathscr{R}_{a,b}^{k_2})}_{> 0 \text{ by }\eqref{eq:waschen}}) + (\underbrace{ Per_\lambda(\mathscr{R}_{a,b}^{k_2})- Per_\lambda(\mathscr{Q}_l^{k_1}}_{> 0 \text{ by Case (ii.1)}}) >0.
\]

\item[(ii.2.b)] $k_2 \in \{a, \dots, k_1\}$. We note that if $k_1\leq a-1$ then this interval is empty. As before, by using \eqref{eq:semiper_square} and \eqref{eq:semiper_rectangle} we obtain the following
\begin{align}
    Per_{\lambda}(\mathscr{R}_{a,b}^{k_2})&-Per_{\lambda}(\mathscr{Q}_l^{k_1}) \nonumber\\
= & (Per_{\lambda}(\mathscr{R}_{a,b})- Per_{\lambda}(\mathscr{Q}_l)) \nonumber
-2\sum_{i=k_2+1}^{k_1} \zeta(\lambda, i)
-2 (k_1-k_2) \zeta(\lambda, l+1)
+2k_2 \sum_{r=a+1}^{l} \frac{1}{r^\lambda} \notag \\
\geq & (Per_{\lambda}(\mathscr{R}_{a,b})- Per_{\lambda}(\mathscr{Q}_l))
-2\sum_{i=1}^{k_1-k_2} \zeta(\lambda, i)
-2 (k_1-k_2)\zeta(\lambda, l+1) \notag \\
=& 2\tilde{F}_1(a,l, k_1-k_2)+2\tilde{F}_2(a,l,k_1-k_2,b),
\end{align}
where the functions $\tilde F_1, \tilde F_2$ are defined in \eqref{def:F_tilde1} and \eqref{def:F_tilde2}. In Case (ii.1.a) we proved that $k\mapsto \tilde{F}_2(a,l, k,b)$ is decreasing in $k$. Since $k_1-k_2 \leq l-1$, we have to evaluate $\tilde{F}_2(a,l, l-1, b)$ which is strictly positive by \eqref{F2tilde_positive}. 

By the conditions \eqref{condition_k1_k2} and recalling $k_1 \leq l-1$ and $k_2 \geq a$ we have that
\begin{align}
    -\alpha^2+C(l-\alpha)=k_1-k_2 \leq l-1-a \iff -(l-a)^2+Ca \leq l-1-a \implies l \geq a+\sqrt{3a}.
\end{align}

Again using Case (ii.1.a) we have that for 
$l \geq a+ \sqrt{a}$, $\tilde{F}_1(a,l, k_1-k_2) > 0$ which yields the claim.

In case that $k_2=k_1$, then $\tilde{F}_1(a,l, k_1-k_2)=F_1(l,a)$ and $\tilde{F}_2(a,l,k_1-k_2,b)=F_2(l,a, b)$, where $F_1,F_2$ are defined in \eqref{def:F1}, \eqref{def:F2}, and we can conclude in the same way. 
\item[(ii.2.c)] $k_2 > \max \{ a, k_1 \}$.  
By using \eqref{eq:semiper_square} and \eqref{eq:semiper_rectangle} we obtain the following
\begin{align}
 &   Per_{\lambda}(\mathscr{R}_{a,b}^{k_2})-Per_{\lambda}(\mathscr{Q}_l^{k_1}) \nonumber \\
&= (Per_{\lambda}(\mathscr{R}_{a,b})- Per_{\lambda}(\mathscr{Q}_l))
-2\sum_{i=1}^{k_2-k_1} \zeta(\lambda, i)
+2 (k_2-k_1) \zeta(\lambda, l+1)
+2\sum_{i=1}^{k_2-k_1}\zeta(\lambda, i+k_1) +2\sum_{i=1}^{k_2-k_1}\zeta(\lambda, i) \\
&=2\tilde{F}_1(a,l, k_2-k_1)+ \underbrace{2\tilde{F}_2(a,l,k_2-k_1,b)
+2\sum_{i=1}^{k_2-k_1}\zeta(\lambda, i+k_1) +2\sum_{i=1}^{k_2-k_1}\zeta(\lambda, i)}_{(\textnormal{Eq.S2})}, \nonumber
\end{align}
where we again find the functions $\tilde F_1, \tilde F_2$ that are defined in \eqref{def:F_tilde1} and in \eqref{def:F_tilde2} respectively.
As in Case (ii.1.a), we have that $\tilde{F}_1(a,l, k_2-k_1) \geq 0$ if $l \geq a+ \sqrt{a}$. Trivially we get that $l \geq a+ \sqrt{a}$ also in this case.
The positivity of $(\textnormal{Eq.S2})$ follows from analogous arguments.
\end{itemize}

Summarizing the results of the Cases (i) and (ii), we can conclude that for $n=l^2+k_1$, we have proven that
    \begin{align}
        \text{argmin}_{\mathcal{P} \in \mathcal{M}_{n}} \{Per_{\lambda}(\mathcal{P}) \} \subset 
        \{ \mathscr{Q}_l^{k_1} \in \mathcal{M}_n \} 
        \cup \left\{ \mathscr{R}_{a,b}^{k_2} \in \mathcal{M}_n \, | \, per(\mathscr{R}_{a,b}^{k_2} ) = per(\mathscr{Q}_l^{k_1})\right \}.
    \end{align}

\subsection{Proofs of the auxiliary lemmas}\label{proof_lemmas}
\begin{proof}[Proof of Lemma~\ref{lem:rows}]
From now on, in order to simplify the notation, we drop the dependence on $r$ in $d^{(r)}_j$.
Let $d_1$ (resp. $d_2$) the length of the horizontal (resp. vertical) strip containing $x$. 
We define a bijection, where we identify every unit square in $Q(x) \subset \mathcal{P}$ with the four distances between the center point $x$ and the north, south, west and east nearest sites in $\mathbb{Z}^2 \cap \mathcal{P}$ along the horizontal and vertical directions, i.e.,   
\begin{align*}
     Q(x) \qquad & \to \qquad  (x_1,d_1-x_1,x_2,d_2-x_2);
\end{align*}
\begin{figure}
\begin{center}
    \includegraphics[scale=1]{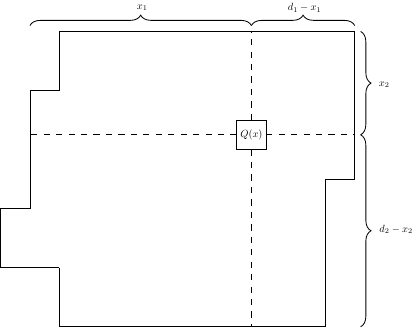}
    \caption{An example of a convex polyomino $\mathcal{P}$, a fixed square $Q(x) \subset \mathcal{P}$ and the respective distances from the sites in $\mathbb{Z}^2 \cap \{\mathbb{R}^2 \setminus \mathcal{P} \}.$}
 \label{fig:N_S_E_W}
\end{center}
    \end{figure}
    see Figure \ref{fig:N_S_E_W} for an example.
Thus, we have
\begin{align}\label{contribution_unit_square}
Per_{\lambda}(Q(x))= \zeta(\lambda, x_1) +\zeta(\lambda, d_1-x_1) + \zeta(\lambda, x_2)+\zeta(\lambda, d_2-x_2),
\end{align}
where $x_1$ and $x_2$ (resp. $d_1-x_1$ and $d_2-x_2$) are respectively the horizontal and vertical distances between $x$ and the north and the west (resp. south and the east) nearest sites in $\mathbb{Z}^2 \cap \mathcal{P}^c$.
Moreover, $Per_{\lambda}^H(\mathcal{P})$ can be written as the sum of all contributions of all rows intersecting $\mathcal{P}$, i.e,. 
\begin{align*}
 Per_{\lambda}^{H}(\mathcal{P} )= Per_{\lambda}^H \left ( \bigcup_{r}  \{ \mathcal{P} \cap r\} \right ) =\sum_{r} Per_{\lambda}^{H}(\mathcal{P} \cap r).
\end{align*}
The horizontal nonlocal perimeter contribution $Per_{\lambda}^H(\mathcal{P} \cap r)$ again  can be written as the sum of all contributions of all unit squares composing $\mathcal{P} \cap r$, i.e.,
\begin{align*}
    Per_{\lambda}^H(\mathcal{P} \cap r)
    =&\sum_{i=1}^{l_1} \zeta(\lambda, i)
    +\sum_{j=2}^n\sum_{i=1}^{l_j} \left ( \zeta(\lambda, i) -
    \sum_{k=d_{j-1}+i}^{i+l_{j-1}+d_{j-1}-1}\frac{1}{k^\lambda} \right )
    +\sum_{i=1}^{l_n}\zeta(\lambda, i)
    +\sum_{j=1}^{n-1}\sum_{i=1}^{l_j} \left ( \zeta(\lambda, i)-
    \sum_{k=d_{j}+i}^{i+l_{j+1}+d_{j}-1}\frac{1}{k^\lambda} \right ). \notag \\
\end{align*}
The case for $n=1$ follows trivially.
\end{proof}

\begin{proof}[Proof of Lemma \ref{reduction_distance}]

We consider a column of $\mathcal{P}$ that contains at least two vertical strips, $S_1,\ldots, S_N$. Suppose $N=2$, and starting from the top and proceeding in lexicographic order, let $l_1$, $l_2$ be the lengths of the first and the second vertical strips. Let $\mathcal{P}'$ be the polyomino obtained from $\mathcal{P}$ by moving along the column the first strip toward the other and by reducing the distance by $d-1$. Then, the lengths of the two strips $S'_1, S'_2$ in $\mathcal{P}'$ are equal to the lengths of $S_1$ and $S_2$, i.e., $l_1$ and $l_2$. 
Using Lemma \ref{lem:rows} and considering first a reduction in distance of the two strips by $d-1$, we compute the difference between the nonlocal perimeter contributions along this column 
\begin{align*}
    Per_{\lambda}^V(S_1 \cup S_2) &- Per_{\lambda}^V(S'_1 \cup S'_2)\\
    =& \sum_{i=1}^{l_1} \zeta(\lambda, i) + \sum_{i=1}^{l_2} \left (\zeta(\lambda, i) -\sum_{k=i+d}^{i+d+l_1-1} \frac{1}{k^\lambda} \right)+\sum_{i=1}^{l_2}\zeta(\lambda, i) + \sum_{i=1}^{l_1} \left (\zeta(\lambda, i) -\sum_{k=i+d}^{i+d+l_2-1} \frac{1}{k^\lambda} \right) \notag \\
    - &\left [ \sum_{i=1}^{l_1} \zeta(\lambda, i) + \sum_{i=1}^{l_2} \left (\zeta(\lambda, i) -\sum_{k=i+d-1}^{i+d+l_1-2} \frac{1}{k^\lambda} \right)+\sum_{i=1}^{l_2}\zeta(\lambda, i) + \sum_{i=1}^{l_1} \left (\zeta(\lambda, i) -\sum_{k=i+d-1}^{i+d+l_2-2} \frac{1}{k^\lambda} \right) \right ] \notag \\
    =& \sum_{i=1}^{l_2} \sum_{k=i+d}^{i+d+l_1-1} \left (\frac{1}{(k-1)^\lambda} -\frac{1}{k^\lambda} \right)+
    \sum_{i=1}^{l_1} \sum_{k=i+d}^{i+d+l_2-1} \left (\frac{1}{(k-1)^\lambda} -\frac{1}{k^\lambda} \right)>0.
\end{align*}
The general case for $N>2$, and for the horizontal strips, follows in the same way. 
\end{proof}

\begin{proof}[Proof of Lemma \ref{lem:reduction_strips}]

We construct the polyomino $\mathcal{P}'$ as follows. Let $x_1,\ldots,x_k \in c^2_{\mathcal{P}} \cap \bigcup_{i=1}^{N}A_i$ be the middle points of the squares that do not belong to $\mathcal{P}$, and let $r_1,\ldots,r_k$ be their corresponding rows. We move along the horizontal axes each unit square of $c^1_{\mathcal{P}} \cap r_i \neq \emptyset$ from its position to $x_i$ for $i=1,\ldots,k$, see the right side in Figure \ref{fig:different_strips}. In other words, investigating rows in $c^2_{\mathcal{P}}$, every time there is a red square in $c^1_{\mathcal{P}}$ but not in $c^2_{\mathcal{P}}$ we move it to $c^2_{\mathcal{P}}$ and call this new polyomino $\mathcal{P}'$.

\begin{figure}[htb!]
\begin{center}
    \includegraphics[scale=0.3]{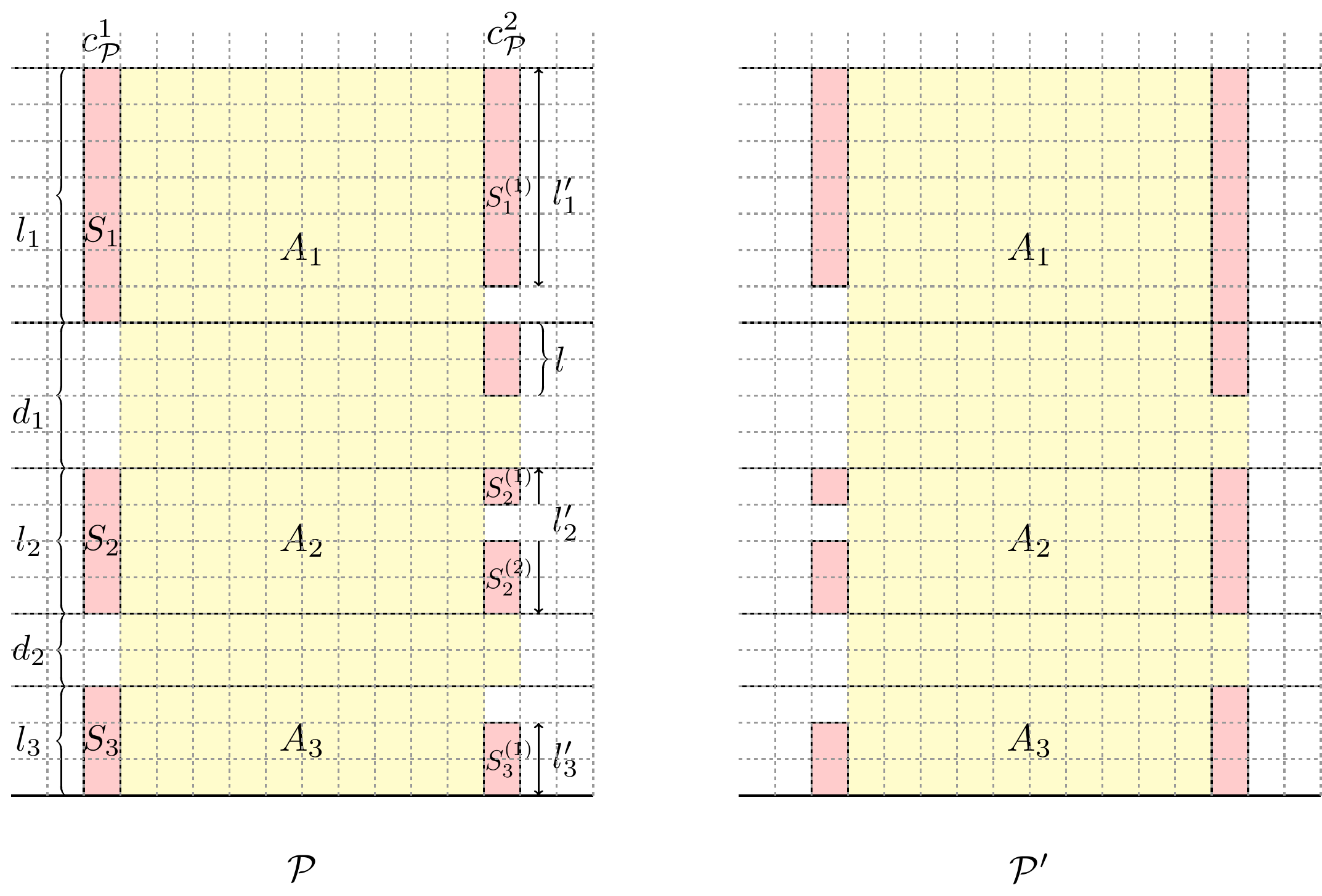}
    \caption{On the left there is an example of a possible polyomino $\mathcal{P}$. The column $c^1_{\mathcal{P}}$ consists of 3 strips of lengths 7, 4 resp. 3.  $A_1,A_2,A_3$ are the rows that intersect the strips $S_1, S_2$ resp. $S_3$. The column $c^2_{\mathcal{P}}$ has four strips which intersect the rows in the areas $A_1,A_2,A_3$.
    On the right, we see the polyomino $\mathcal{P}'$ obtained from $\mathcal{P}$ with the same area and such that $Per_{\lambda}(\mathcal{P}') < Per_{\lambda}(\mathcal{P})$. The red unit squares belonging to the polyominos, whereas the white unit squares are not in $\mathcal{P}$. The yellow squares are not investigated yet and can either belong to $\mathcal{P}$ or not.}
 \label{fig:different_strips}
\end{center}
    \end{figure}

Note that if $c^2_{\mathcal{P}'} \cap (\bigcup_{i=1}^{N}A_i)^c= \emptyset$, then $Per_{\lambda} (c^1_{\mathcal{P}} \cup c^2_{\mathcal{P}}) = Per_{\lambda} (c^1_{\mathcal{P}'} \cup c^2_{\mathcal{P}'})$, because in this case we just swapped the unit squares in the two columns. Suppose that there are $\beta=\sum_{i=1}^{N} m_i>0$ vertical strips which do not intersect with rows in $A_1,\ldots, A_N$, i.e., in $c^2_{\mathcal{P}} \cap \mathcal{P} \cap \Big (\bigcup_{i=1}^{N} A_i \Big)^c$. In order to simplify the proof we suppose $\beta=1$ and we denote by $l$ the length of this vertical strip, see the example in Figure \ref{fig:different_strips}. 
Then, we have
\begin{align*}
    Per_{\lambda} (c^1_{\mathcal{P}} \cup c^2_{\mathcal{P}}) = Per_{\lambda} \left(\bigcup_{i=1}^N S_i \right) + Per_{\lambda} \left(\bigcup_{i=1}^N\bigcup_{j=1}^{m_i} S^{(j)}_i \right)
    +\sum_{i=1}^{l} \zeta(\lambda, i)-K_1,
\end{align*}
where $K_1$ accounts for the terms in the double sum which are counted twice due to the contribution of the perimeter from the presence of the unit squares in $S_1^{(1)},\ldots,S_N^{(m_N)}$. In other words, the contribution of the nonlocal perimeter of the columns consists of the contributions of the strips in these columns, which do not intersect with $A_1,\ldots, A_N$. 
Analogously, let $S'_1,\ldots, S'_{N'}$ (resp. $S_1'^{(1)},\ldots,S_1'^{(m_1)}, S_2'^{(1)},\ldots,S_2'^{(m_2)},\ldots,S_{N'}'^{(1)},\ldots, S_{N'}'^{(m_{N'})}$)  be the vertical strips in $c^1_{\mathcal{P}'}$ (resp. $c^2_{\mathcal{P}'}$), we write
\begin{align*}
    Per_{\lambda} (c^1_{\mathcal{P}'} \cup c^2_{\mathcal{P}'}) & = Per_{\lambda} \left(\bigcup_{i=1}^{N'} S'_i \right) + Per_{\lambda} \left(\bigcup_{i=1}^{N'} \bigcup_{j=1}^{m_i} S'^{(j)}_i \right)
    +\sum_{i=1}^{l} \zeta(\lambda, i)-K_2 \\
    & =   Per_{\lambda} \left(\bigcup_{i=1}^N S_i \right) + Per_{\lambda} \left(\bigcup_{i=1}^N\bigcup_{j=1}^{m_i} S^{(j)}_i \right) 
    +\sum_{i=1}^{l} \zeta(\lambda, i) -K_2,
\end{align*}
where $K_2$ accounts for the terms in the double sum which are counted twice due to the contribution of the perimeter from the presence of the unit squares in $S_1'^{(1)},\ldots, S_{N'}'^{{(m_{N'})}}$. Since $K_2 \geq K_1$ we have that
\begin{align*}
       Per_{\lambda} (c^1_{\mathcal{P}'} \cup c^2_{\mathcal{P}'}) =     Per_{\lambda} (c^1_{\mathcal{P}} \cup c^2_{\mathcal{P}}) +K_1-K_2 \leq 0.
\end{align*}
The case  $\beta>1$ follows in an analogous way.
\end{proof}

\begin{proof}[Proof of Lemma \ref{lem:increasing_alpha_2}]
We compute by hand $f(2)$ and we find that $f(2)>0$. Then, assume $x \geq 3$ and
  we simplify the function $f(x)$ in the following way.
    \begin{align*}
        f(x) = &\sum_{r=1}^{x^2-1}\frac{1}{r^{\lambda-1}}+
        \sum_{r=x^2}^{x^2+x-1} \frac{x^2(x+1)^2-x(x+1)r}{r^\lambda} 
        -x^2 \sum_{r=x^2}^{x^2+x-1} \frac{1}{(r+x+1)^\lambda} 
        -x^2 \left (\frac{1}{(x^2+2x+1)^\lambda} +\frac{1}{(x^2+2x+2)^\lambda} \right ) \notag \\
        &-x\sum_{r=x^2}^{x^2+x-1} \frac{1}{r^{\lambda-1}}
        + \frac{x}{x+1} 
        -x^2\sum_{r=x^2+x}^{(x+1)^2} \frac{x^2+2x+2-r}{r^\lambda}.
        \end{align*}
        We note that
        \begin{align}
            \sum_{r=x^2}^{x^2+x-1} \frac{x^2(x+1)^2-x(x+1)r}{r^\lambda} 
        -x^2 \sum_{r=x^2}^{x^2+x-1} \frac{1}{(r+x+1)^\lambda} 
        -x^2 \left (\frac{1}{(x^2+2x+1)^\lambda} +\frac{1}{(x^2+2x+2)^\lambda} \right ) \geq 0.
        \end{align}
        Indeed we consider the first term of the first and the second sum and we compare them with the term in the brackets:
        \begin{align}
        &\frac{x^2(x+1)}{x^{2\lambda}} 
        -x^2 \left (\frac{1}{(x^2+x+1)^\lambda}+\frac{1}{(x^2+2x+1)^\lambda} +\frac{1}{(x^2+2x+2)^\lambda} \right ) \geq \frac{x^2(x+1)}{x^{2\lambda}}  -\frac{3x^2}{(x^2+x+1)^\lambda} \geq 0,
        \end{align}
        since $x \geq 3$. Moreover
\begin{align}
    \sum_{r=x^2+1}^{x^2+x-1} \frac{x^2(x+1)^2-x(x+1)r}{r^\lambda} 
        -x^2 \sum_{r=x^2+1}^{x^2+x-1} \frac{1}{(r+x+1)^\lambda} 
        &= \sum_{r=1}^{x-1} \left (\frac{x^2(x+1)^2-x(x+1)(r+x^2)}{(r+x^2)^\lambda} 
        - \frac{x^2}{(r+x^2+x+1)^\lambda} \right ) \notag \\
        & \geq \sum_{r=1}^{x-1} \frac{x^3-x(x+1)r}{(r+x^2)^\lambda} >0.
\end{align}
Thus,
        \begin{align}
        f(x) \geq & \sum_{r=1}^{x^2-1}\frac{1}{r^{\lambda-1}}
        -x\sum_{r=x^2}^{x^2+x-1} \frac{1}{r^{\lambda-1}}
        -x^2\sum_{r=x^2+x}^{(x+1)^2} \frac{(x+1)^2+1-r}{r^\lambda}.
    \end{align}
Let us call 
\[
g(x):=\sum_{r=1}^{x^2-1}\frac{1}{r^{\lambda-1}}
        -x\sum_{r=x^2}^{x^2+x-1} \frac{1}{r^{\lambda-1}}
        -x^2\sum_{r=x^2+x}^{(x+1)^2} \frac{(x+1)^2+1-r}{r^\lambda}.
\]
Next, we prove that $g(x)$ is increasing in $x$.
\begin{align}\label{increasing_g_lemma}
    g(x+1)-g(x)=&\underbrace{\sum_{r=x^2}^{x^2+2x}\frac{1}{r^{\lambda-1}}-\sum_{r=x^2+2x+1}^{x^2+3x+1}\frac{1}{r^{\lambda-1}}}_{(\textnormal{Eq.A1})}
        +x \underbrace{\left (\sum_{r=x^2}^{x^2+x-1} \frac{1}{r^{\lambda-1}}-\sum_{r=x^2+2x+1}^{x^2+3x+1} \frac{1}{r^{\lambda-1}} \right)}_{(\textnormal{Eq.A2})} \notag \\
        &+x^2 \sum_{r=x^2+x}^{(x+1)^2} \frac{(x+1)^2+1-r}{r^\lambda} 
        -(x+1)^2 \sum_{r=x^2+3x+2}^{(x+2)^2} \frac{(x+2)^2+1-r}{r^\lambda}.
\end{align}
We have the following lower bound for (Eq.A1):
\begin{align*}
 (\textnormal{Eq.A1})=   \sum_{r=x^2}^{x^2+2x}\frac{1}{r^{\lambda-1}}-\sum_{r=x^2+2x+1}^{x^2+3x+1}\frac{1}{r^{\lambda-1}}
    & \geq \sum_{r=x^2+x+1}^{x^2+2x}\frac{1}{r^{\lambda-1}};
\end{align*}
moreover, 
\begin{align*}
    (\textnormal{Eq.A2}) 
    &=\sum_{r=x^2}^{x^2+x-1} \left (\frac{1}{r^{\lambda-1}}-\frac{1}{(r+2x+1)^{\lambda-1}} \right )-\frac{1}{(x^2+3x+1)^{\lambda-1}}
    \geq -\frac{1}{(x^2+3x+1)^{\lambda-1}}.
\end{align*}
Thus, by reshuffling the terms in the sum
\begin{align*}
    \eqref{increasing_g_lemma} 
       & \geq \sum_{r=x^2+x+1}^{x^2+2x}\frac{1}{r^{\lambda-1}} +x^2 \sum_{r=x^2+x}^{(x+1)^2} \frac{(x+1)^2+1-r}{r^\lambda} 
        -(x+1)^2 \sum_{r=x^2+3x+1}^{(x+2)^2} \frac{(x+2)^2+1-r}{r^\lambda} \notag \\
        &= \sum_{r=x^2+x+1}^{x^2+2x} \left (\frac{x^2(x^2+2x+2)}{r^{\lambda}}- (x+1)^2\frac{x^2+x+5-r}{(r+x)^\lambda} \right ) +\frac{x^2(x+2)}{(x^2+x)^\lambda}+\frac{1}{(x+1)^{2\lambda}} \notag \\
        &-(x+1)^2 \sum_{r=x^2+2x+1}^{x^2+2x+4} \frac{x^2+x+5-r}{(r+x)^\lambda} \geq \sum_{r=x^2+x+1}^{x^2+2x} \frac{(x+1)^2(r-x-5)}{r^{\lambda}}+\frac{x^2(x+2)}{(x^2+x)^\lambda}+\frac{(x-4)(x+1)^2}{(x^2+3x+1)^\lambda} ,
\end{align*}
since $x\geq 3$. Moreover, if $r>x+5$ the first numerator is positive and in this case the minimal value of $r$ is $x^2+x+1$ and $x^2+x+1>x+5$ for $x\geq 2$. Then we get that
\begin{align*}
   g(x+1)-g(x)\geq \frac{x^2(x+2)}{(x^2+x)^\lambda}+\frac{(x-4)(x+1)^2}{(x^2+3x+1)^\lambda}.  
\end{align*}
If $x\geq 4$ then, $  g(x+1)-g(x) > 0$. Suppose $x=3$, we obtain
\begin{align}
    g(4)-g(3)=\frac{45}{12^\lambda}-\frac{16}{19^\lambda}>0,
\end{align}
since $\lambda>1$.
Thus, $g(x)$ is increasing in $x$ and $x \geq 3$. Then we can compute by hand that for $\lambda>1.8$
\begin{align}
    g(3)&= \sum_{r=1}^{8}\frac{1}{r^{\lambda-1}}
        -3\sum_{r=9}^{11} \frac{1}{r^{\lambda-1}}
        -9\sum_{r=12}^{16} \frac{17-r}{r^\lambda}>0.
\end{align}
\end{proof}

\section{Metastability of the two dimensional \emph{bi-axial} long-range Ising model}
\label{s:application}

In this section we explain how the solution of the nonlocal discrete isoperimetric
problem is related to the rigorous study of metastability for a two-dimensional
\emph{bi-axial} long-range Ising model. Theorem~\ref{main_theorem} yields a complete
characterization of the minimizers of the discrete nonlocal perimeter among
polyominoes with fixed area. For the bi-axial long-range Ising model, this result
directly translates into a characterization of the configurations minimizing the
energy within each foliation of fixed magnetization, up to negligible finite-volume
corrections, since at fixed magnetization the Hamiltonian differs from the nonlocal
perimeter only by an explicit bulk term (see Equation~(\ref{delta_per})).
This correspondence provides the main variational ingredient in the \emph{pathwise
approach} to metastability. In the subsequent sections, we introduce the metastability
problem, emphasize the structural differences between short-range and long-range
interactions, and describe in detail the heuristic metastable scenario for the
present model. In particular, Section~\ref{sec:rigorous} contains rigorous results,
together with a discussion of the remaining open problems and the main obstacles to a
complete metastability analysis.

\subsection{The metastability problem}
\label{s:two}
Metastability is a phenomenon that arises in stochastic systems out of equilibrium,
in which the dynamics remains trapped for very long times in locally stable
configurations (metastable states), before undergoing a sudden transition toward a
lower-energy state, typically the stable one. 

The mathematical study of metastability revolves around three main problems:
(i) the estimation of transition times between metastable and stable states, in
particular the determination of their asymptotic scale; (ii) the characterization
of critical configurations, namely those configurations that the system must cross
in order to exit the metastable state; (iii) the identification of the tube of
trajectories, that is, the set of paths that are followed with high probability
during the metastable–stable transition.

In this section we introduce the main definitions and discuss the most significant
results obtained for the short-range Ising model, as well as the heuristic picture
currently available in the long-range case. For rigorous results in the context of
the model under consideration we refer to Section~\ref{sec:rigorous}.

Let $\Lambda$ denote the two dimensional torus with a sufficiently large side-length $L$. We define a collection of random variables (or Ising spins) $\sigma: \Lambda \longrightarrow \{-1,+1\}^{\Lambda}$. We denote with $\mathcal{X} := \{-1,+1\}^{\Lambda}$  the configuration space, and define the Hamiltonian of the short-range Ising in the usual way:
\[
H^\textnormal{n.n}(\sigma) := -\sum_{{x,y \in \Lambda:}\atop{d_E(x,y)=1}} \sigma_x \sigma_y - h \sum_{x\in \Lambda} \sigma_x.
\]

For simplicity, let us denote by ${\bf +1}$ (resp. ${\bf -1}$) the configuration given by $\sigma=(+1)_{x\in \Lambda}$  (resp. $\sigma=(-1)_{x\in \Lambda}$). We say that two configurations in  $\mathcal{X}$ are 
\emph{communicating} if and only if they differ at most for the value 
of the spin at one site.
We consider then Glauber dynamics with Metropolis weights, i.e., a discrete time Markov chain $\sigma^t\in\mathcal{X}$, 
for any $t\in \mathbb{N}_{\geq0}$, with null transition matrix for non-communicating configurations and instead transition matrix
\begin{equation}
\label{mod015}
p_\beta(\eta,\eta')
:=
 \frac{1}{2|\Lambda|}e^{-\beta[H^\textnormal{n.n}(\eta')-H^\textnormal{n.n}(\eta)]_+}
\end{equation}
for any communicating configurations $\eta,\eta'\in\mathcal{X}$ such 
that $\eta\neq\eta'$
(where, for any $a\in\mathbb{R}$, we let $[a]_+=a$  if $a>0$, and $[a]_+=0$ if $a<0$). In case $\eta=\eta'$ we pose: 
\begin{equation}
\label{mod020}
p_\beta(\eta,\eta)
=1-\sum_{\sigma\neq\eta}p_\beta(\eta,\sigma)
\end{equation}
for any $\eta\in\mathcal{X}$.
A pivotal  quantity of interest is the expectation of the first excursion time  $\tau_+$ of the Markov chain from  the metastable state ${\bf -1}$ towards the stable state {\bf +1}, i.e., $\tau_+:=\inf\{t\in \mathbb{N}:\sigma^t=\textbf{+1}, \sigma^0=\textbf{-1}\}$. 
 
Following the so called \emph{path-wise approach} to metastability (see the general reference \cite{olivieri2005} and \cite{manzo2004}), the crucial idea is finding the optimal path connecting the metastable state to the stable one, and computing its energy height. Along the exit path, the system necessarily visits those particular configurations, called the \emph{critical droplets}, at which the optimal path attains its maximum.
Following \cite{manzo2004},  we have indeed that: 
\begin{equation}
\lim_{\beta\to\infty} \frac{1}{\beta}\ln(\mathbb{E}_{\textbf{-1}}(\tau_+))=\Phi^\textnormal{n.n}(\textbf{-1},\textbf{+1})-H^\textnormal{n.n}(\textbf{-1}),
\end{equation}
where $\Phi^\textnormal{n.n}(\textbf{-1},\textbf{+1})$ is the solution of  the variational problem in the path-space:
\begin{equation}
\label{minimax}
\Phi^\textnormal{n.n}(\textbf{-1},\textbf{+1}):=
\min_{\omega:\textbf{-1}\to\textbf{+1}}\max_{\eta\in\omega} H^\textnormal{n.n}(\eta),
\end{equation}
where $\omega$ denotes a general path of communicating configurations. The short-range Ising model has a nice continuity property of the Glauber dynamics which greatly simplifies the solution of this \emph{mini-max} problem (\ref{minimax}). 
 Given any $\sigma\in\mathcal{X}$, 
if we define: 
$$
\mathcal{V}_n:=\left\{\sigma\in\mathcal{X}:|\textnormal{Supp}(\sigma)|=n \right\},
$$
where $\textnormal{Supp}(\sigma):=\{x\in\Lambda: \sigma_x= +1\}$ is the set of sites with positive spins,
we have the following \emph{foliation} of the state space: $$\mathcal{X}=\bigcup_{n=0}^{|\Lambda|} \mathcal{V}_n.$$
It is clear that, by continuity of the dynamics, every path going from \textbf{-1} to \textbf{+1} has to cross each one of the sets $\mathcal{V}_n$. Hence, the knowledge of the energy minimizers in the foliation is of crucial importance for solving the \emph{mini-max} problem in Equation \eqref{minimax}. For the Ising short-range model one can find a minimizer in $\mathcal{V}_k$ communicating with one in $\mathcal{V}_{k+1}$, for any $k\leq |\Lambda|-1$. In other words one has to find in each foliation for a fixed area the shape of minimal perimeter. In fact,  the excitation energy 
 of a configuration $\sigma_n \in \mathcal{V}_n$ with respect to $\textbf{-1}$ for the short-range Ising model $$
  \Delta H^{\text{n.n}}(\sigma_n):=H^{\text{n.n}}(\sigma_n)-H^{\text{n.n}}(\textbf{-1}),$$ can be written in terms of the \emph{classical perimeter} of the cluster of pluses:
$$
\Delta H^{\text{n.n}}(\sigma_n)=2|per[\text{Supp}(\sigma_n)])|-2h|\text{Supp}(\sigma_n)|.
$$
In Figure~\ref{fig:short} we have reported the minimizers $\bar{\sigma}_n$ of $\Delta H^{\text{n.n}}$ on the foliation $\mathcal{V}_n$, with $n\in\{1,\ldots,100\}$.
The local minima in the figure are squares or rectangular configurations of pluses, while the local maxima (i.e., the \emph{saddles}) are configurations with a single protuberance. The configurations attaining the \emph{mini-max} are quasi-square configurations with one protuberance, of sides $l_c-1$ and $l_c$, where $l_c$ is the \emph{critical length} given by $l_c:=[2/h]+1$ (e.g., in the case of Figure~\ref{fig:short}, $l_c=5$).  

  \begin{figure}[htb!]
\begin{center}
    \includegraphics[scale=0.35]{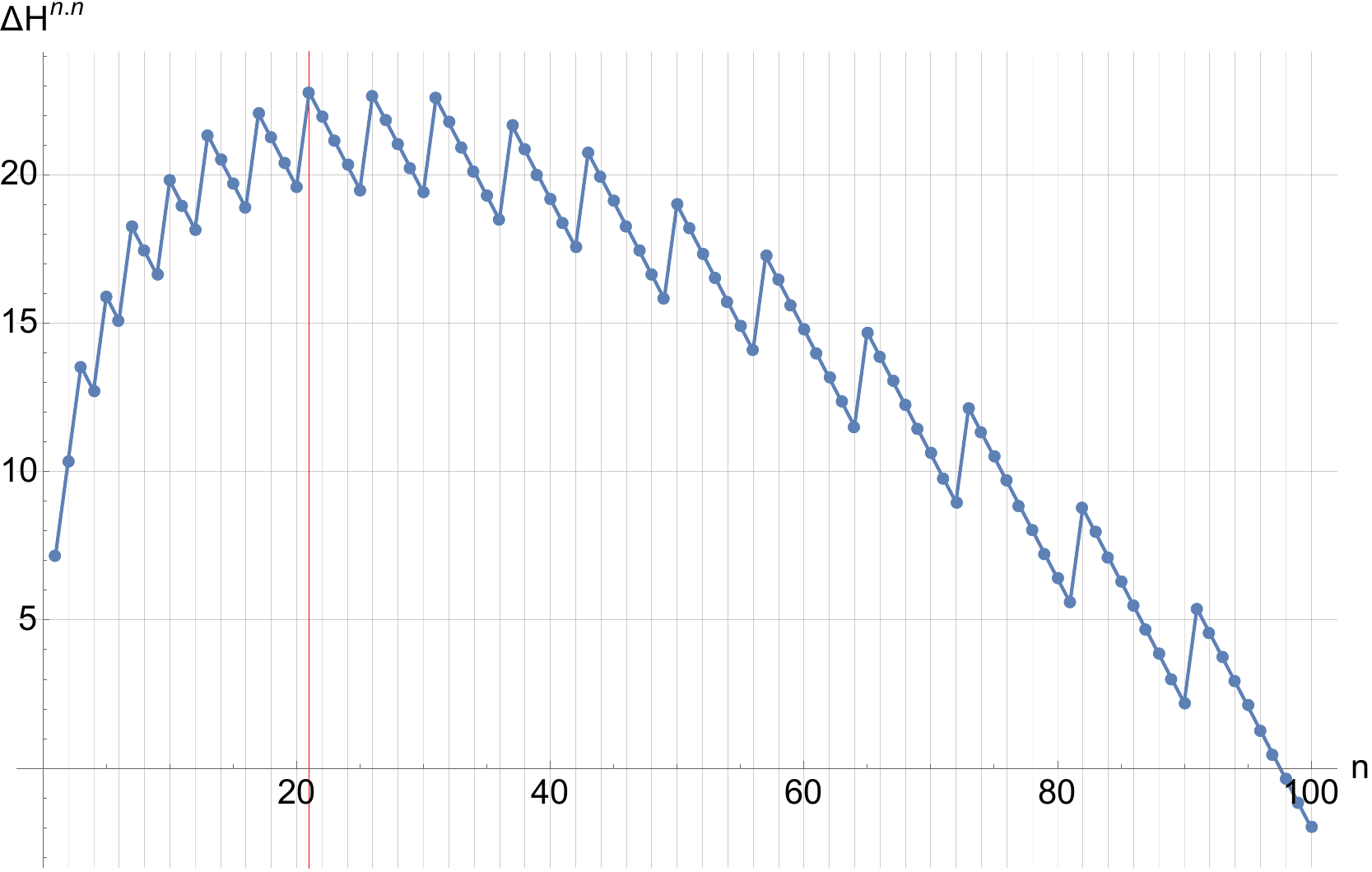} 
    \caption{Plot of  $\Delta H^{\text{n.n}}(\bar{\sigma}_n)$ vs. $n$, for the short-range 2d Ising model with $h=0.41$.}
 \label{fig:short}
\end{center}
    \end{figure}

\subsection{The metastability problem for the bi-axial long-range Ising model}
The metastability problem for a long-range model under Glauber dynamics with Metropolis weights has been studied rigorously for the first time in \cite{vanenter2019}, for a one-dimensional system. In \cite{vanenter2019} the nucleation process of the critical droplet of pluses, plunged in a see of minuses, has been indeed studied, showing that the long-range interaction changes substantially the nucleation process, since the corresponding short-range model does not present any metastable behavior.

Let us consider now the two-dimensional \emph{bi-axial} long-range Ising model with external magnetic field $h$ (see, for example, the model 3 in \cite{vanenter2019}). The Hamiltonian of the model is defined as follows: 

\begin{equation}\label{hamiltonian_function}
H^\lambda(\sigma)=-\sum_{x,y \in \Lambda,\atop{x \neq y}} J^\lambda_{x,y} \sigma_x \sigma_y -h \sum_{x \in \Lambda} \sigma_x,
\end{equation}
where $J^\lambda_{x,y}$ is defined as:
\begin{equation}
J^\lambda_{x,y}:= 
\begin{cases}
\frac{1}{|x_1-y_1|^{\lambda}} \qquad & \text{ if } x_2=y_2, x_1\neq y_1,  \notag \\
\frac{1}{|x_2-y_2|^{\lambda}} \qquad & \text{ if } x_1=y_1, x_2\neq y_2,  \notag \\
0 \qquad & \text{ otherwise,}
\end{cases}
\end{equation}
where $x=(x_1,x_2)\in \Lambda$,  $y=(y_1,y_2)\in \Lambda$, and  with $\lambda>1$. \\

The energy excitation of a configuration with respect to $\textbf{-1}$ is:
\begin{align}\label{def:delta_H}
  \Delta H^\lambda(\sigma):=H^\lambda(\sigma)-H^\lambda(\textbf{-1}).  
\end{align}
In particular, $\Delta H^\lambda(\sigma)$ can be written in terms of the nonlocal discrete perimeter $Per_{\lambda}$ defined in (\ref{def:per}):

\begin{equation}
\label{delta_per}
\Delta H^\lambda(\sigma)=2Per_{\lambda}(\mathcal{P}[\text{Supp}(\sigma)])-2h|\text{Supp}(\sigma)|+\mathcal{O}(L^{-(\lambda-1)}),
\end{equation}
where we denote with $\mathcal{P}[A]$ the minimal (by inclusion) polyomino containing the set $A\subset\Lambda$. The correction of order $\mathcal{O}(L^{-(\lambda-1)})$ comes from the fact that we consider the two dimensional torus of side $L$ instead of $\mathbb{Z}^2$. In order to quantify this correction, let us consider a unit square polyomino $Q(x)$ centered at the site $x$ belonging to the torus $\Lambda=\{1,\ldots,L\}^2$. Hence: 
\begin{align}
    Per_\lambda(Q(x))=
    \begin{cases}
        4 \sum_{r=1}^{\frac{L-1}{2}} \frac{1}{r^\lambda} & \text { if } L \text{ is odd,} \notag \\
        4 \sum_{r=1}^{\frac{L-2}{2}} \frac{1}{r^\lambda} +\frac{2}{(L/2)^\lambda} & \text { if } L \text{ is even,}
    \end{cases}
\end{align}
where we have used \eqref{contribution_unit_square}.
Thus, we analyze the difference with an infinite volume in case $L$ even, the other is similar. Note that
\begin{align}
    4 \zeta(\lambda)- 4 \sum_{r=1}^{\frac{L-2}{2}} \frac{1}{r^\lambda}=
    4 \zeta(\lambda, L/2) < 4 \int_{\frac{L}{2}}^\infty \frac{1}{r^\lambda} dr = 4 \frac{\lambda-1}{(L/2)^{\lambda-1}}. \nonumber
\end{align}

Therefore, for sufficiently large lattice side $L$ the correction in Equation (\ref{delta_per}) is negligible. The parameter $\lambda$ in Equation (\ref{delta_per}) tunes the strength of the long-range interaction. In particular, in the limit $\lambda\to\infty$, we recover the short-range two dimensional Ising model. We have indeed:
\begin{equation}
\label{e:limit_energy}
\lim_{\lambda\to\infty} H^\lambda(\sigma)= H^\textnormal{n.n}(\sigma),
\nonumber
\end{equation}
for any $\sigma\in\mathcal{X}$.
Our main result, contained in Theorem~\ref{main_theorem}, fully characterizes the configurations of minimal energy on a set $\mathcal{V}_n$ with a fixed number $n$ of positive spins. 
By Theorem~\ref{main_theorem}, it follows that the minimizers $\bar{\sigma}_n\in \text{argmin}_{\sigma \in \mathcal{V}_n} \Delta H^\lambda(\sigma)$ are such that $\textnormal{Supp}(\bar{\sigma}_n)\in \mathscr{M}_n\cap \mathbb{Z}^2$, with $\mathscr{M}_n$ defined in  \eqref{minimizers1}.
Hence, according to the value of $n$, $\bar{\sigma}_n$ are droplets of pluses with a shape 
of a square, square with protuberance, quasi-squares,  quasi-square with protuberance, rectangles or rectangles with protuberance.

In particular, we have that:
\begin{align}
    \text{Supp}(\bar{\sigma}_n)=
    \begin{cases}
        \mathscr{Q}_l \qquad & \text{ if } n=l^2, \notag \\
        \mathscr{Q}_l^{k_1} \text{ or } \mathscr{R}^{k_2}_{a,b}  & \text{ if } n=l^2+k_1 = ab+k_2, \, a+b + \textbf{1}_{k+2\neq 0}=2l+1, \notag \\
        \mathcal{Q}_l^{k_1}  & \text{ if } n=l^2+k_1 = ab+k_2, \, a+b + \textbf{1}_{k+2\neq 0}>2l+1, \notag \\
       \mathscr{R}_{l,l+1} & \text{ if } n=l(l+1), \notag \\
        \mathscr{R}^{k_1}_{l,l+1}  \text{ or } \mathscr{R}^{k_2}_{a,b} & \text{ if } n=l(l+1)+k_1 = ab+k_2, \, a+b + \textbf{1}_{k+2\neq 0}=2l+2, \notag \\
        \mathscr{R}^{k_1}_{l,l+1}   & \text{ if } n=l(l+1)+k_1 = ab+k_2, \, a+b + \textbf{1}_{k+2\neq 0} > 2l+2. \notag
    \end{cases}
\end{align}
Then, $\Delta H^\lambda(\sigma)$ defined in Equation \eqref{delta_per} is equal to, depending on the cases,
\begin{align}
\label{e:delta}
    \Delta H^\lambda(\bar{\sigma}_n) = 
    \begin{cases}
8 l \sum_{i=1}^{l} \zeta(\lambda,i) -2hn & \text{ if } n=l^2, \\
8l \sum_{i=1}^{l} \zeta(\lambda,i) + 4\sum_{i=1}^l \zeta(\lambda,i) + 4l\zeta(\lambda,l+1) -2hn & \text{ if } n=l(l+1). 
\end{cases}
\end{align}
If $n=l^2+ k_1$ and $a+b +  \textbf{1}_{k_2\neq 0} =2l+ 1, \, ab + k_2 = l^2+k_1$, where $k_1\in \{1,\dots,l-1\}$ and $k_2\in \{0,\dots,a-1\}$.
\begin{align}
    \Delta H^\lambda(\bar{\sigma}_n)=\min \biggl \{ &8 l \sum_{i=1}^{l} \zeta(\lambda,i) +4k_1 \zeta(\lambda,l+1) +4\sum_{i=1}^{k_1}\zeta(\lambda,i) -2hn, \\
&4a \sum_{i=1}^{b} \zeta(\lambda,i) + 4b\sum_{i=1}^a \zeta(\lambda,i) +4k_2 \zeta(\lambda,b+1) +4\sum_{i=1}^{k_2}\zeta(\lambda,i) -2hn \biggr \}. 
\end{align}
If $n= l(l+1) +k_1$ and $a+b +\textbf{1}_{k_2\neq 0}=2l+ 2 $ , $ab + k_2=l(l+1)+k_1$, where $k_1\in \{1,\dots,l\}$ and $k_2\in \{0,\dots,a-1\}$
\begin{align}
 \Delta H^\lambda(\bar{\sigma}_n)=\min \biggl \{&8l \sum_{i=1}^{l} \zeta(\lambda,i) + 4\sum_{i=1}^{l} \zeta(\lambda,i) + 4l\zeta(\lambda,l+1) + 4k_1 \zeta(\lambda, l+2) +4\sum_{i=1}^{k_1}\zeta(\lambda,i) -2hn, \\
&4a \sum_{i=1}^{b} \zeta(\lambda,i) + 4b\sum_{i=1}^a \zeta(\lambda,i) +4k_2 \zeta(\lambda,b+1) +4\sum_{i=1}^{k_2}\zeta(\lambda,i) -2hn \biggr \}.   
\end{align}
In Figure~\ref{fig:long_2_4}
we have reported the energy excitation  $\Delta H^\lambda(\bar{\sigma}_n)$ for the minimizers $\bar{\sigma}_n$ on $\mathcal{V}_n$ as function of the number of positive spins $n$, for a particular value of the external magnetic field and of the parameter $\lambda$ (e.g., $h=0.41,\lambda=2.4$). The  vertical red line denotes the number $n_c$ corresponding to the minimizer $\bar{\sigma}_{n_c}$ with the maximal energy. Therefore, in analogy with the short-range model, we expect  the configuration $\bar{\sigma}_{n_c}$ to be the \emph{critical droplet} of the nucleation process, i.e., the configuration attaining the mini-max:
$$
\Phi^{\lambda}(\textbf{-1},\textbf{+1}):=
\min_{\omega:\textbf{-1}\to\textbf{+1}}\max_{\eta\in\omega} H^{\lambda}(\eta),
$$
with $\lambda=2.4$.
However, the differences with the short range model are evident, with an increment of the dimension $l_c$ of the critical configuration (\emph{critical length}). 
As expected, the critical length  for a fixed value of the magnetic field is a decreasing function of parameter $\lambda$, as we see by comparing Figure~\ref{fig:long_1_8} with Figure~\ref{fig:long_2_4}: we can see indeed that $l_c(\lambda=2.4;h=0.41)=14<l_c(\lambda=1.8;h=0.41)=62$.
Furthermore, in Figure~\ref{surface} we have plotted the surface of $\Delta H^\lambda(\bar{\sigma}_n)$ as a function of the two parametere $\lambda$ and $n$, for a given value of the magnetic field $h=0.41$. The impact of the long-range interaction is evident for values of $\lambda$ smaller than 3.

\begin{figure}[htb]
\begin{center}
\includegraphics[scale=0.4]{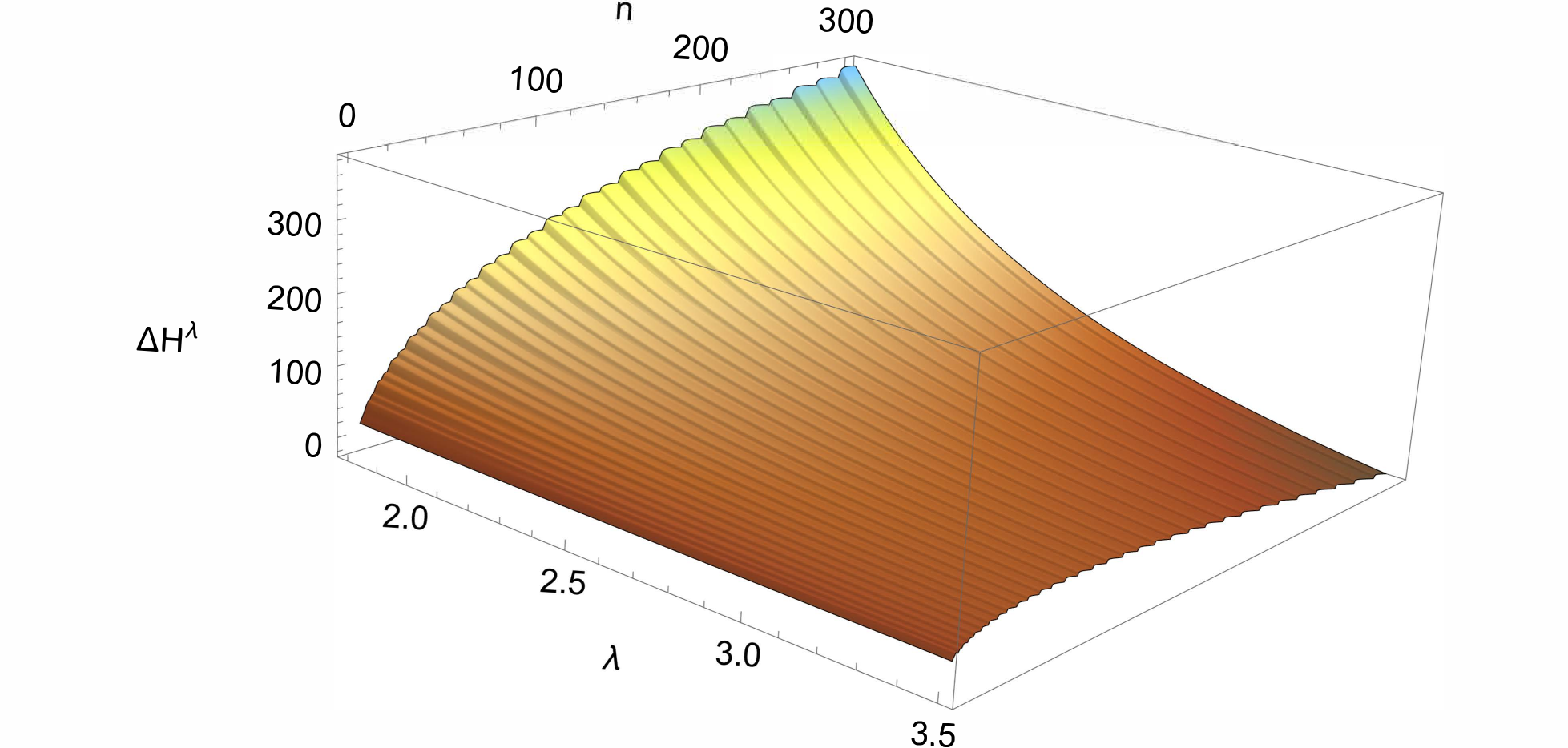} 
    \caption{Surface plot of $\Delta H^\lambda(\bar{\sigma}_n)$ for $h=0.41$.}
 \label{surface}
\end{center}
\end{figure}

\begin{figure}[htb]
\begin{center}
    \includegraphics[scale=0.35]{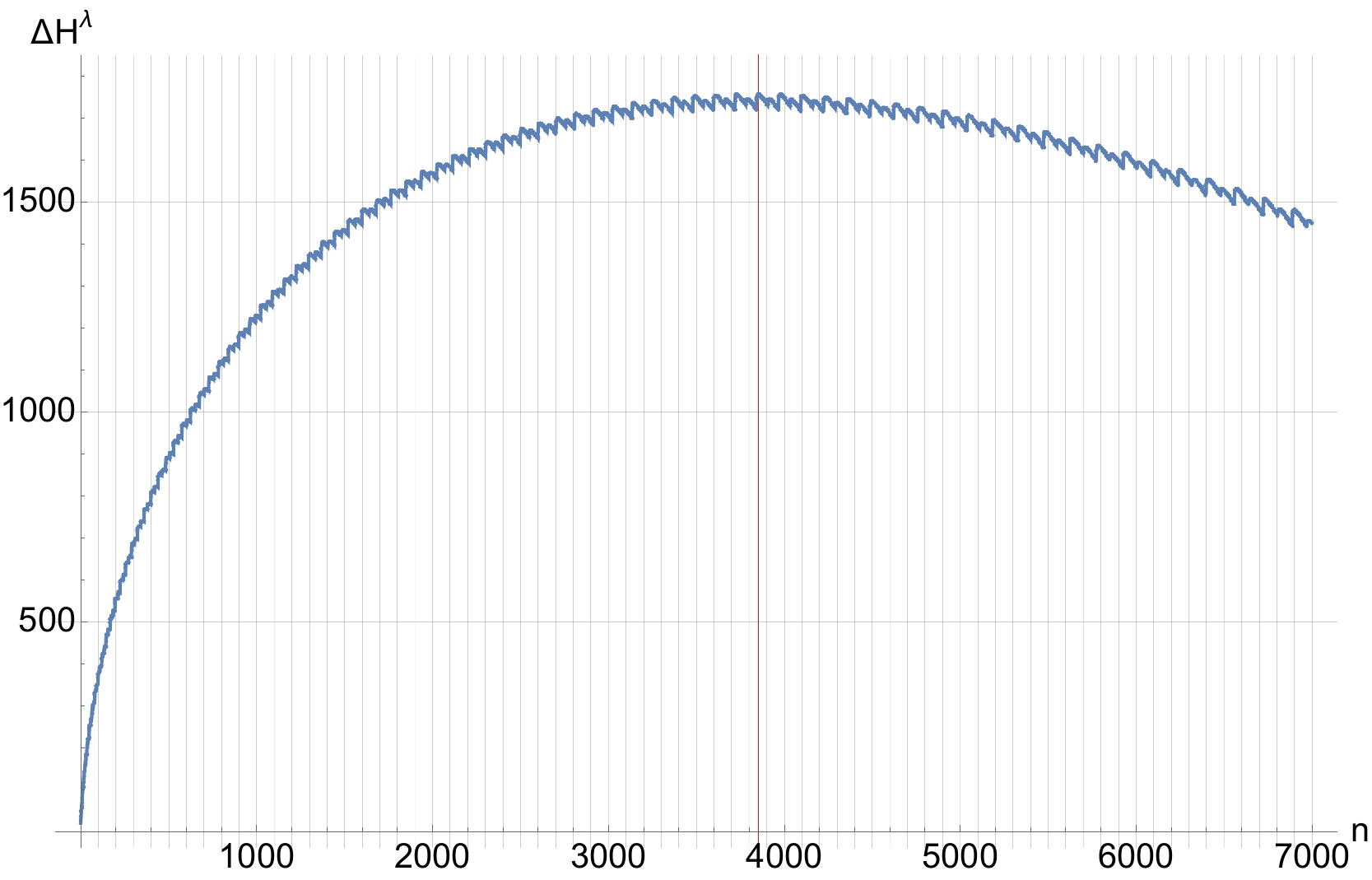} 
    \caption{ Plot of  $\Delta H^\lambda(\bar{\sigma}_n)$ vs. $n$, for the long-range bi-axial 2d Ising model with $h=0.41$, $\lambda=1.8$.}
 \label{fig:long_1_8}
\end{center}
    \end{figure}

\begin{figure}[htb]
\begin{center}
    \includegraphics[scale=0.35]{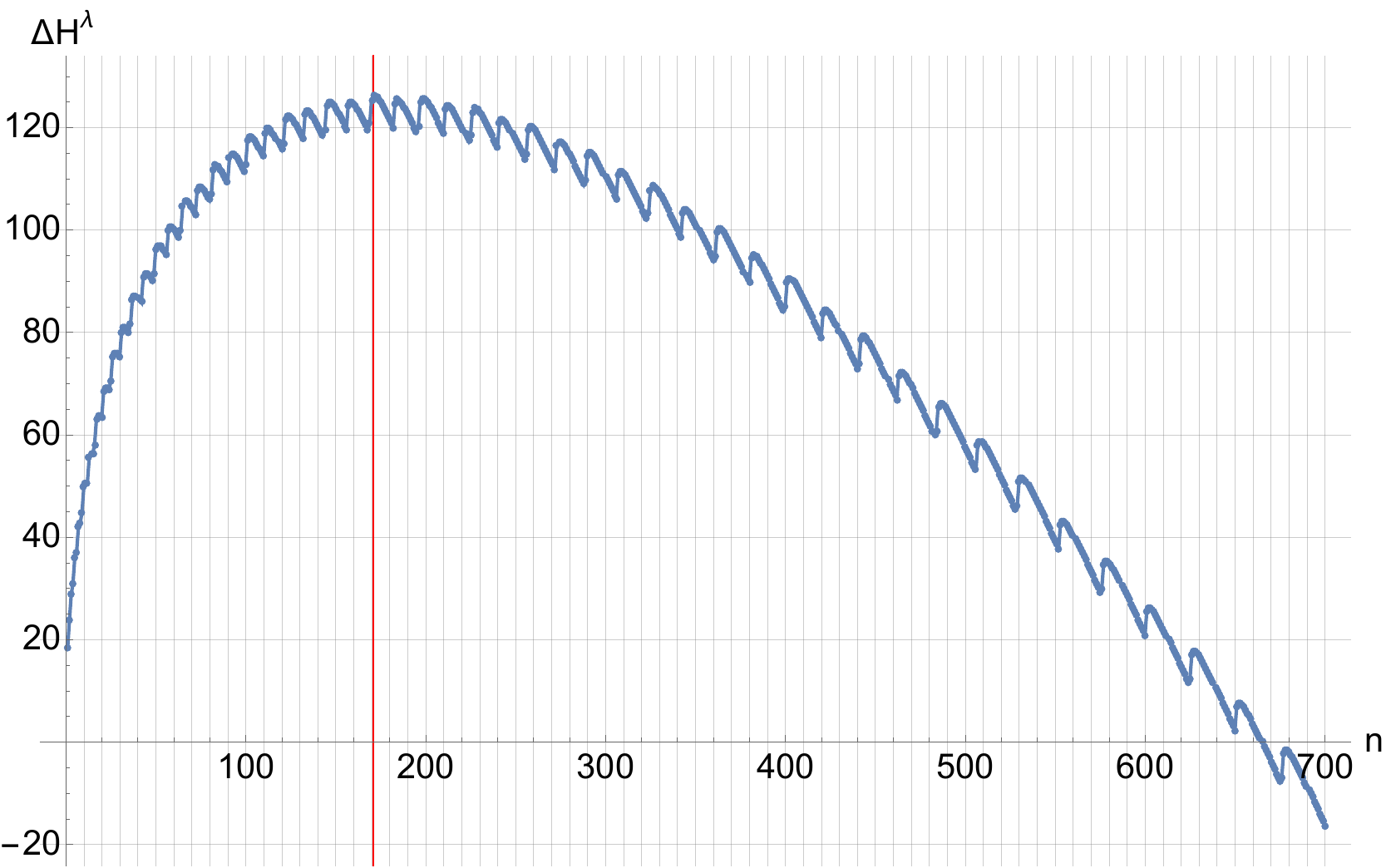} 
    \caption{ Plot of  $\Delta H^\lambda(\bar{\sigma}_n)$ vs. $n$, for the long-range bi-axial 2d Ising model with $h=0.41$, $\lambda=2.4$.}
 \label{fig:long_2_4}
\end{center}
    \end{figure}

Despite in the short-range Ising model the minimizers on the foliations $\mathcal{V}_n$ and $\mathcal{V}_{n+1}$ can be chosen to be communicating configurations, this is not the general case for the long-range bi-axial model. In fact, in our model the protuberance grows on the smaller side, differently from the short-range Ising model. In Figure~\ref{fig:square32} on the left there is the energy minimizer $\hat{\sigma}_{32}$ in the foliation $\mathcal{V}_{32}$, i.e. a  quasi-square of side $l=5$. On the right, we see the configuration $\eta$, i.e., a  quasi-square with a double protuberance on the shorter side. We have that $H^\lambda(\eta)-H^\lambda(\bar{\sigma}_{32})=1/l^\lambda -1/(l+1)^\lambda>0$. As expected, this difference tends to zero as $\lambda\to\infty$, since in the short-range Ising model these two configurations have the same energy.
\begin{figure}[hbt]
\begin{center}
\includegraphics[scale=0.32]{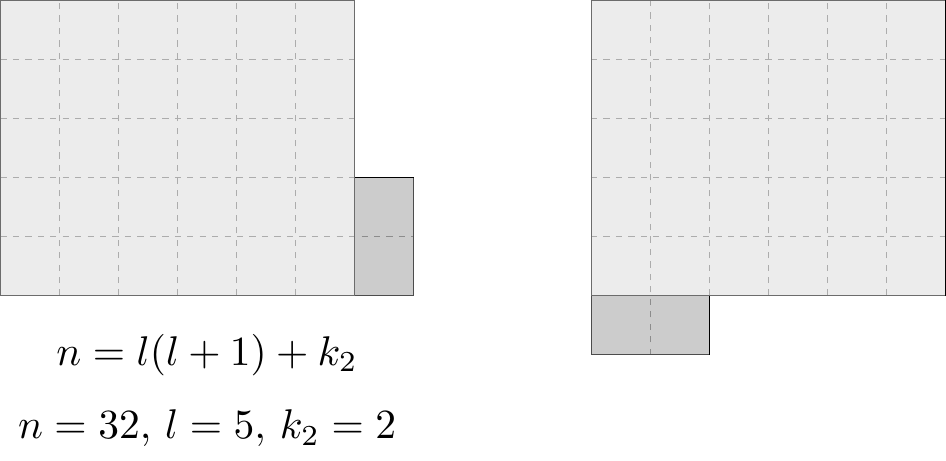} 
    \caption{Anisotropy in the biaxial model: quasi-squares configurations with $l=5$. On the left the energy minimizer $\hat{\sigma}_{32}$ in $\mathcal{V}_{32}$, and on the right a configuration $\eta\in\mathcal{V}_{32}$ with $H^\lambda(\eta)>H^\lambda(\bar{\sigma}_{32}$). }
 \label{fig:square32}
\end{center}
    \end{figure}

%%%%%%%%%%%%%%%%%%%%%%%%%%%%
\subsection{Critical length of the square droplets}
\label{s:critical_length}
By Equation \eqref{e:delta}, if we restrict our attention to the case $n=l^2$, then the minimizer of the nonlocal discrete perimeter is a square droplet of size $l$, so that:
\begin{align}
\label{e:square}
    f^\lambda_h(l):=\Delta H^\lambda(\bar{\sigma}_{l^2} ) %=-2hl^2 +8l^2 \zeta(\lambda) +8l \sum_{j=1}^{l-1} \frac{1}{j^{\lambda -1}} -8l^2 \sum_{j=1}^{l-1} \frac{1}{j^{\lambda}} \\
    &=-2hl^2 +8l^2 \zeta(\lambda, l) +8l \sum_{j=1}^{l-1} \frac{1}{j^{\lambda -1}}. 
\end{align}
We can look at the maximum in $l$ of the function $f^\lambda_h(l)$ defined in Equation~ \eqref{e:square}, in order to have some information about the scaling of the critical length of the minimal saddle configuration of the first excursion from \textbf{-1} to \textbf{+1}.With an abuse of language, we also call this global maximum \emph{critical length}.  
In Figure~\ref{f:lambda_crit} the critical length $l_c$ are reported as function of the exponent $\lambda$ and of the external magnetic field $h$.
\begin{figure}[htb]
\begin{center}
\includegraphics[scale=0.5]{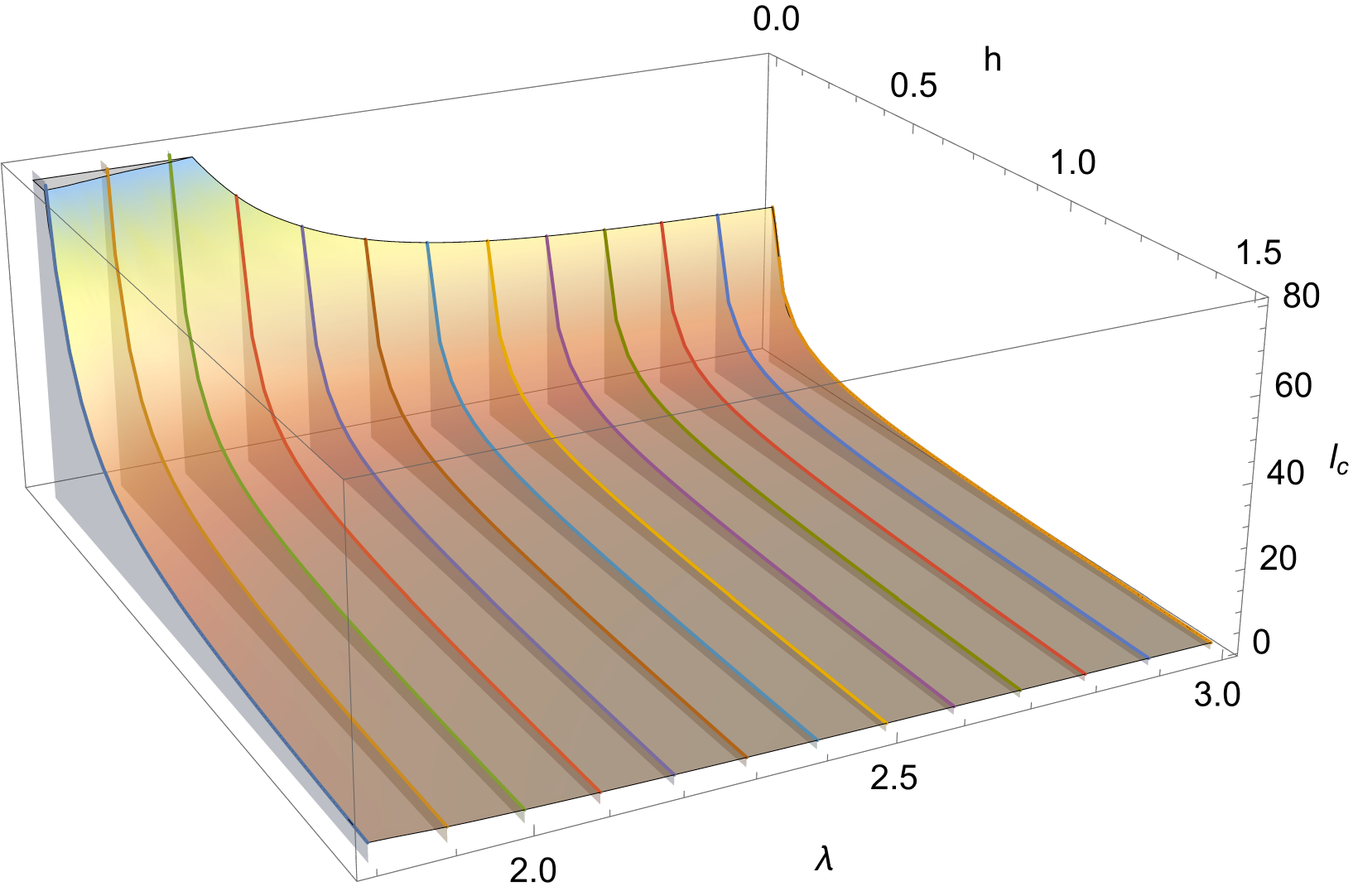} 
    \caption{Surface plot of the $l_c$ as function of $h$ and $\lambda$.}
 \label{f:lambda_crit}
\end{center}
\end{figure}

 Obtaining the scaling of $l_c$ with respect to $\lambda$ and $h$ is not analytically feasible, so that the critical length $l_c$ can only be estimated numerically. In order to visualize the effect of the long-range, we can derive the second order derivative of $f^\lambda_h$ with respect to $l$.
  We know indeed that for the short-range Ising model $$\frac{\partial^2}{\partial l^2}H^{\textnormal{n.n}}(\bar{\sigma}_n)=-4h,
$$
so that, for a fixed value of $h$, we expect that a plateau of the second order derivative will indicate a short-range type of behavior of the long-range model.\\
By recalling the definition of the Hurwitz-zeta function given in Equation \eqref{e:zeta}, we have that $\zeta(\lambda,1)=\zeta(\lambda)$, where we denoted with $\zeta(\lambda)$ the Riemann zeta function. Moreover, we have that:
\begin{equation}
\label{e:hur_first}
\frac{\partial}{\partial l}\zeta(\lambda,l)=-\lambda \zeta(\lambda+1,l).
\end{equation}
 By \eqref{eq:forgen}, $f^\lambda_h$ can be rewritten as:
\begin{equation}
    \label{e:frew}
    f^\lambda_h(l)=-2 h l^2 +8l^2 \zeta(\lambda,l) +8l (\zeta(\lambda-1)-\zeta(\lambda-1,l),
\end{equation}
for  any $\lambda>2$.
Its partial derivative with respect to $l$ can be written as:
\begin{eqnarray*}
\frac{\partial}{\partial l}
f^\lambda_h(l)&=&-4hl+16l \zeta(\lambda,l)-8l^2\lambda \zeta(\lambda+1,l)+8\left(\zeta(\lambda-1)-\zeta(\lambda-1,l)  \right)+8(\lambda-1)l \zeta(\lambda,l) \nonumber\\
&=&-4hl-8l^2\lambda \zeta(\lambda+1,l)+ 8 l (\lambda+1) \zeta(\lambda,l)- 8\zeta(\lambda-1,l)+8\zeta(\lambda-1).
\end{eqnarray*}
Analogously, we calculate the second order partial derivative of $f^\lambda_h$ with respect to $l$:
\begin{eqnarray}
\label{e:der2}
\frac{\partial^2}{\partial l^2}
f^\lambda_h(l)&=&-4h-8l\lambda\left(\lambda +3\right) \zeta(\lambda+1,l) +8\lambda^2 l^2   \zeta(\lambda+2,l) +8(\lambda+2) \zeta(\lambda,l).
\end{eqnarray}
In Figure~\ref{f:second_order} we reported the surface plot of the second order derivative of $f^\lambda_h$ with respect to $l$, with $h=0.4$, and for $\lambda\in[2,4]$. We can see that in this parameter space $\frac{\partial^2}{\partial l^2}
f^\lambda_h(l)<0$. In particular, the flat plateau with height $-1.6$ corresponds to a behavior similar to the short-range Ising model, for which:
$$\frac{\partial^2}{\partial l^2}H^{\textnormal{n.n}}(\bar{\sigma}_n)=-4h=-1.6.
$$
 Only in the region $\lambda<3$ the effect of the nonlocal interaction is pronounced.
\begin{figure}[htb!]
\begin{center}
\includegraphics[scale=0.5]{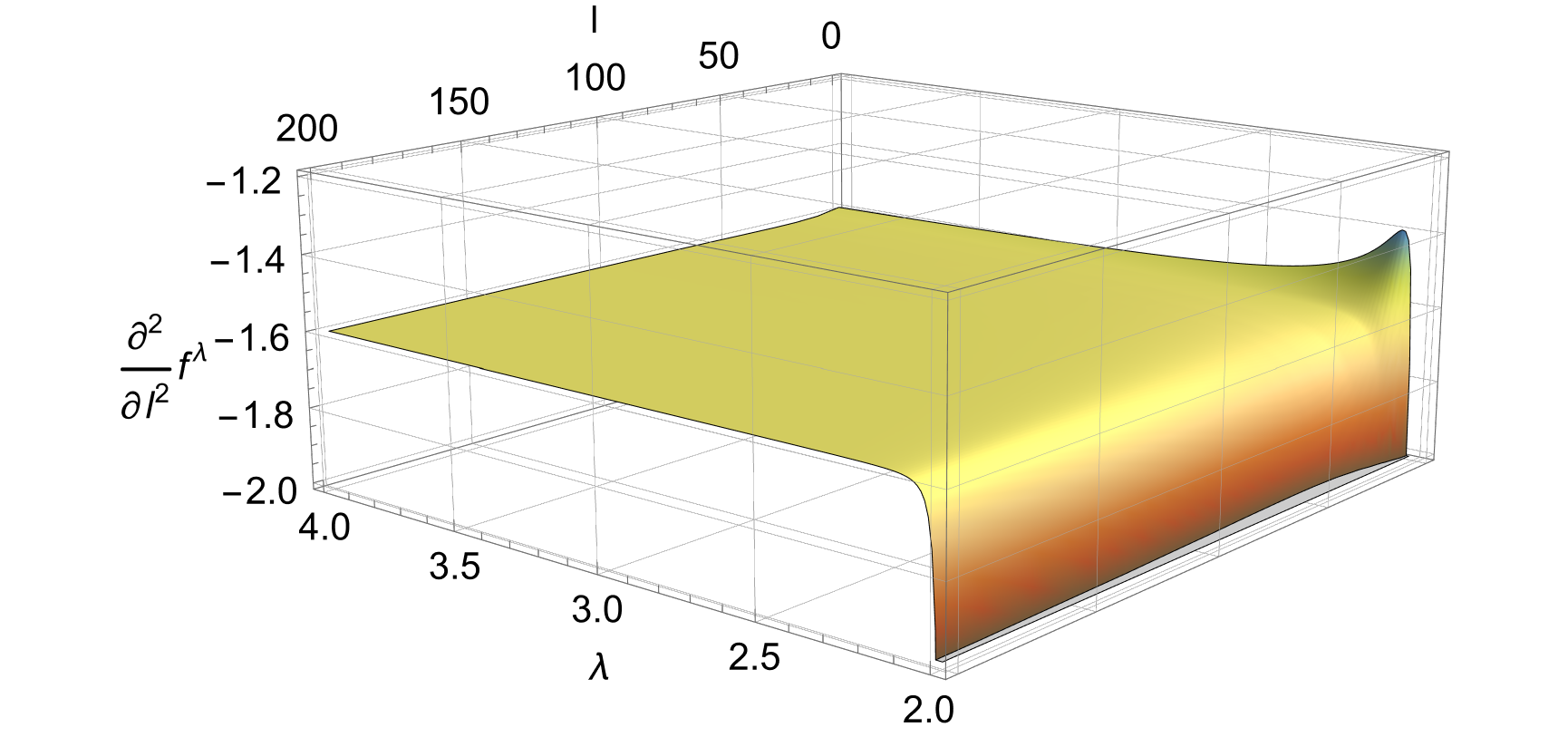} 
    \caption{Surface plot of $\frac{\partial^2}{\partial l^2}
f^\lambda_h(l)$ as function of $l$ and $\lambda$, with $h=0.4$.}
 \label{f:second_order}
\end{center}
\end{figure}

\subsection{Rigorous results for the metastable behaviour}\label{sec:rigorous} 

To identify $\textbf{-1}$ as the unique metastable state and to obtain an estimate of the transition time between $\textbf{-1}$ and the stable state $\textbf{+1}$, it is necessary to determine the maximal energy barrier separating $\textbf{-1}$ from $\textbf{+1}$ and to prove that any other configuration has stability level strictly smaller than $\textbf{-1}$ and $\textbf{+1}$. 

Formally, one should prove the corresponding result for the one-dimensional
long-range Ising model, i.e., Theorem~3.1 in~\cite{vanenter2019}, in our setting. In
this way, we identify the metastable state $\textbf{-1}$ and then, following
\cite{manzo2004}, the transition time can be estimated as
\begin{equation}\label{eq:ET}
\lim_{\beta\to\infty} \frac{1}{\beta}\ln(\mathbb{E}_{\textbf{-1}}(\tau_+))=\Phi^\lambda(\textbf{-1},\textbf{+1})-H^\lambda(\textbf{-1}),
\end{equation}
with 
$$
\Phi^{\lambda}(\textbf{-1},\textbf{+1}):=
\min_{\omega:\textbf{-1}\to\textbf{+1}}\max_{\eta\in\omega} H^{\lambda}(\eta).
$$
In particular, in order to prove (\ref{eq:ET}) we need to solve highly non trivial
\emph{model-dependent} problems, i.e., the following restatement of Theorem~3.1 in \cite{vanenter2019}:
\begin{theorem}\label{t:3.1}
We have:
\begin{enumerate}
  \item $\Phi^\lambda(\mathbf{-1},\mathbf{+1}) = \Gamma + H^{\lambda}(\mathbf{-1}),$
  \item $V_{\mathbf{-1}} = \Gamma > 0,$ and
  \item $V_{\sigma} < \Gamma \quad \text{for any } \sigma \in \mathcal{X} \setminus \{\mathbf{-1},\mathbf{+1}\},$
\end{enumerate}
where $\Gamma:=\sup_{\sigma\in\mathcal{X}\setminus\mathbf{+1}} V_\sigma$ and $V_\sigma$ is the \emph{stability level} of the configuration $\sigma$ defined as
$V_\sigma := \Phi^\lambda(\sigma,\mathcal{I}_\sigma) - H^{\lambda}(\sigma) \ge 0$, 
with $\mathcal{I}_\sigma$ the set of all states $\eta\in\mathcal{X}$  such that
$H^{\lambda}(\eta) < H^{\lambda}(\sigma)$.
\end{theorem}

The determination of the min--max energy
barrier $\Phi^\lambda(\mathbf{-1},\mathbf{+1})$ relies on the identification of the maximal energy along optimal paths, which in turn depends on the precise description of the minimal-energy configurations within each foliation of fixed magnetization. The solution of this variational problems  corresponds to point~(1) of Theorem~\ref{t:3.1}. 
In the present two-dimensional setting, Theorem~\ref{main_theorem} identifies, through the solution of the discrete isoperimetric problem, the relevant minimal-energy representatives of each magnetization level and thus the natural candidates for the critical saddle configuration. 
This provides a lower bound for $\Phi^\lambda(\mathbf{-1},\mathbf{+1})$, but in order to obtain a matching upper bound, one needs to exhibit a reference path starting from 
$\textbf{-1}$ and reaching $\textbf{+1}$, not overcoming the energy $\Gamma$. The description of such a path, as a sequence of configurations differing by a single spin (in accordance with the Glauber dynamics), depends heavily on the choice of $\lambda$, which determines the shape of the clusters in the critical configuration. For certain values of $\lambda$, this can generate a large number of possible configuration sequences to analyze and the optimal path may not pass through the foliation of minimal-energy saddle configurations. 
%As a result, the upper bound may not coincide exactly with the lower bound. To compensate for this discrepancy, one would then need to prove that all paths from  $\textbf{-1}$ to  $\textbf{+1}$ reach the energy barrier value given by the upper bound. 
This makes it extremely difficult to solve $\Phi^\lambda(\mathbf{-1},\mathbf{+1})$.

Moreover, to complete the study of metastability, one must verify points (2) and (3) of Theorem~\ref{t:3.1}, which require showing that the stability level of the metastable state is maximal and precisely equal to $\Phi^\lambda(\mathbf{-1},\mathbf{+1})-H^\lambda(\mathbf{-1})$. These properties are known in the literature as \emph{recurrence properties}. The proof of Theorem~\ref{main_theorem} already provides a key ingredient for the recurrence step: it gives an explicit procedure to associate with any polyomino not belonging to the extremal set $\mathscr{M}_n$ a polyomino of the same area but with a strictly smaller nonlocal perimeter (see Figure~\ref{fig:diagramm}). In the language of the bi-axial long-range Ising model, this corresponds to a recurrence property for all configurations outside $\mathscr{M}_n$: for any configuration $\sigma$ containing a cluster whose shape differs from those in $\mathscr{M}_n$ we have that $V_\sigma< \Phi^\lambda(\mathbf{-1},\mathbf{+1})-H^\lambda(\sigma)$. What remains to be proven, therefore, is a recurrence property for configurations containing clusters with the shape of the polyominoes in $\mathscr{M}_n$, as well as the construction of a reference path to establish that the stability level of $\mathbf{-1}$ is indeed $\Phi^\lambda(\mathbf{-1},\mathbf{+1})-H^\lambda(\sigma)$. The proof of the recurrence property for $\mathscr{M}_n$ is standard: one expects that small rectangles (or squares) of minus spins tend to shrink, while large rectangles tend to grow. On the other hand, the construction of the reference path, as mentioned before, remains genuinely difficult.

%%%%%%%%%%%%%%%%%%%%%%%%%%%%%%%%%%%%%%%%%%%%%%%%%%%%%%%%%
\section{Conclusion and future directions}
\label{s:conclusion}
%%%%%%%%%%%%%%%%%%
In the present manuscript we have addressed and solved for the first time a nonlocal discrete isoperimetric problem. We considered indeed what we have called a \emph{nonlocal bi-axial discrete perimeter}
for a general two-dimensional polyomino $\mathcal{P}$, and we have then found and characterized its minimizers in the class of polyominoes with fixed area. A first natural generalization of our results is to consider different nonlocal kernels (e.g., Euclidean distance with power-law/ exponential decay) or other geometries (three-dimensional lattices or more general graph structures). 

Moreover, in our article we explained  how the solution of the nonlocal discrete isoperimetric problem is related to the rigorous study of the metastable behavior of a \emph{long-range bi-axial Ising model}. The next step is to rigorously solve the metastability problem for this model. 
We remark that our algorithmic proof of Proposition~\ref{prop:no_rectangle} can already be used for proving  a recurrence property of the process on
the set of states at given stability level, in order to control the tail of the tunneling time (see for instance Theorem~3.1 in \cite{manzo2004}).
Thanks to Proposition~\ref{prop:no_rectangle} we were able indeed to associate to any configuration not in $\mathscr{M}^{\text{ext}}_n$ a configuration in $\mathscr{M}_n^{\text{ext}}$ with smaller nonlocal perimeter and with the same area. In the language of the bi-axial long range Ising model, this corresponds to associate to any configuration not belonging to the extremal set, an extremal configuration, with lower energy and same magnetization. 
With the remaining propositions we associated to any configuration in $\mathscr{M}_n^{\text{ext}}$ a configuration in $\mathscr{M}_n$ with strictly smaller nonlocal perimeter and with the same area.
The next challenge will be the characterization of the critical droplets and of the optimal nucleation path. 

\begin{appendix}\label{s:appendix}
\section{Elementary relations}
 Let $A,B\geq 1, B\geq A$ be positive integer number, then 
\begin{equation}\label{eq:forgen}
\sum_{i=A}^B \zeta(\lambda, i) = (B-(A-1))\zeta(\lambda) - \sum_{k=1}^{B-1} \frac{B-k}{k^\lambda} + \sum_{k=1}^{A-2}\frac{(A-1) -k}{k^{\lambda}}.
\end{equation}
In particular, for $A=1$ we find
\begin{equation}\label{eq:forA=1}
\sum_{i=1}^B \zeta(\lambda, i)  = B\zeta(\lambda) - \sum_{k=1}^{B-1} \frac{B-k}{k^\lambda}. 
\end{equation}
Note that for every $C\geq 1$ we have trivially that
\begin{equation}\label{eq:boundsum}
\sum_{i=A+C}^{B+C} \zeta(\lambda, i)  = \sum_{i=A}^B \zeta(\lambda, i)  - \sum_{i=A}^{B} \sum_{k=i}^{C-1+i} \frac{1}{k^{\lambda}}.
\end{equation}
Throughout the manuscript, when we write $\sum_{r=x}^y f(r)$ with $y<x$, it means that the sum is equal to 0.\\

%%%%%%%%%%%%%%%%%%%%%%%%%%%%%%%%%%%%%%%%%%%%%%%%%%%%%%%%

\section{Proof of Proposition \ref{prop:rectangle_quasisquare}}
\label{proof:prop3.3}
Assume $n=l(l+1)$.
Suppose without loss of generality that $a$ and $l$ (resp. $b$ and $l+1$) are the lengths of the vertical (resp. horizontal) sides of $\mathscr{R}_{a,b}$ and $\mathscr{R}_{l,l+1}$. By Equation \eqref{eq:contribution_strips}, we have 
\begin{align}\label{eq:quasi_square_rectangle_GAP}
   &Per_{\lambda}(\mathscr{R}_{l,l+1})=2l \sum_{i=1}^{l+1} \zeta(\lambda, i)  +2(l+1) \sum_{i=1}^l \zeta(\lambda, i)   
   \,\,\,\,\, \text{ and } \,\,\,\,\,
   Per_{\lambda}(\mathscr{R}_{a,b})=2a \sum_{i=1}^b \zeta(\lambda, i)   +2b \sum_{i=1}^a \zeta(\lambda, i). 
\end{align}
By \eqref{eq:forA=1} and \eqref{eq:quasi_square_rectangle_GAP}, we have
\begin{align}
  &\frac{Per_{\lambda}(\mathscr{R}_{l,l+1})}{2}=2l(l+1)\zeta(\lambda) - l\sum_{k=1}^{l} \frac{l+1-k}{k^\lambda} - (l+1)\sum_{k=1}^{l-1} \frac{l-k}{k^\lambda}, \label{gapquasisquarell+1} \\
  &\frac{Per_{\lambda}(\mathscr{R}_{a,b})}{2}=2ab\zeta(\lambda) - a\sum_{k=1}^{b-1} \frac{b-k}{k^\lambda} - b\sum_{k=1}^{a-1} \frac{a-k}{k^\lambda}. \label{rectanglell+1}
\end{align}
Thus, recalling $l(l+1)=ab$, we obtain
\begin{align}\label{deltall+1}
\Delta := \frac{Per_{\lambda}(\mathscr{R}_{a,b})-Per_{\lambda}(\mathscr{Q}_l)}{2}&
=l\sum_{k=1}^{l} \frac{l+1-k}{k^\lambda} + (l+1)\sum_{k=1}^{l-1} \frac{l-k}{k^\lambda}
-a \sum_{k=1}^{b-1} \frac{b-k}{k^{\lambda}} -b \sum_{k=1}^{a-1} \frac{a-k}{k^{\lambda}} .
\end{align}
We will first treat the cases $l=2,3$ by hand and assume in the following that $l\geq 4$. Let first $l=2$, then
\[
\Delta = 2 \left( 2 + \frac{1}{2^{\lambda}} \right) + 3 - \left ( 5+ \frac{4}{2^\lambda} + \frac{3}{3^\lambda} + \frac{2}{4^\lambda} +\frac{1}{5^\lambda} \right )= 2-\frac{2}{2^{\lambda}} -\frac{3}{3^{\lambda}} -\frac{2}{4^{\lambda}}-\frac{1}{5^\lambda}> 0
\]
for each $\lambda\geq 1.3$. When $l=3$ we have either $a=1$ and $b=12$ or $a=2$ and $b=6$. We compute $\Delta$ in the first case
\[
\Delta = 6 - \frac{6}{3^\lambda} - \left ( \frac{8}{4^\lambda} + \frac{7}{5^\lambda} + \frac{6}{6^\lambda} 
+ \frac{5}{7^\lambda} + \frac{4}{8^\lambda} + \frac{3}{9^\lambda} + \frac{2}{10^\lambda} + \frac{1}{11^\lambda}\right )
> 6 - \frac{6}{3^\lambda} - \frac{36}{4^\lambda} >0
\]
for $\lambda \geq 1.5$. In the second case, we have
\[
\Delta = 11 + \frac{2}{2^\lambda} - \frac{3}{3^\lambda} - \frac{4}{4^\lambda} - \frac{2}{5^\lambda}>0
\]
for each $\lambda>0$. 

For the general case, recalling that $b=\frac{l(l+1)}{a}$, it is easy to see that we can write the difference $\Delta$ as the sum of two functions $f_1(a,l)$ and $f_2(a,l,b)$, i.e. 
\begin{equation*}
\Delta = f_1(a,l) + f_2(a,l,b),
\end{equation*}
where 
\begin{align}
&f_1(a,l)=\left (a+\frac{l(l+1)}{a}-2l-1 \right )\sum_{k=1}^{a-1} \frac{1}{k^{\lambda-1}}
-(l-a)\sum_{k=a}^{l} \frac{1}{k^{\lambda-1}} 
-(l+1-a)\sum_{k=a}^{l-1} \frac{1}{k^{\lambda-1}}
+\frac{l(l+1)-a^2}{a^\lambda} \label{def_f1ll+1}, \\
&f_2(a,l,b)=\sum_{k=a+1}^{l} \frac{l(l+1)-ak}{k^{\lambda}}-\sum_{k=l}^{b-1} \frac{l(l+1)-ak}{k^{\lambda}}. \label{deff2ll+1}
\end{align}
We conclude the claim once we showed that both functions are strictly positive.

\noindent
\textbf{Step 1:}
We will first show that $l \mapsto f_1(a,l)$ is increasing in $l$ and in particular that $f_1(a,l)>0$. Let us write again:
\begin{equation}\label{def_rest_ll+1}
\begin{split}
f_1(a,l+1) &=  f_1(a,l) + \left (\frac{2(l-a)+2}{a} \right ) \sum_{k=1}^{a-1} \frac{1}{k^{\lambda-1}} 
-(l-a+1) \left ( \frac{1}{l^{\lambda-1}} + \frac{1}{(l+1)^{\lambda-1}} \right ) -2\sum_{k=a}^l \frac{1}{k^{\lambda-1}}+ \frac{2(l+1)}{a^\lambda} \\
& := f_1(a,l) + rest(a,l).
\end{split}
\end{equation}
The function is increasing in $l$ if the rest-term $rest(a,l)>0$ for all $l \geq \max \{a + 1, 4\}$. 
We compute further
\begin{equation*}
\begin{split}
rest(a,l+1) - rest(a,l) =& \frac{2}{a}\sum_{k=1}^{a}\frac{1}{k^{\lambda-1}}
-\left ( \frac{3}{(l+1)^{\lambda-1}} +\frac{1}{(l+2)^{\lambda-1}} \right )
+ (l-a+1) \left ( \frac{1}{l^{\lambda-1}} - \frac{1}{(l+2)^{\lambda-1}} \right ). 
\end{split}
\end{equation*}
In easy to see that the last expression is positive for $a \in \{1,2,3,4\}$ and $l \geq \max \{a + 1, 4\}$. Thus, suppose $a\geq 5$ and observe that $\max \{a + 1, 4\}=a+1$ and
\begin{align*}
&\frac{2}{a}\sum_{k=1}^{a}\frac{1}{k^{\lambda-1}}
-\left ( \frac{3}{(l+1)^{\lambda-1}} +\frac{1}{(l+2)^{\lambda-1}} \right )
+  \left ( \frac{(l-a+1)}{l^{\lambda-1}} - \frac{(l-a+1)}{(l+1)^{\lambda-1}} \right )+  \left ( \frac{(l-a+1)}{(l+1)^{\lambda-1}} - \frac{(l-a+1)}{(l+2)^{\lambda-1}} \right )\notag \\
&> \frac{2}{a}\sum_{k=1}^{a}\frac{1}{k^{\lambda-1}}
-\left ( \frac{3}{(l+1)^{\lambda-1}} +\frac{1}{(l+2)^{\lambda-1}} \right )
+ 2(l-a+1) \left ( \frac{1}{(l+1)^{\lambda-1}} - \frac{1}{(l+2)^{\lambda-1}} \right )\notag \\
& \geq   \frac{2}{a}\sum_{k=1}^{a}\frac{1}{k^{\lambda-1}} -\frac{4}{(l+1)^{\lambda-1}} 
\geq \frac{1}{2}\sum_{k=1}^{4}\frac{1}{k^{\lambda-1}} - \frac{4}{6^{\lambda-1}}>0,
\end{align*}
since $l \geq a+1$, $a \geq 4$ and $\lambda>1.8$. We have proven that $rest(a,l)$ is an increasing function of $l$ for each fixed $a$ and $l \geq a + 1$. It remains to show that $rest(a, \max \{a + 1,4 \}) > 0$. For $a \geq 4$, we have $\max \{a + 1,4 \}=a+1$ and
\begin{equation}\label{eq:wasch}
\begin{split}
rest(a, a + 1) &= \frac{4}{a} \sum_{k=1}^{a-1} \frac{1}{k^{\lambda-1}} 
-\frac{4}{(a+1)^{\lambda-1}} -\frac{2}{(a+2)^{\lambda-1}} + \frac{4}{a^\lambda},
\end{split}
\end{equation}
which is greater than
\begin{equation}\label{eq:wasch1}
\eqref{eq:wasch}\ > \frac{4}{a} \sum_{k=1}^{a-1} \frac{1}{k^{\lambda-1}} - \frac{6a+4}{(a+1)^{\lambda}} .
\end{equation}
Lower bounding \eqref{eq:wasch1} by

\begin{equation*}
\sum_{k=1}^{a-1} \frac{1}{k^{\lambda-1}}-\frac{3a^2+2a}{2(a+1)^{\lambda}} \geq \int_{1}^{a} \frac{1}{k^{\lambda-1}} dk-\frac{3a^2+2a}{2(a+1)^{\lambda}}
=\frac{a^{2-\lambda}-1}{2-\lambda}-\frac{3a^2+2a}{2(a+1)^{\lambda}} >0
\end{equation*}
since
\begin{equation*}
2(a+1)^{\lambda}(a^{2-\lambda} -1) > (2-\lambda) (3a^2+2a),
\end{equation*}

which is true for $a\geq 4$ and $1.8 \leq \lambda < 2$. The other case for $\lambda \geq 2$ follows in a similar manner. 
Next, we suppose $a \leq 3$ and we will prove $rest(a,4)>0$. Indeed, 
\begin{align*}
    rest(a,4)= \left (\frac{10-2a}{a} \right ) \sum_{k=1}^{a-1} \frac{1}{k^{\lambda-1}} 
-(5-a) \left ( \frac{1}{4^{\lambda-1}} + \frac{1}{5^{\lambda-1}} \right ) -2\sum_{k=a}^4 \frac{1}{k^{\lambda-1}}+ \frac{10}{a^\lambda}.
\end{align*}
For $a=1$, we obtain
\begin{align*}
   rest(1,4)= -4\left ( \frac{1}{4^{\lambda-1}} + \frac{1}{5^{\lambda-1}} \right ) -2\sum_{k=1}^4 \frac{1}{k^{\lambda-1}}+ 10>0,
\end{align*}
since $\lambda>1.8$. The other cases $a=2,3$ follow in a similar manner.

We have proven that $f_1(a,l)$ is an increasing function in $l$ for all $l\geq a+1$. Let us prove now that $f_1(a,l)>0$ by proving that $f_1(a,a+1)>0$. Indeed, we have that
\begin{align*}
f_1(a,a+1) & =\frac{2}{a}\sum_{k=1}^{a-1} \frac{1}{k^{\lambda-1}}
+\frac{2}{a^\lambda}
-\frac{1}{(a+1)^{\lambda-1}}. 
\end{align*}
In easy to see that for $a=1$, we obtain $f_1(1,2)>0$. Thus, suppose $a\geq 2$ and we approximate the previous expression with an integral
\begin{align*}
f_1(a,a+1) 
& \geq \frac{2}{a}\int_{1}^{a} \frac{1}{k^{\lambda-1}} dk
+\frac{2}{a^\lambda}-\frac{1}{(a+1)^{\lambda-1}} 
=
\frac{2}{(2-\lambda)a}(a^{2-\lambda}-1)
-\frac{1}{(a+1)^{\lambda-1}} 
+\frac{2}{a^\lambda} \notag \\
&\geq
\frac{2}{a^\lambda} \frac{a+2-\lambda}{2-\lambda}
-\frac{1}{a^{\lambda-1}} -\frac{2}{(2-\lambda)a} \notag \\
&=\frac{1}{2-\lambda} \left (\frac{\lambda a +2 (2-\lambda)}{a^\lambda}-\frac{2}{a} \right )>0,
\end{align*}
where the positivity follows for each $\lambda>1$.

\noindent
\textbf{Step 2:} In a second step we will prove that $a\mapsto f_2(a,l,\lfloor l(l+1)/a\rfloor)$ is decreasing in $a$ and in particular that $f_2(a,l, b) > 0$. First note that
\begin{equation}
f_2(a,l, \lfloor l(+1)/a\rfloor) \geq f^{Int}_2(a,l) :=  \int_{a}^{l-1} \frac{l(l+1)-ak}{k^{\lambda}} d k -\int_{l}^{l(l+1)/a} \frac{l(l+1)-ak}{k^{\lambda}} d k. 
\end{equation}

The map $a\mapsto \frac{l(l+1)-ak}{k^{\lambda}}$ is decreasing in $a$, moreover if $a$ increases then the area of integration shrinks faster in $a$, namely as $(l-a)$, than the area of integration of the second integral $(l+1)(l-a)/a$. Therefore the first term is dominant and the function becomes decreasing in $a$. Naturally $a\leq l-3$ hence it remains to show that $f^{Int}_2(l-3,l)\geq 0$. Note that $f^{Int}_2(l-3,l)$ is decreasing in $l$ by the same arguments as before. Taking $l$ to infinity we obtain $f_2^{Int}(l-3,l)$ converges to 0 from above.

\section{Proof of Proposition \ref{prop:rectangle_strip_attached2}}
\label{prop:3.5}
We distinguish again two cases: in Case (i) we suppose $k_1=0$, instead in Case (ii) we assume $k_1 \neq 0$.
\paragraph{Cases (i).} Let $k_1=0$.
We will prove that for $n=l(l+1)$ the nonlocal perimeter of a quasi-square $\mathscr{R}_{l,l+1}$ is smaller than the perimeter of a rectangle with a $k_2$-protuberance attached and area $ab+k_2=n$. 
Suppose first that the protuberance is attached along the shortest side, thus $k_2 \leq a-1$.

We can use \eqref{deltall+1} and \eqref{def_f1ll+1} to obtain again

\begin{equation}
\begin{split}
 Per_{\lambda}(\mathscr{R}^{k_2}_{a,b})- Per_{\lambda}(\mathscr{R}_{l,l+1}) &= 2f_1(a,l)+2 \overline{f}_2(a,l,k_2,b)+\frac{2{k_2}}{a}\sum_{r=1}^{a-1}\frac{a-r}{r^{\lambda}}- 2\sum_{r=1}^{k_2-1}\frac{k_2-r}{r^\lambda}-\frac{2k_2}{b^\lambda},
\end{split}
\end{equation}
where
\begin{align}\label{def:f2barra2}
    \overline{f}_2(a,l,k_2,b)= \sum_{r=a+1}^{l} \frac{l(l+1)-ar}{r^{\lambda}}-\sum_{r=l}^{b-1} \frac{l(l+1)-ar}{r^{\lambda}} \underset{\eqref{deff2ll+1}}{\geq} f_2(a,l,b).
\end{align}
From the proof of Proposition \ref{prop:rectangle_square}
 we obtain
\begin{align*}
 Per_{\lambda}(\mathscr{R}^{k_2}_{a,b})- Per_{\lambda}(\mathscr{R}_{l,l+1}) & \geq  \underbrace{2f_1(a,l)+ 2f_2(a,l,b)}_{> 0}
 +\underbrace{\frac{2{k_2}}{a}\sum_{r=1}^{a-1}\frac{a-r}{r^{\lambda}}- 2\sum_{r=1}^{k_2-1}\frac{k_2-r}{r^\lambda}-\frac{2k_2}{b^\lambda}}_{(\textnormal{Eq.A.1})}.
\end{align*}
Thus, we conclude that $(\textnormal{Eq.A.1}) > 0$ by arguing in the same manner of Case (i) of Proposition \ref{prop:rectangle_strip_attached1}.
Next, assume that the protuberance is attached along the longer side, thus $k_2 \leq b-1$. We observe that if $k_2 \leq a$, then the rectangle with side lengths $a,b$ and the protuberance $k_2$ attached along the shorter side has smaller nonlocal perimeter than $\mathscr{R}^{k_2}_{a,b}$. Indeed, denoting by $\overline{\mathscr{R}}^{k_2}_{a,b}$ the last one, we again have that the difference is equal to 
\begin{align*}
    Per_{\lambda}(\mathscr{R}^{k_2}_{a,b})- Per_{\lambda}(\overline{\mathscr{R}}^{k_2}_{a,b})= k_2 \sum_{r=a+1}^{b} \frac{1}{r^\lambda} >0
\end{align*}
and we conclude the claim in an analogous manner as the in the proof of Proposition \ref{prop:rectangle_strip_attached1}. 

Thus, suppose that $a \leq k_2 \leq b-1$,

we can use \eqref{deltall+1}, \eqref{def_f1ll+1}, \eqref{def:f2barra2} to obtain again
\begin{equation*}
\begin{split}
 Per_{\lambda}(\mathscr{R}^{k_2}_{a,b})- Per_{\lambda}(\mathscr{R}_{l,l+1}) &= 2f_1(a,l)+2 \overline{f}_2(a,l,k_2,b)+\frac{2{k_2}}{a}\sum_{r=1}^{a-1}\frac{a-r}{r^{\lambda}}- 2\sum_{r=1}^{k_2-1}\frac{k_2-r}{r^\lambda}+2k_2\sum_{r=a+1}^{b-1}\frac{1}{r^\lambda}.
\end{split}
\end{equation*}
From \eqref{def:f2barra2} and the proof of Proposition \ref{prop:rectangle_strip_attached1}
 we obtain analogously
\begin{align}
 Per_{\lambda}(\mathscr{R}^{k_2}_{a,b})- Per_{\lambda}(\mathscr{R}_{l,l+1}) & \geq  \underbrace{2f_1(a,l)+ 2f_2(a,l,b)}_{>0}
 +\underbrace{\frac{2{k_2}}{a}\sum_{r=1}^{a-1}\frac{a-r}{r^{\lambda}}- 2\sum_{r=1}^{k_2-1}\frac{k_2-r}{r^\lambda}+2k_2\sum_{r=a+1}^{b-1}\frac{1}{r^\lambda}}_{(\textnormal{Eq.A.2})}
\end{align}
and we can conclude by arguing as before, that
$\text{argmin}_{\mathcal{P} \in \mathcal{M}_{n}} \{Per_{\lambda}(\mathcal{P}) \}= \{ \mathscr{R}_{l,l+1}\}$ for $n=l(l+1)$.

\paragraph{Case (ii).} Let $k_1 \neq 0$ and consider a quasi-square with a protuberance attached along the shorter side and area $n=l(l+1)+k_1$, $\mathscr{R}_{l,l+1}^{k_1}$. First, we note that if $\overline{\mathscr{R}}_{l,l+1}^{k_1}$ is a quasi-square with a protuberance attached along the longer side and area $n=l(l+1)+k_1$, then 
    \begin{align*}
        Per_\lambda(\overline{\mathscr{R}}_{l,l+1}^{k_2})-Per_\lambda(\overline{\mathscr{R}}_{l,l+1}^{k_2})
        =\frac{k_2}{(l+1)^\lambda}>0.
    \end{align*}

Moreover, we observe that the classical perimeter of $\mathscr{R}_{l,l+1}^{k_1}$ is $per(\mathscr{R}_{l,l+1}^{k_1})=2(2l+2)$ and we will prove the following two results.
\begin{itemize}
\item[(ii.1)] Let $0 \leq k_2 \leq b-1$ and consider $\mathscr{R}_{a,b}^{k_2} \in \mathcal{M}_n$ with classical perimeter greater than the classical perimeter of $\mathscr{R}_{l,l+1}^{k_1}$, i.e. $2(a+b+1)=2(2l+2+C)$ for some $C \geq 1$, $C \in \mathbb{N}$. Then, 
\[
Per_{\lambda}(\mathscr{R}_{a,b}^{k_2}) > Per_{\lambda}(\mathscr{R}_{l,l+1}^{k_1}).
\]
\item[(ii.2)] Let $0 \leq k_2 \leq b-1$ and consider a rectangle $\overline{\mathscr{R}}_{a,b}^{k_2} \in \mathcal{M}_n$ with a protuberance attached along the longer side and classical perimeter greater than the classical perimeter of $\mathscr{R}_{l,l+1}^{k_1}$, i.e. $2(a+b+1)=2(2l+2+C)$ for some $C \geq 1$, $C \in \mathbb{N}$. Then,
\[
Per_{\lambda}(\overline{\mathscr{R}}_{a,b}^{k_2}) > Per_{\lambda}(\mathscr{R}_{l,l+1}^{k_1}).
\]
\end{itemize}

\subparagraph{Case (ii.1).} We suppose without of generality that $a<b$.
By using \eqref{gapquasisquarell+1} and \eqref{rectanglell+1} we obtain
\begin{align}\label{eq:difference_quasisquare_rectangle_proof_k}
    Per_{\lambda}(\mathscr{R}_{a,b}^{k_2})-Per_{\lambda}(\mathscr{R}_{l,l+1}^{k_1})
&=2 \tilde{h}_1(a,l, k_1)+ 2 \tilde{g}_2(a,l,k_1,b)
        +\frac{k_2}{a} \sum_{r=1}^{a-1} \frac{a-r}{r^\lambda} 
    - \frac{k_2}{b^\lambda}-\sum_{r=1}^{k_2-1} \frac{k_2-r}{r^\lambda} ,
\end{align}
where
\begin{align}
    & \tilde{h}_1(a,l, k_1)= 
    f_1(a,l)-\frac{k_1}{a} \sum_{r=1}^{a-1} \frac{a-r}{r^\lambda} 
    + \frac{k_1}{l^\lambda}+\sum_{r=1}^{k_1-1} \frac{k_1-r}{r^\lambda}, \label{def:h_tilde1}\\
    &\tilde{g}_2(a,l,k_1,b)= \sum_{r=a}^{l-1} \frac{l(l+1)-ar}{r^\lambda} - \sum_{r=l}^{b-1} \frac{l(l+1)+k_1-ar}{r^\lambda},\nonumber
  \end{align}
and $f_1(a,l)$ is defined in \eqref{def_f1ll+1}.  Call
\begin{align}
     \tilde{h}_2(a,l,k_1,b)= \sum_{r=a+1}^{l-1} \frac{l(l+1)-ar}{r^{\lambda}}-\sum_{r=l}^{b-1} \frac{l(l+1)+k_1-ar}{r^{\lambda}}, \label{def:h_tilde2}
\end{align}
and note that $\tilde{g}_2(a,l,k_1,b) \geq    \tilde{h}_2(a,l,k_1,b)$.

Then,
\begin{align*}
    Per_{\lambda}(\mathscr{R}_{a,b}^{k_2})-Per_{\lambda}(\mathscr{R}_{l,l+1}^{k_1}) 
    & \geq 2 \tilde{f}_1(a,l, k_1)+ 2 \tilde{h}_2(a,l,k_1, b).
\end{align*}

We conclude once we showed that both functions are strictly positive by the arguments below.
\begin{itemize}
\item[\text{(ii.1.a)}] First of all, we show that $\tilde{h}_2(a,l,k_1,b)>0$. We will prove that $k_1\mapsto \tilde{h}_2\left(a,l,k_1, \lfloor\frac{l(l+1)+k_1}{a} \rfloor \right)$ is decreasing in $k_1$. 
\begin{align*}
    \tilde{h}_2\left(a,l,k_1+1, \lfloor\frac{l(l+1)+k_1}{a} \rfloor \right)&-\tilde{h}_2 \left(a,l,k_1,\lfloor\frac{l(l+1)+k_1}{a} \rfloor \right) \\
    &=-\sum_{r= \lfloor\frac{l(l+1)+k_1}{a} \rfloor}^{\lfloor \frac{l(l+1)+k_1+1}{a} \rfloor-1} \frac{l(l+1)+k_1-ar}{r^{\lambda}}
    -\sum_{r=l}^{\lfloor \frac{l(l+1)+k_1+1}{a} \rfloor-1} \frac{k_1}{r^{\lambda}} <0.
\end{align*}
Then, we can consider the function $\tilde{h}_2(a,l,l-1, \lfloor\frac{l(l+1)+k_1}{a} \rfloor )$, i.e., 
\begin{align}
  \tilde{h}_2\left(a,l,l-1,\lfloor\frac{l(l+1)+k_1}{a} \rfloor \right) &= \sum_{r=a+1}^{l-1} \frac{l(l+1)-ar}{r^{\lambda}}-\sum_{r=l}^{\lfloor\frac{l(l+1)+l-1}{a}\rfloor-1} \frac{l(l+1)+l-1-ar}{r^{\lambda}} \\
   & \geq \sum_{r=a+1}^{l-1} \frac{l(l+1)-ar}{r^{\lambda}}-\sum_{r=l- \lfloor\frac{l-1}{a \rfloor}}^{\lfloor\frac{l(l+1)}{a}\rfloor-1} \frac{l(l+1)-ar}{r^{\lambda}} \label{eq:wa}
\end{align}
and we can conclude by arguing as in the second step of the proof of Proposition \ref{prop:rectangle_quasisquare}.

\item[$(ii.1.b)$] It remains to show that $\tilde{h}_1(a,l, k_1) > 0$.
We can use the results in the proof of Proposition \ref{prop:rectangle_quasisquare} in order to prove the positivity of $\tilde{h}_1(a,l, k_1)$.
First, we show that $l \mapsto \tilde{h}_1(a,l, k_1)$ is increasing in $l$. We write
\begin{equation*}
\begin{split}
\tilde{h}_1(a,l+1, k_1) &=  f_1(a,l+1) + \frac{k_1}{(l+1)^\lambda}-\frac{k_1}{a} \sum_{r=1}^{a-1} \frac{a-r}{r^\lambda} 
   +\sum_{r=1}^{k_1-1} \frac{k_1-r}{r^\lambda}, \\
& = f_1(a,l) + rest(a,l)  + \frac{k_1}{(l+1)^\lambda}-\frac{k_1}{a} \sum_{r=1}^{a-1} \frac{a-r}{r^\lambda} 
   +\sum_{r=1}^{k_1-1} \frac{k_1-r}{r^\lambda}, \\
&=\tilde{h}_1(a,l, k_1)+\underbrace{rest(a,l)+ \frac{k_1}{(l+1)^\lambda} - \frac{k_1}{l^\lambda}}_{=\overline{rest}(a,l)},
\end{split}
\end{equation*}
where $rest(a,l)$ is defined in \eqref{def_rest_ll+1}. 
The function is increasing in $l$, if the term $\overline{rest}(a,l)>0$ for all $l \geq a+1$. 
\begin{align}
    \overline{rest}(a,l+1)-\overline{rest}(a,l) &=rest(a,l+1)-rest(a,l)- \frac{k_1}{(l+1)^\lambda}+ \frac{k_1}{(l+2)^\lambda} + \frac{k_1}{l^\lambda}- \frac{k_1}{(l+1)^\lambda} \notag \\
    & \geq k_1 \left ( \frac{1}{l^\lambda} + \frac{1}{(l+2)^\lambda}-\frac{2}{(l+1)^\lambda} \right )>0,
\end{align}
where we obtain the first inequality as in the proof of Proposition \ref{prop:rectangle_quasisquare}. The positivity is due to decrease of the function $f(x)=\frac{1}{x^\lambda}-\frac{1}{(x+1)^\lambda}$ for $\lambda>1$.

\end{itemize}

\subparagraph{Case (ii.2).} 
We distinguish again three cases:
\begin{itemize}
\item[$(ii.2.a)$] $k_2 \leq k_1$. By using \eqref{eq:quasi_square_rectangle_GAP} we obtain the following
\begin{align*}
    Per_{\lambda}(\mathscr{R}_{a,b}^{k_2})-Per_{\lambda}(\mathscr{R}_{l,l+1}^{k_1})
& \geq  (Per_{\lambda}(\mathscr{R}_{a,b})- Per_{\lambda}(\mathscr{R}_{l,l+1}))
-2\sum_{i=1}^{k_1-k_2} \zeta(\lambda, i)
-2 (k_1-k_2) \zeta(\lambda, l+1) \notag \\
&= 2\tilde{h}_1(a,l, k_1-k_2)+2\tilde{h}_2(a,l,k_1-k_2,b),
\end{align*}
where the functions $\tilde h_1, \tilde h_2$ are defined in \eqref{def:h_tilde1} and \eqref{def:h_tilde2}. We note that if $k_2=k_1$, then $\tilde{h}_1(a,l, k_1-k_2)=f_1(a,l)$ and $\tilde{h}_2(a,l, k_1-k_2,b)=f_2(a,l,b)$ where $f_1,f_2$ are defined in \eqref{def_f1ll+1}, \eqref{deff2ll+1}. 
Then, we know that $\tilde{h}_2(a,l, k, b)$ is decreasing in $k$ as we proved in Case (ii.1), and $k_1-k_2 \leq l-1$, so $\tilde{h}_2\left(a,l, l-1, \lfloor \frac{l(l+1)+l}{a} \rfloor \right) > 0$ by \eqref{eq:wa}. Moreover, $\tilde{h}_1(a,l, k_1-k_2) \geq 0$ analogously to Case (ii.1). 

\item[$(ii.2.b)$] Let $k_2 > k_1$.  
By using \eqref{eq:quasi_square_rectangle_GAP} we obtain the following 
\begin{align*}
  &  Per_{\lambda}(\mathscr{R}_{a,b}^{k_2})-Per_{\lambda}(\mathscr{R}_{l,l+1}^{k_1})\\
 &\geq 2\tilde{h}_1(a,l, k_2-k_1)+ \underbrace{2\tilde{h}_2(a,l,k_2-k_1,b)
+2\sum_{i=1}^{k_2-k_1}\zeta(\lambda, i+k_1) +2\sum_{i=1}^{k_2-k_1}\zeta(\lambda,i)
+4 (k_2-k_1) \zeta(\lambda, l+1)}_{(\textnormal{Eq.A.3})},
\end{align*}
where the functions $\tilde h_1, \tilde h_2$ are defined in \eqref{def:h_tilde1} and in \eqref{def:h_tilde2}.
We observe that $\tilde{h}_1(a,l, k_2-k_1) > 0$ as we proved before. The proof is concluded, once we show that $(\textnormal{Eq.A.3}) > 0$. Indeed,

we get that 
\begin{align*}
   (\textnormal{Eq.A.3}) > 
      \sum_{r=a+1}^{l-1} \frac{l^2-ar}{r^{\lambda}}-\sum_{r=l}^{\frac{l^2+k_2-k_1}{a}-1} \frac{l^2-ar}{r^{\lambda}}
\end{align*}
and conclude by arguing as in the second step of the proof of Proposition \ref{prop:rectangle_quasisquare}.
\end{itemize}

We have proven that for $n=l(l+1)+k_1$, we have 
    \begin{align*}
        \text{argmin}_{\mathcal{P} \in \mathcal{M}_{n}} \{Per_{\lambda}(\mathcal{P}) \} \subset 
        \{ \mathscr{R}_{l,l+1}^{k_1} \in \mathcal{M}_n \} 
        \cup \left\{ \mathscr{R}_{a,b}^{k_2} \in \mathcal{M}_n \, | \, per(\mathscr{R}_{a,b}^{k_2}) = per(\mathscr{R}_{l,l+1}^{k_1} ) \right\}.
    \end{align*}
\end{appendix}

%%%%%%%%%%%%%%%%%%%%%%%%%%%%%%%%%%%%%%%%%%%%%%
%% Support information, if any,             %%
%% should be provided in the                %%
%% Acknowledgements section.                %%
%%%%%%%%%%%%%%%%%%%%%%%%%%%%%%%%%%%%%%%%%%%%%%
\begin{acks}[Acknowledgments]
V.J. is grateful for the partial support of “Gruppo Nazionale per l’Analisi Matematica, la Probabilit\`a e le loro Applicazioni” (GNAMPA-INdAM).
\end{acks}

%%%%%%%%%%%%%%%%%%%%%%%%%%%%%%%%%%%%%%%%%%%%%%
%% Funding information, if any,             %%
%% should be provided in the                %%
%% funding section.                         %%
%%%%%%%%%%%%%%%%%%%%%%%%%%%%%%%%%%%%%%%%%%%%%%
\begin{funding}
W.M.R. and V.J. are funded by the Vidi grant VI.Vidi.213.112 from the Dutch Research Council. 
\end{funding}

%%  imsart-number.bst  will be used to                     %%
%%  MR numbers will be added by VTeX.                      %%
%%                                                         %%
%%  Use \cite{...} to cite references in text.             %%
%%                                                         %%
%%%%%%%%%%%%%%%%%%%%%%%%%%%%%%%%%%%%%%%%%%%%%%%%%%%%%%%%%%%%%

\bibliographystyle{abbrv}
\bibliography{biblio}

@article{alonso1996three,
  title={The three dimensional polyominoes of minimal area},
  author={Alonso, Laurent and Cerf, Rapha{\"e}l},
  journal={The Electronic Journal of Combinatorics},
  volume={3},
  number={1},
  pages={R27},
  year={1996}
}

@article{caffa2010,
	affiliation = {University of Texas at Austin, Department of Mathematics, Austin, TX 78712; Columbia University, Mathematics Department, 2990 Broadway, New York, NY 10027; Université Paul Sabatier, Institut de Mathématiques, (UMR CNRS 5219), 31062 Toulouse Cedex 4, FRANCE},
	author = {Caffarelli, L. and Roquejoffre, J.‐M. and Savin, O.},
	doi = {10.1002/cpa.20331},
	journal = {Communications on Pure and Applied Mathematics},
	language = {Undetermined},
	month = {6},
	number = {9},
	pages = {1111 - 1144},
	title = {Nonlocal minimal surfaces},
	volume = {63},
	year = {2010},
}

@article{figa2015,
  title={Isoperimetry and stability properties of balls with respect to nonlocal energies},
  author={Figalli, Alessio and Fusco, Nicola and Maggi, Francesco and Millot, Vincent and Morini, Massimiliano},
  journal={Communications in Mathematical Physics},
  volume={336},
  pages={441--507},
  year={2015},
  publisher={Springer}
}

@article{vanenter2019,
  title={Nucleation for one-dimensional long-range Ising models},
  author={van Enter, Aernout CD and Kimura, Bruno and Ruszel, Wioletta and Spitoni, Cristian},
  journal={Journal of Statistical Physics},
  volume={174},
  pages={1327--1345},
  year={2019},
  publisher={Springer}
}

@article{coquille2018,
  title={Absence of dobrushin states for 2 d long-range ising models},
  author={Coquille, Loren and van Enter, Aernout CD and Le Ny, Arnaud and Ruszel, Wioletta M},
  journal={Journal of Statistical Physics},
  volume={172},
  pages={1210--1222},
  year={2018},
  publisher={Springer}
}

@article{savin2012,
  title={{$\Gamma$}-convergence for nonlocal phase transitions},
  author={Savin, Ovidiu and Valdinoci, Enrico},
  journal={Annales de l'Institut Henri Poincar{\'e} C, Analyse non lin{\'e}aire},
  volume={29},
  number={4},
  pages={479--500},
  year={2012},
  organization={Elsevier}
}

@article{cozzi2017,
  title={Planelike interfaces in long-range Ising models and connections with nonlocal minimal surfaces},
  author={Cozzi, Matteo and Dipierro, Serena and Valdinoci, Enrico},
  journal={Journal of Statistical Physics},
  volume={167},
  pages={1401--1451},
  year={2017},
  publisher={Springer}
}

@article{manzo2004,
  title={On the essential features of metastability: tunnelling time and critical configurations},
  author={Manzo, Francesco and Nardi, Francesca R and Olivieri, Enzo and Scoppola, Elisabetta},
  journal={Journal of Statistical Physics},
  volume={115},
  number={1},
  pages={591--642},
  year={2004},
  publisher={Springer}
}

@book{olivieri2005,
  title={Large deviations and metastability},
  author={Olivieri, Enzo and Vares, Maria Eul{\'a}lia},
  number={100},
  year={2005},
  publisher={Cambridge University Press}
}

@article{chak2011,
  title={Modulation and correlation lengths in systems with competing interactions},
  author={Chakrabarty, Saurish and Nussinov, Zohar},
  journal={Physical Review B—Condensed Matter and Materials Physics},
  volume={84},
  number={14},
  pages={144402},
  year={2011},
  publisher={APS}
}

@article{de2000,
  title={Dipolar effects in magnetic thin films and quasi-two-dimensional systems},
  author={De’Bell, K and MacIsaac, AB and Whitehead, JP},
  journal={Reviews of Modern Physics},
  volume={72},
  number={1},
  pages={225},
  year={2000},
  publisher={APS}
}

@article{giuliani2012,
  title={Striped periodic minimizers of a two-dimensional model for martensitic phase transitions},
  author={Giuliani, Alessandro and M{\"u}ller, Stefan},
  journal={Communications in Mathematical Physics},
  volume={309},
  number={2},
  pages={313--339},
  year={2012},
  publisher={Springer}
}

@article{giuliani2016,
  title={Periodic striped ground states in Ising models with competing interactions},
  author={Giuliani, Alessandro and Seiringer, Robert},
  journal={Communications in Mathematical Physics},
  volume={347},
  pages={983--1007},
  year={2016},
  publisher={Springer}
}

@article{daneri2019,
  title={Exact periodic stripes for minimizers of a local/nonlocal interaction functional in general dimension},
  author={Daneri, Sara and Runa, Eris},
  journal={Archive for Rational Mechanics and Analysis},
  volume={231},
  pages={519--589},
  year={2019},
  publisher={Springer}
}

@article{degiorgi,
  title={Sulla propriet{\`a} isoperimetrica dell’ipersfera, nella classe degli insiemi aventi frontiera orientata di misura finita},
  author={Giorgi, E de},
  journal={Atti Accad. Naz. Lincei. Mem. Cl. Sci. Fis. Mat. Nat. Sez. I (8)},
  volume={5},
  pages={33--44},
  year={1958}
}

@article{elliott1961,
  title={Phenomenological discussion of magnetic ordering in the heavy rare-earth metals},
  author={Elliott, R Jo},
  journal={Physical Review},
  volume={124},
  number={2},
  pages={346},
  year={1961},
  publisher={APS}
}

@article{randa1985,
  title={Axial next-nearest-neighbor Ising (ANNNI) and extended-ANNNI models in external fields},
  author={Randa, J},
  journal={Physical Review B},
  volume={32},
  number={1},
  pages={413},
  year={1985},
  publisher={APS}
}

@article{Affonso,
  title={Long-range Ising models: Contours, phase transitions and decaying fields},
  author={Affonso, L and Bissacot, R. and Endo, E.O. and  Handa, S.},
  journal={Journal of European Mathematical Society},
  year={2024},
}

\end{document}